\documentclass{article}
\usepackage[utf8]{inputenc}
\usepackage{amsmath}
\usepackage{amsthm}
\usepackage{amsfonts}
\usepackage{amssymb, amsbsy}
\usepackage{color}
\usepackage{hyperref}
\usepackage{enumitem}
\usepackage{comment}
\usepackage{bbold}
\usepackage{mathtools}
\usepackage{esint}
\usepackage{mathabx}
\usepackage{soul}
\usepackage{anysize}

\usepackage{cancel}

\date{\today}

\newcommand{\Bil}{\operatorname{Bil}}
\newcommand{\ueps}{u_{\e}}
\newcommand{\veps}{v_{\e}}
\newcommand{\weps}{w_{\e}}

\newcommand{\peps}{p_{\e}}
\newcommand{\jeps}{J_{\e}}
\newcommand{\Jeps}{J_{\e}}

\newcommand{\qeps}{q_{\e}}
\newcommand{\heeps}{\widehat{e}_{\e}}
\newcommand{\hheps}{\widehat{h}_{\e}}

\newcommand{\Aeps}{A_{\e}}
\newcommand{\Feps}{F_{\e}}
\newcommand{\1}{\mathbb{1}}

\newcommand{\oet}{\Omega_{\e}(t)}
\newcommand{\get}{\Gamma_{\e}(t)}
\newcommand{\geps}{\Gamma_{\e}}
\renewcommand{\oe}{\Omega_{\e}}
\newcommand{\reps}{R_{\e}}
\newcommand{\Reps}{R_{\e}}
\newcommand{\Deps}{D_{\e}}
\newcommand{\tueps}{\widetilde{u}_{\e}}

\newcommand{\uepsin}{u_{\e}^{\rm{in}}}
\newcommand{\repsin}{R_{\e}^{\rm{in}}}

\newcommand{\bfxe}{\left[\fxe\right]_Y}
\newcommand{\fxe}{\frac{x}{\e}}
\newcommand{\feps}{F_{\e}}
\newcommand{\psieps}{\psi_{\e}}
\newcommand{\rmin}{R_{\mathrm{min}}}
\newcommand{\rmax}{R_{\mathrm{max}}}
\newcommand{\hueps}{\widehat{u}_{\e}}
\newcommand{\hreps}{\widehat{R}_{\e}}
\newcommand{\hReps}{\widehat{R}_{\e}}
\newcommand{\teps}{\mathcal{T}_{\e}}

\newcommand{\hweps}{\widehat{w}_{\e}}

\newcommand{\hjeps}{\widehat{J}_{\e}}
\newcommand{\hJeps}{\widehat{J}_{\e}}
\newcommand{\hpsieps}{\widehat{\psi}_{\e}}
\newcommand{\phieps}{\phi_{\e}}

\newcommand{\hDeps}{\widehat{D}_{\e}}
\newcommand{\hAeps}{\widehat{A}_{\e}}

\newcommand{\hFeps}{\widehat{F}_{\e}}

\newcommand{\m}{\mathfrak{m}}

\newcommand{\R}{\mathbb{R}}
\newcommand{\N}{\mathbb{N}}
\newcommand{\Z}{\mathbb{Z}}

\newcommand{\e}{\varepsilon}
\newcommand{\norm}[2]{\left\|#1\right\|_{#2}}

\newcommand{\rightwts}[1]{\xrightharpoonup[]{2,#1}}
\newcommand{\rightsts}[1]{\xrightarrow[]{2,#1}}

\newcommand{\Oe}{{\Omega_\e}}

\newcommand{\Ge}{{\Gamma_\e}}
\newcommand{\tOe}{{(0,t) \times \Oe}}

\renewcommand{\div}{\operatorname{\nabla \cdot}}

\newcommand{\per}{\mathrm{per}}

\newcommand{\tw}{\tilde{w}}
\newcommand{\ie}{i.\,e.\,}

\newtheorem{theorem}{Theorem}[section]
\newtheorem{corollary}{Corollary}[theorem]
\newtheorem{lemma}[theorem]{Lemma}
\newtheorem{proposition}[theorem]{Proposition}
\newtheorem{remark}[theorem]{Remark}
\newtheorem{definition}[theorem]{Definition}

\begin{document}

\title{Rigorous derivation of an effective model for coupled Stokes
advection, reaction and diffusion with freely evolving microstructure}
\author{M.~Gahn\footnote{Institute for Mathematics, University Heidelberg, 
Im Neuenheimer Feld 205, 69120 Heidelberg, Germany.} \and M.~A.~Peter\footnote{Institute of Mathematics, University of Augsburg and Centre for Advanced Analytics and Predictive Sciences, University of Augsburg,
86135 Augsburg, Germany} \and I.~S.~Pop\footnote{Hasselt University, Faculty of Sciences, Agoralaan Gebouw D, Diepenbeek 3590, Belgium} \and D.~Wiedemann\footnote{Institute of Mathematics, University of Augsburg, 86135 Augsburg, 
Germany. Current address: Department of Mathematics, 
Technical University of Dortmund,
44227 Dortmund, Germany}}

\maketitle

\medskip

\begin{abstract}  
We consider the homogenisation of a coupled Stokes flow and advection--reaction--diffusion problem in a perforated domain with an evolving microstructure of size $\varepsilon$. 
Reactions at the boundaries of the microscopic interfaces lead to 
the formation of a solid layer having a variable, a priori unknown thickness. This results in a
growth or shrinkage of the solid phase 
and, thus, the domain evolution is not known a priori but induced by the advection--reaction--diffusion process. 
The achievements of this work are the existence and uniqueness of a weak microscopic solution and the rigorous derivation of an effective model for $\varepsilon \to 0$, based on $\varepsilon$-uniform a priori estimates.
As a result of the limit passage, the processes on the macroscale are described by an advection--reaction--diffusion problem coupled to Darcy's equation with effective coefficients (porosity, diffusivity and permeability) depending on local cell problems.
These local problems are formulated on cells, which depend on the macroscopic position and evolve in time. In particular, the evolution of these cells depends on the macroscopic concentration.
Thus, the cell problems (respectively the effective coefficients) are coupled to the macroscopic unknowns and vice versa, leading to a strongly coupled micro--macro model.
For pure reactive--diffusive transport coupled with microscopic domain evolution but without advective transport, homogenisation results have recently been presented. 
We extend these models by advective transport which is driven by the Stokes equation in the a priori unknown evolving pore domain.

\end{abstract}


\paragraph{MSC2020 subject classification:} 35B27, 35K57, 35R35, 76M50, 76S05, 80A19  

\section{Introduction}
Reactive transport in evolving porous media occurs in various real-life applications, 
for instance in mineral precipitation and dissolution \cite{BBPR16, Noo08,SRFMK17}, biofilm growth \cite{SK17}, colloid deposition \cite{ENM22} and water diffusion into absorbent particles \cite{FM02,SDH17}. 
The transport of the soluble species is driven by advection and diffusion.
An illustrative case of this phenomenon is precipitation and dissolution taking place in a porous medium made up of solid grains and pore space, in which the void space is filled with a fluid, such as water, and soluble species can be diffusively and advectively transported. Reactions of the species can yield formation of a solid layer at the fluid--solid interface via precipitation, such as salt. Conversely, the reverse process of dissolution is also possible and solute species can be released back into the fluid. 
We emphasise that the thickness of the precipitate layer is not know a priori but depends on the other unknowns in the model. Therefore, in mathematical terms, the fluid--solid interface is a free boundary.

\paragraph{Literature overview} 
The upscaling of reaction--diffusion problems in evolving domains is important for modelling dissolution and precipitation processes \cite{Noo08, BBPR16, BK17, GFKN20, SRFMK17} as well as biofilm formation processes \cite{SK17}. For example, for dissolution and precipitation in a porous medium, a precipitate layer may be added to or be dissolved from the pore walls, implying that the overall solid part (and, implicitly, the void space) is evolving unless there is a local balance between precipitation and dissolution, see e.g.~\cite{Noo08}.
In \cite{Noo08, RNF12, REK15, BBPR16, RRP16, SK17, SK17a, BVP20}, such processes are modelled as free boundaries by means of a level-set function or phase-field approaches.
However, these models are only formally upscaled by asymptotic two-scale expansions.
For fixed microstructure, related (advection–) diffusion problems were homogenised rigorously in  \cite{MP05, MeirmanovZimin}.

In the case of a fixed domain, rigorous homogenisation results for reactive transport  with nonlinear interface and bulk terms are presented in \cite{gahn2017derivation, gahn_2022, KNP16} and in \cite{Graf2014,DN15,Bunoiu2019} for different scalings of the interface terms.
The homogenisation of non-linear models involving slowly diffusing quantities (diffusion of order $\e^2$) are for instance considered in \cite{BourgeatLuckhausMikelic, HJM94, PR10} and with (fast) diffusion on internal boundaries in \cite{Graf2014CRM,Graf2014_SIMA,Graf2014_NA}.
In \cite{KNP16}, a precipitation--dissolution model involving a multi-valued dissolution rate is homogenised.

Rigorous homogenisation for evolving microstructure can be performed by transforming the equations on a fixed periodic substitute domain and passing to the limit there. This approach was presented in \cite{Pet07} and applied for (advection--)reaction--diffusion processes \cite{GahnNRP, Pet07b,Pet09,Pet09b} and for an elasticity problem in \cite{EM17}. In \cite{wiedemann2021homogenisation}, this transformation approach was used for the homogenisation of quasi-stationary Stokes flow and in \cite{DarcyMemory2024} for an instationary Stokes flow. In the latter two works, a given generic microscopic domain evolution was considered and the fluid velocity at the moving interface was given. The domain evolution was formulated in a generic setting and, for the homogenisation, sufficient conditions on the domain evolution were presented.
In \cite{AA23}, the transformation approach was justified by showing that the homogenisation process commutes with the transformation, i.e.~transforming the problem on a periodic reference domain, performing the homogenisation there and transforming the problem back actually gives the homogenisation result for the original equation. Moreover, two-scale transformation rules for the homogenisation correctors are presented in \cite{AA23} enabling a two-scale back-transformation, which provides a completely transformation-independent homogenisation result. The transformation rules are extended for the Stokes problem in \cite{wiedemann2021homogenisation}.

In \cite{GahnPop2023Mineral, WIEDEMANN2023113168}, the homogenisation of reactive transport for a priori not given evolving microstructure, which depends on the unknowns itself, was considered. The geometrical setting is the same as in the present work. 
More precisely, we assume that, around each grain, the microstructure evolution is radially symmetric and the free boundary can be described locally by the (uniform) thickness of the layer. 
However, we consider here additional advective transport, where the fluid velocity is modelled by the Stokes equations. This requires not only the inclusion of the Stokes equations in a non-linear manner, but we also have to deal with the new advective term in the advection--reaction--diffusion equation. Since the coefficents of this advective term arise as solution of the Stokes equations, it provides low regularity and complicates the homogenisation. Hence, let us summarize the new aspects and highlights of our paper, which we will specify in more detail straight afterwards.
\\

\noindent\textbf{Highlights}
\begin{itemize}
    \item Existence and uniqueness for a coupled Stokes and advection--reaction--diffusion system in an evolving microdomain
    \item Essential bounds for the concentration uniformly with respect to $\e$
    \item General (strong) two-scale compactness results based on $\e$-uniform a priori estimates associated to problems in evolving microdomains
    \item Coupled Darcy and advection--reaction--diffusion system as the homogenised limit, with effective coefficients depending on cell solutions on evolving micro-cells
\end{itemize}

\paragraph{Content of the paper} 
In what follows, let the time interval $(0,T)$ be given by $T>0$. The heterogeneous medium $\Omega \subset \R^n$ (for $n \in \{2,3,4\}$) consists of the $\e$-scaled pore space $\Oe$ which is obtained by removing the solid grains from $\Omega$. The centres of the grains are distributed periodically, where the distance between two neighbouring centres is of order $\e$ and, thus, the pore space $\Oe(t)$ depends on $\e$ and the time $t$ for $t \in  [0,T]$.
In the pore space $\Oe(t)$, we consider an advection--reaction--diffusion equation for the solute species $u_\e$ and the Stokes equations for the fluid velocity $v_\e$ and the fluid pressure $p_\e$. 
We assume that the periodically distributed solid grains are spheres with radii $\reps$ of the order of the microscopic length scale $\e$. The grains evolve independently of each other and we assume that the evolution is radially symmetric and, thus, well described by the radii. The radii are given as a solution of 
ordinary differential equations (ODEs), which depend on the 
solute 
concentration on the surface of the grains and, therefore, on the unknowns of the system.

To establish existence of a microscopic solution and for deriving the macroscopic model, we transform the problem onto a substitute domain which is periodic and constant with respect to time. Since the evolution of the domain is coupled with the solution of the system, the transformation mapping depends on the unknowns of the model and, thus, the resulting transformed equations become highly non-linear with coefficients depending on the transformation. 
To show the existence and uniqueness of a weak solution of the microscopic model (on the fixed domain $\Oe$) we use Sch\"afer's fixed point theorem, where we construct the fixed point operator as the composition of the following three operators. The first provides the domain evolution for given concentration (solves the ODE for $\reps$), the second provides a solution of the Stokes equation for given domain evolution, and the third operator gives the solution of the nonlinear advection--reaction--diffusion equation for given fluid flow and domain evolution. 
In order to obtain the continuity of these operators, a very subtle choice for the corresponding spaces is required.
Moreover, the uniqueness of the solution for the full microscopic problem is obtained by energy methods, which require additional Lipschitz continuity for the operators.

To pass to the limit $\varepsilon \to 0$ and derive the macroscopic model, we use the concept of two-scale convergence, respectively the unfolding method. In order to deal with the nonlinear terms, we need the strong two-scale convergence. In particular, we require such strong compactness results for the concentration $\ueps$ and the nonlinear coefficients arising from the domain transformation. These coefficients depend on the transformation mapping, and since the evolution of the microscopic domain can be described by the radius $\reps$, the strong two-scale convergence of the coefficients can be reduced to the strong (two-scale) convergence of $\reps$.
Let us point out a crucial difficulty in the derivation of the strong two-scale convergence of $\ueps$. Due to the nonlinear structure of the problem, we are not able to obtain $\varepsilon$-uniform a priori bounds for the time derivative $\partial_t \ueps$, which together with an $L^2$-energy bound for $\ueps$ and its gradient would guarantee the strong two-scale convergence of $\ueps$, see \cite{Gahn, MeirmanovZimin}. In our case, we only obtain an $\varepsilon$-uniform bound for $\partial_t (\jeps \ueps)$, where $\jeps$ denotes the Jacobi-determinant of the transformation mapping.
Now, for the strong convergence of $\ueps$, the idea is to use a Kolmogorov--Simon-type compactness result, for which we have to control shifts with respect to space and time, and for the latter we use the bound on $\partial_t (\jeps \ueps)$. Such a result was already obtained in \cite{GahnPop2023Mineral}.  In our paper, we present a similar strong two-scale compactness result with a simplified proof and under weaker a priori bounds. We point out that in \cite{WIEDEMANN2023113168} the strong convergence of $\ueps$ for the pure reaction--diffusion was obtained by showing an estimate for the time-shifts directly from the weak variational equation. 
We emphasise that, in contrast to the results in \cite{GahnPop2023Mineral,WIEDEMANN2023113168}, due to the additional advective term in our microscopic problem, we need an $\varepsilon$-uniform essential bound for the concentration $\ueps$ in order to control the time derivative $\partial_t (\jeps \ueps)$. 
In view of the application, this essential boundedness is natural.

For the strong convergence of the radius $\reps$, the standard Kolmogorov compactness theorem is applied. For this, we use the $\varepsilon$-uniform energy bounds for $\ueps$ and its gradient, and that  $\reps$ depends Lipschitz-continuously on  $\ueps$. 
Such a result (for the same a priori estimates) was also obtained in \cite{GahnPop2023Mineral}, but under slightly stronger assumptions on the initial value for $\reps$. In \cite{WIEDEMANN2023113168}, the strong convergence was obtained by comparing the microscopic ODE for $\reps$ with the (expected) macroscopic ODE. For this, the strong convergence of $\ueps$ was necessary. We emphasise that, in our approach, this strong convergence is necessary only to identify the limit problem for the macroscopic radius, and not to obtain the strong convergence of $\reps$. 
With the strong two-scale convergence of the coefficients and of the concentration $\ueps$, and the weak two-scale convergence of the fluid velocity $\veps$, we are able to pass to the limit in the microscopic reaction--advection--diffusion equation. Finally, to obtain the full macromodel, we have to derive Darcy's law for the Stokes system with the pulled-back coefficients. For the latter, again using the strong convergence of $\reps$, we obtain the strong two-scale convergence, and therefore we can pass to the limit $\varepsilon \to 0$. 
We mention that a similar homogenisation result was obtained in \cite{wiedemann2021homogenisation}, but for an a priori known evolving microstructure. However, in the present paper, we use a slightly different technique. In fact, for the derivation of the so-called two-pressure Stokes problem (with pulled back coefficients), we use a two-pressure decomposition as in \cite{allaire1997one}. Further, we first formulate the Darcy-law with its cell problems on the fixed reference elements.

The resulting effective equations consist of a homogenised advection--reaction--diffusion equation describing the species transport, an ODE (depending on the macroscopic variable) describing the local reference cell evolution and a Darcy law for evolving microstructure. 
The porosity as well as the effective diffusivity and permeability can be obtained by means of the solutions of cell problems, which take the local microstructure into account.
Since we perform the homogenisation on the the fixed microscopic domain $\Oe$ with the pulled-back coefficients depending on the chosen transformation, we obtain, in a first step, cell problems on a fixed reference cell with coefficients depending on a local transformation. By transforming these equations to the actual local space- and time-dependent cell domain, we obtain a completely transformation-independent homogenisation result.
In particular, the cell problems and the effective coefficients depend on the macroscopic quantities via the local cell geometry. This leads to a strongly coupled micro--macro system.

The paper is organised as follows:
In Section \ref{sec:MicroModel}, we present the microscopic model on the evolving  domain $\Oe(t)$ and transform it onto the fixed substitute domain $\Oe$.
In Section \ref{SectionMainResults}, we formulate the macroscopic model and give the main results of the paper. The proof of existence and uniqueness of a  microscopic solution are given in Section \ref{sec:Micromodel:Existence+Uniqueness}. Afterwards, we derive in Section \ref{SectionAprioriEstimates} the $\varepsilon$-uniform
a priori estimates, which form the basis for the compactness results necessary to pass to the limit $\varepsilon \to 0$. These compactness results are given in Section \ref{sec:Homogenisation}, where we also derive the macroscopic model with the effective coefficients depending on cell problems formulated on fixed reference elements and with coefficients depending on the transformation. Moreover, we transform the effective diffusivity, Darcy velocity and permeability on the space- and time-dependent cells, which yields a transformation-independent homogenisation result.
In Appendix \ref{sec:AppendixA}, we state technical auxiliary results. In Appendix \ref{SectionTwoScaleConvergence}, we give a brief introduction to the two-scale convergence and the unfolding operator, and state some basic properties.

\paragraph{Basic notations}
Let $\Omega \subset \R^n$ be open and $f : \Omega \to \R^{m\times l}$ for $n,m,l \in \N$. We  denote the Fr\'echet derivative  of $f$ by $\partial f$. Especially, for $l=1$ ($f:\Omega \rightarrow \R^m$) we identify $\partial f$ with its Jacobian matrix given by $(\partial f(x))_{ij} \coloneqq  \partial_{x_j} f_i(x) = \partial_i f_j(x)$ for $i=1,\ldots,n$, and $j=1,\ldots,m$. Further, we define  the gradient of $f:\Omega \rightarrow \R^m$ by $\nabla f \coloneqq \partial f^\top$. For a matrix-valued function $A: \Omega \rightarrow \R^{n\times n}$, we define its divergence $\nabla \cdot A  : \Omega \rightarrow \R^n$ for $i=1,\ldots , n$ by
\begin{align*}
    [\nabla \cdot A]_i := \sum_{j=1}^n \partial_j A_{ji}.
\end{align*}
Thus, the operator $\div$ applies to the columns of $A$.
In particular, for a vector field $v : \Omega \rightarrow \R^n$, we define the Laplace operator by $\Delta v := \nabla \cdot (\nabla v) = (\nabla \cdot (\nabla v_i))_{i=1}^n$. With $\phi: \Omega \rightarrow \R$, we have the following identities, which will be used frequently throughout the paper:
\begin{align}
\nabla \cdot (A v)  &= [\nabla \cdot A] \cdot v + A :\nabla v,
\\
\nabla \cdot (\phi A) &= \phi \nabla \cdot A + A^\top \nabla \phi,
\end{align}
with the Frobenius product $B:C:= \mathrm{tr}(B^{\top} C) =  \sum_{i,j=1}^n B_{ij}C_{ij}$ for $B,C \in \R^{n\times n}$.

Now, let $\psi : (0,T) \times \Omega \rightarrow \widetilde{\Omega} \subset \R^n$ such that  $\psi(t,\cdot)$ is a diffeomorphism from $\Omega$ to $\widetilde{\Omega}$. We define its Jacobian determinant $J:= \det(\nabla \psi)$ and write $A(x):= J(x)  \left(\partial \psi(x)\right)^{-1}$ for $x \in \Omega$. Then, the Piola identity (see \cite[p.117]{MarsdenHughes})
\begin{align}\label{PiolaIdentity}
\nabla \cdot A = 0
\end{align}
holds.
Together with the Jacobi formula for the derivative of the determinant, we get for $V:= A \partial_t \psi$ 
\begin{align}\label{Id:dtDet_V}
    \partial_t J = \nabla \cdot V.
\end{align}

Let $n,d\in \N$, then for $\Omega\subset \R^n$ we denote by $L^p(\Omega)^d, \, W^{1,p}(\Omega)^d $ the standard Lebesgue and Sobolev spaces with $p \in [1,\infty]$. In particular, for $p=2$, we write $H^1(\Omega)^d:= W^{1,2}(\Omega)^d$. For the norms, we omit the upper index $d$, for example we write $\|\cdot\|_{L^p(\Omega)}$ instead of $\|\cdot \|_{L^p(\Omega)^d}$. For a separable Banach space $X$ and $p \in [1,\infty]$, we denote the usual Bochner spaces by $L^p(\Omega,X)$ 
or $L^p((0, T), X)$ when the time is involved. 
For the dual space of $X$, we use the notation $X'$. Further, we consider the following periodic function spaces. Let $Y:= (0,1)^n$ be the unit cell in $\R^n$, then $C^{\infty}_{\per}(Y)$ is the space of smooth functions on $\R^n$ which are $Y$-periodic and $W^{1,p}_{\per}(Y)$ is the closure of $C^{\infty}_{\per}(Y)$ with respect to the norm on $W^{1,p}(Y)$. Again, we use the notation $H^1_{\per}(Y):= W^{1,2}_{\per}(Y)$. Finally, for a subset $Y^{\ast}\subset Y$ with $\partial Y \subset \partial Y^{\ast}$, we denote by $H^1_{\per}(Y^{\ast})$ the space of  functions from  $H_{\per}^1(Y)$ restricted to $Y^{\ast}$.

\section{The microscopic model}\label{sec:MicroModel}
In this section, we introduce the microscopic model, which is at first formulated on a time-dependent evolving microscopic domain $\Oe(t)$. More precisely, the transport and fluid equations are formulated in the current configuration (Eulerian coordinates) $\Oe(t)$. The heterogeneous domain $\Oe(t)$ is given as a rectangle periodically perforated by solid grains of spherical shape, where the radii of the solids differ and change in time, and are described by $\reps(t,x)$. The evolution of the radius depends on the concentration at the surface of the grains. In a first step, we transform the model to a fixed (time-independent) domain with periodically distributed solid grains of identical radii (Lagrangian coordinates). 

\subsection{The microscopic model in the evolving domain}
Let $T>0$ and, for $n \in \{2,3,4\}$, we consider a rectangle  $\Omega = (a,b)\subset \R^n$ with  $a,b \in \Z^n$ and $a_i< b_i$ for all $i \in \{1, \dots, n \}$. Further, let $\e>0$ be a sequence tending to zero such that $\e^{-1} \in \N$. We denote the unit cube by $Y:=(0,1)^n$ and we define 
\begin{align*}
    K_{\e}:= \left\{k \in \Z^n\, : \, \e (k + Y) \subset \Omega \right\}.
\end{align*}
We describe by $\reps$ the radii of the solid grains in the evolving microdomain, where the evolution of $\reps$ is given by the ordinary differential equation $\eqref{MicroscopicModel_ODE}$ below. More precisely, for  $\reps: (0,T)\times \Omega \rightarrow \left[\e R_\mathrm{min} , \e R_\mathrm{max}\right]$, constant on every microcell $\e (Y + k)$ with $k\in K_{\e}$, we define the evolving microdomain $\Oe(t)$ by
\begin{align*}
    \Oe(t) := \Omega \setminus \bigcup\limits_{k \in K_\e} \overline{B_{\Reps,k}(t)},
\end{align*}
where 
\begin{align*}
    B_{\Reps,k}(t) \coloneqq B_{\Reps(t,x)}( \e( \m + k))
\end{align*}
with $x \in \e (k + Y)$ and  $\m$ the center of the cube $Y$. Further, we denote the evolving surface of the solid grains by
\begin{align*}
    \geps(t):= \partial \Oe(t) \setminus \partial \Omega.
\end{align*}
The non-cylindrical domains are now given by
\begin{equation}\label{eq:QH}
Q_{\e}^T := \bigcup_{t\in (0,T)} \{t\} \times \oet, \quad 
\text{ and } \quad 
H_{\e}^T := \bigcup_{t \in (0,T)} \{t\} \times \get.
\end{equation}
Finally, we let $\nu_{\e}$ stand for the unit normal of $\get$ pointing out of $\oet$.

 We consider the following microscopic problem:   \\[0.5em]
 {\bf Problem P$_\e$}. Find $\ueps, \reps, \peps: Q_{\e}^T \rightarrow \R$ and $\veps: Q_{\e}^T \rightarrow \R^3$ such that $\ueps$ solves the transport equation 
\begin{subequations}\label{MicroscopicModel}
\begin{align}
\partial_t \ueps - \div \big(D\nabla \ueps - \veps \ueps  \big) &= f(\ueps) &\mbox{ in }& Q_{\e}^T,
\\\label{eq:RH}
-(D\nabla \ueps - \veps \ueps) \cdot \nu_{\e} &= -\partial_t \reps \ueps + \rho g(\ueps, \e^{-1} \Reps) &\mbox{ on }& H_{\e}^T,
\\\label{eq:DirichletBCueps}
\ueps &= 0 &\mbox{ on }& (0,T)\times \partial \Omega,
\\
\label{MicroscopicModel_ICu}
\ueps(0) &= \uepsin &\mbox{ in }& \Omega_{\e}(0),
\end{align}
the radius of the balls in the microcells solves the ODE
\begin{align}
    \label{MicroscopicModel_ODE}
\partial_t \reps &= \e  \fint_{\geps\left(t,\e \bfxe\right)} g\left(\ueps(t,z),\frac{\reps(t,z)}{\e}\right) d\sigma_z  &\mbox{ in }& Q_{\e}^T,
\\
\label{MicroscopicModel_ICR}
\reps(0) &= \repsin &\mbox{ in }& \Omega_{\e}(0),
\end{align}
and the fluid velocity $\veps$ and the fluid pressure $\peps$ are solutions of the Stokes problem
\begin{align}\label{eq:Stokes}
    -\e^2 \div (e(\veps)) + \nabla \peps &= h_{\e} &\mbox{ in }& Q_{\e}^T,
    \\
    \nabla \cdot \veps &= 0 &\mbox{ in }& Q_{\e}^T,
    \\\label{eq:Stokes:NormalVelocityBC}
    \veps &=-\partial_t \reps \nu_{\e} &\mbox{ on }& H_{\e}^T
    \\\label{eq:Stokes:NormalStressBC}
    \left(- \e^2 e(\veps) + \peps I \right) \nu &= p_\mathrm{b} \nu &\mbox{ on }& (0,T)\times \partial \Omega,
\end{align}
with the symmetric gradient
\begin{align*}
    e(\veps) := \frac12 \left(\nabla \veps + \nabla \veps^T \right).
\end{align*}
\end{subequations}
Observe that the velocity $|\veps|$ in \eqref{eq:Stokes:NormalVelocityBC} equals $|\partial_t \reps|$, which is analogous to assuming that the molar density of the solid mineral layer is negligible when dissolved in the fluid phase. Therefore, the free boundary and the fluid move with the same velocity. We refer to \cite{Peter2007,Noo09} for details. Further, by assuming that the growth is radially symmetric, the mass balance in \eqref{eq:RH} does not necessarily hold pointwisely, but only in an averaged sense. More precisely, this mass balance is obtained when averaging \eqref{eq:RH} over each of the grains, so one can interpret this as a cell-averaged Rankine--Hugoniot condition. 

Furthermore, note that we have chosen the viscosity in the Stokes problem equal to $\tfrac{1}{2}$ for sake of shorter notation. 
Since the model considered here is dimensionless, this is not a restriction but the result of the choice of the scaling.

\subsection{Microscopic problem on the fixed domain}
We reformulate {\bf Problem P$_\e$} on an in-time cylindrical and $\e$-periodic domain  by transforming \eqref{MicroscopicModel} via a family of diffeomorphisms $\psi_\e(t, \cdot ) : \Oe \to \Oe(t)$, for a periodic reference domain $\Oe$, which is given by a reference radius $\bar{R} \in [\rmin,\rmax]$ via
\begin{align}
    \Oe = \Omega \setminus \bigcup\limits_{k \in K_\e}  \overline{B_{\e \bar{R}}( \e( \m + k)}.
\end{align}
We construct $\psi_\e$ explicitly by the corresponding displacement mapping $\widecheck{\psi}_\e(t,x) = \psi_\e(t,x) -x$.
Therefore, we construct the following generic displacement in the reference cell:
let $\chi \in C^\infty([0,\infty);[0,1])$ be monotonically decreasing with $\chi(y) = 1$ for $s \leq \rmax$ and $\chi(s) = 0$ for $s \geq (\rmax +0.5) /2$, then
\begin{align*}
    \check{\psi}(R;y):= (R -R_\mathrm{max})  \tfrac{y-m}{\|y-m\|} \chi(\| y-\m \|).
\end{align*}
Then, the mapping
\begin{align}\label{eq:def:psi}
    y \mapsto \psi(R;y) := y + \check{\psi}(R;y) = y + (R-\rmax) \tfrac{y-m}{\|y-m\|} \chi(\| y-\m \|),
\end{align}
 is a diffeomorphism from 
$Y^*_{\rmax}$ onto $Y^*_R$ for every $R \in [\rmin,\rmax]$, where $Y^*_R \coloneqq (0,1)^n \setminus \overline{B_{\bar{R}}( \m)}$ for $R \in  [\rmin,\rmax]$ and $Y^*\coloneqq  Y^*_{\bar{R}} = Y^*_{\rmax}$.

By $\e$-scaling the generic displacement, we obtain
\begin{align}\label{eq:def:check_psi_eps}
\widecheck{\psieps}(t,x):=  \e \check{\psi}\left(\e^{-1}\reps(t,x);\{\tfrac{x}{\e}\}\right)
\end{align}
for $R_\e(t,\cdot)$ constant on $\e (k + Y^*)$, where $[x]$ denotes the integer part of $x$ and  $\{x\} \coloneqq x - [x]$. We note that $\widecheck{\psi}$ is zero in a neighbourhood of the boundary of $Y$ and, thus, $\widecheck{\psieps}$ is also smooth at the cell interfaces $\e (k +\partial Y)$ for $k \in K_\e$, even though $\Reps(t,x)$ and $\{\tfrac{x}{\e}\}$ contain jumps at the interface of the cells.
From the $\e$-scaled displacements, we obtain the $\e$-scaled diffeomorphisms
\begin{align}\label{eq:def:psi_eps}
    \psieps(t,x)= x +  \e \check{\psi}\left(\e^{-1}\reps(t,x);\{\tfrac{x}{\e}\}\right)
\end{align}

We use the following notation for the Jacobians of $\psieps$, its adjugate and the transformed diffusion coefficient:
\begin{align}
    \label{def:Feps}
    \Feps &= \partial \psieps = \nabla \psieps^\top,
    \\
    \label{def:Jeps}
    \Jeps &= \det(\Feps),
    \\
    \label{def:Aeps}
    A_{\e}&= \jeps \feps^{-1},
    \\
    \label{def:Deps}
    D_{\e} &= \jeps \feps^{-1} D \feps^{-T}.
\end{align}

Having such a well-posed mapping between the radii $R_\e$ and the domain transformation $\psi_\e$ at hand, the microscopic model in the fixed domain reads as follows: 
\\
\begin{subequations}\label{MicroModelFixedDomain}
\noindent Reaction--diffusion--advection equation:
\begin{align}\label{eq:StrongFormDiffusion:Transformed}
\begin{aligned}
    \partial_t (\jeps \ueps) - \div \left( D_{\e} \nabla \ueps  - \ueps \Aeps(\veps-\partial_t\psieps)  \right) &= \jeps f(\ueps) &\mbox{ in }& (0,T)\times \oe,
    \\
    -(D_{\e} \nabla \ueps - \ueps \Aeps(\veps-\partial_t\psieps)) \cdot \nu &= \jeps \rho g(\ueps, \e^{-1}\Reps) &\mbox{ on }& (0,T)\times \geps,
    \\
    \ueps&= 0 &\mbox{ on }& (0,T)\times  \partial \Omega,
    \\
    \ueps(0) &= \ueps^\mathrm{in} &\mbox{ in }& \Oe.
\end{aligned}
\end{align}
\\

\noindent Evolution of the radii:
\begin{align}
\label{eq:StrongFormR:Transformed}
    \partial_t \reps &= \e \fint_{\geps\left(\e \bfxe\right)} g\left(\ueps(t,z),\frac{\reps(t,z)}{\e}\right) \,\mathrm{d}\sigma_z  &\mbox{ in }& (0,T)\times \oe,
    \\
    \reps(0) &= \reps^\mathrm{in} &\mbox{ in }& \oe,
\end{align}
with $ \geps\left(z\right) \coloneqq \e \Gamma + z \textrm{ for } z \in \R^n.$
\\

\noindent Stokes equations: 
\begin{align}\label{eq:StrongFormStokes:v:Transformed}
\begin{aligned}
    -\e^2 \div(A_\e e_\e(v_\e)) +  A_{\e}^\top \nabla \peps  &= J_\e h_{\e} &\mbox{ in }& (0,T)\times \oe,
    \\
    \nabla \cdot (A_{\e} \veps ) &= 0 &\mbox{ in }& (0,T)\times \oe,
    \\
    \veps &= -\partial_t \reps \nu &\mbox{ on }& (0,T)\times \geps,
    \\
    \left(-\e^2e_\e(v_\e)  +  \peps I    \right)A_{\e}^\top\nu &= p_\mathrm{b} A_{\e}^\top\nu &\mbox{ on }& (0,T)\times \partial \Omega,
\end{aligned}
\end{align}
\end{subequations}
where 
\begin{align}
e_\e(v) \coloneqq \frac{1}{2} \left(\feps^{-\top} \nabla v + \nabla v^\top \feps^{-1} \right).
\end{align}

\begin{remark}\label{RemarksTrafoVelocity}\
\begin{enumerate}[label = (\roman*)]
\item The divergence equation for the Stokes system can be written as 
\begin{align*}
A_{\e} : \nabla \veps  = \nabla  \cdot (A_{\e} \veps) = 0.
\end{align*}
\item The boundary condition for the fluid velocity $\veps$ on $\geps$ can be expressed as
\begin{align*}
   -\partial_t \reps \nu = \partial_t \psieps  \qquad\mbox{on } (0,T)\times \geps.
\end{align*}
\item Since $\psieps(t,x) = x$ holds on $\partial \Omega$, we obtain for the stress boundary condition for the fluid
\begin{align*}
    \left(-\e^2 e(\veps) + \peps I \right) \nu  = p_\mathrm{b} \nu.
\end{align*}
However, in what follows, we also use the transformed stress boundary condition including $\psieps$, which also covers more general geometries. For example, in \cite{wiedemann2021homogenisation} an a priori known evolving microscopic domain was considered with perforations intersecting the outer boundary $\partial \Omega$. In this case, one has to consider $p_\mathrm{b} \circ \psieps$. 
\item From the Piola identity $\nabla \cdot A_{\e} = 0$ and the Jacobi formula for the derivative of the determinant, we obtain 
\begin{align*}
    \nabla \cdot (A_{\e} \weps ) = -\nabla \cdot (A_{\e} \partial_t \psieps ) = -\partial_t \jeps.
\end{align*}
\end{enumerate}

\end{remark}

In order to formulate the Stokes equations as a problem with a no-slip condition on $\geps$, we subtract the Dirichlet boundary values $-\partial_t \Reps \nu = \partial_t \psieps$  from the solution $\veps$. More precisely, we consider $\weps:= \veps - \partial_t \psieps$. From the Piola identity $\eqref{PiolaIdentity}$ (see also $\eqref{Id:dtDet_V}$), we get $\nabla \cdot (\Aeps \weps) = -\partial_t \jeps$.  Moreover, to simplify the notation, we subtract the normal stress boundary values from the pressure $\peps$, i.e.~$q_\e \coloneqq \peps - p_\mathrm{b}$.  Thus, we get the following Stokes system:
\begin{align}
\begin{aligned}
    -\e^2 \div(A_\e e_\e(w_\e)) + A_{\e}^\top \nabla \qeps  &= J_\e h_{\e} +\e^2 \div( A_\e e_\e(\partial_t \psieps)) - A_{\e}^\top \nabla p_\mathrm{b} &\mbox{ in }& (0,T)\times \oe,
    \\
    \nabla \cdot (A_{\e} \weps ) &= - \nabla \cdot (A_{\e} \partial_t \psieps ) &\mbox{ in }& (0,T)\times \oe,
    \\
    \weps &= -\partial_t \reps \nu &\mbox{ on }& (0,T)\times \geps,
    \\
    \left(-\e^2e_\e(w_\e)  +  \qeps I    \right)A_{\e}^\top\nu &=  \e^2e_\e(\partial_t \psieps) A_{\e}^\top\nu  &\mbox{ on }& (0,T)\times \partial \Omega.
\end{aligned}
\end{align}

\subsection{The weak formulation}

We give the definition of a weak solution of the microscopic model $\eqref{MicroModelFixedDomain}$ in the fixed domain. 
\begin{definition}\label{def:DefinitionWeakSolutionMicroModel}
We call $(\ueps,\reps,\weps,\qeps)$ a weak solution the micromodel $\eqref{MicroModelFixedDomain}$ if 
\begin{align*}
    \ueps &\in L^2((0,T),H^1_{\partial \Omega}(\oe)) \quad \mbox{ with } \partial_t (\jeps \ueps) \in L^2((0,T),H^1_{\partial \Omega}(\oe)'),
    \\
    \reps &\in W^{1,\infty}((0,T),L^2(\oe))\cap L^{\infty}((0,T)\times \oe) \quad \mbox{ with } \partial_t \reps \in L^{\infty}((0,T)\times \oe),
    \\
    \weps &\in L^2((0,T),H^1_{\Ge}(\oe)^n) ,
    \\
    \qeps &\in L^2((0,T),L^2(\oe)),
\end{align*}
such that, for every $\phieps \in H_{\partial \Omega}^1(\oe)$, there holds almost everywhere in $(0,T)$  
\begin{subequations}\label{eq:WeakForm:Microproblem}
\begin{align}
\begin{aligned}\label{eq:weakForm:u_eps}
    \langle \partial_t (\jeps \ueps) , \phieps \rangle_{H^1_{\partial \Omega}(\oe)',H^1_{\partial \Omega}(\oe)} + \int_{\oe} &\left[ D_{\e} \nabla \ueps - \ueps \Aeps\weps\right] \cdot \nabla \phieps \,\mathrm{d}x
    \\
    &= \int_{\oe} \jeps f(\ueps) \phieps dx - \rho \int_{\geps} \jeps g(\ueps, \e^{-1} \Reps) \phieps \,\mathrm{d}\sigma. 
\end{aligned}
\end{align}
It holds almost everywhere in $(0,T)$ that  
\begin{align}
\begin{aligned}\label{eq:weakForm:r_eps} \partial_t \reps =\e \fint_{\geps\left(\e \bfxe\right)} g\left(\ueps(t,z),\frac{\reps(t,z)}{\e}\right) \,\mathrm{d}\sigma_z  &\mbox{ in }& (0,T)\times \oe .
\end{aligned}
\end{align}
Moreover, for all $\eta_{\e} \in H^1_{\geps}(\oe)^n$, there holds almost everywhere in $(0,T)$  
\begin{align}
\begin{aligned}\label{eq:weak_w_eps_fixed_domain}
 \e^2 &\int_{\oe} \jeps e_\e(\weps ) : e_\e (\eta_{\e})  \,\mathrm{d}x - \int_{\oe } \qeps \nabla \cdot (A_{\e} \eta_{\e})  \,\mathrm{d}x \\
    &= -
    \e^2 \int_{\oe} \jeps e_\e(\partial_t \psi_\e) : e_\e (\eta_{\e}) \,\mathrm{d}x
    +
    \int_{\oe } J_{\e} h_{\e } \cdot \eta_{\e} \,\mathrm{d}x - \int_{\Oe }A_{\e}^{\top} \nabla p_\mathrm{b}   \cdot \eta_{\e} \,\mathrm{d}\sigma,
\end{aligned}
\end{align} 
and, almost everywhere in $(0,T)\times \Oe$, we have
\begin{align}
    \div(A_\e w_\e) &= - \partial_t \jeps.
\end{align}
Further, we have the initial condition $(\ueps,\reps)(0) = (\ueps^\mathrm{in},\reps^\mathrm{in})$.
\end{subequations}
\end{definition}

\begin{remark}\label{Rem:Def_Weak_formulation}\
\begin{enumerate}[label = (\roman*)]
\item We note that we can rewrite equation \eqref{eq:weak_w_eps_fixed_domain} in the following way
\begin{align*}
    \e^2 &\int_{\oe}  e_{\e}(\weps ) : A_\e^\top \nabla \eta_{\e} \,\mathrm{d}x - \int_{\oe } \qeps \div(A_{\e }  \eta_{\e}) \,\mathrm{d}x 
    \\
    &= -\e^2 \int_{\oe}  e_{\e}(\partial_t \psieps ) : \Aeps^\top \nabla \eta_{\e} \,\mathrm{d}x +\int_{\oe } J_{\e} h_{\e } \cdot \eta_{\e} \,\mathrm{d}x - \int_{\Oe} A_{\e }^\top\nabla p_\mathrm{b}  \cdot \eta_{\e} \,\mathrm{d}x.
\end{align*}
\item The initial condition for $\ueps$ makes sense since the regularity of $\ueps$, $\jeps$, and $\partial_t (\jeps \ueps)$ imply $\ueps \in C^0([0,T],L^2(\oe))$. In fact, as a consequence of the product rule, we obtain $\partial_t \ueps \in L^2((0,T),H_{\partial \Omega}^1(\Oe))$, which implies the continuity of $\ueps$ as a mapping from $[0,T]$ to $L^2(\Oe)$. However, in the context of homogenisation, it is not appropriate to work with $\partial_t \ueps$ since this function does not fulfill a uniform \textit{a priori} bound with respect to $\e$.

\item From the transformation of surface integrals, see for example \cite[Theorem 1.7-1]{ciarlet1988mathematical}, we first obtain a boundary term of the form
\begin{align*}
    -\rho \int_{\geps} \jeps |\nabla \psieps^{-T}\nu | g(\ueps, \e^{-1} \Reps) \phi \,\mathrm{d}\sigma
\end{align*}
in $\eqref{eq:weakForm:u_eps}$.
However, in our case, we have $|\nabla \psieps^{-T}\nu | = 1$ owing to the choice of the transformation in $\eqref{eq:def:psi_eps}$. For more details, we refer to the \cite[Appendix A]{GahnPop2023Mineral}.

\item\label{Rem:Def_Weak_formulation_Adv} For the derivation of \textit{a priori} estimates for the microscopic solutions, it might be helpful to consider a time-integrated version of the transport equation $\eqref{eq:weakForm:u_eps}$ and choose time-dependent test functions $\phieps(t,x)$. However, under the regularity assumptions in Definition \ref{def:DefinitionWeakSolutionMicroModel}, $\ueps$ is not an admissible test function since the advective term is not well-defined. In fact, the H\"older inequality is not applicable owing to low regularity with respect to the time variable. In the existence proof, where we work with the Galerkin method and, therefore, with more regular approximations, this causes no problems since the advective term cancels out after an integration by parts (where we have to use the boundary and divergence conditions for $\ueps$ and $\weps$), see equation $\eqref{Eq:Adv_Term_Apriori}$ in the proof of the \textit{a priori} estimates.

\item To keep the notation of the weak solution more simple, we only require enough regularity such that the variational equations are well defined (pointwise in time). We will show, see Lemma \ref{LemmaAprioriEstimates} and \ref{LemmaLinftyBoundUeps}, that the microscopic solution is in fact more regular. More precisely, we additionally have
\begin{align*}
    \ueps &\in L^{\infty}((0,T)\times \Oe),
    \\
    \weps &\in L^{\infty}((0,T),H^1_{\geps}(\Oe)^n),
    \\
    \qeps &\in L^{\infty}((0,T),L^2(\Oe)).
\end{align*}
\end{enumerate}
\end{remark}

In what follows, we summarise the assumptions on the data. For $h_{\e}$, we neglect its dependence on the transformation $\psieps$ and we also assume that $f_{\e}$ does not depend explicitly on $t$ and $x$. In this, way we can keep the notation simpler and the focus on the relevant aspects of the paper. Of course, more general assumptions are possible under suitable regularity conditions. In what follows, we use the notion of two-scale convergence, the definition of which can be found in Section \ref{SectionTwoScaleConvergence} in the appendix.

\noindent\textbf{Assumptions on the data:}
\begin{enumerate}[label = (A\arabic*)]
\item We assume that $g\in C^{0,1}(\R^2)$ is bounded. More precisely, there exists a constant $C_g>0$ such that
\begin{align}\label{eq:def:C_g}
|g| \leq C_g,
\end{align}
and a constant $L_g \in  \R$ such that for every $u_1,u_2, r_1,r_2 \in \R$ 
\begin{align}
|g(u_1,r_1) -g(u_2,r_2)|  \leq L_g (|u_1 -u_2| +|r_1-r_2|)
\end{align}
holds.
Moreover, we assume that
\begin{align}\label{eq:g_Bound_below}
g(u,r) \geq 0 \textrm{ if } r\leq \rmin,
\\\label{eq:g_Bound_above}
g(u,r) \leq 0 \textrm{ if } r\geq \rmax
\end{align}
for $u \in \R$ in order to ensure that the radii do not become too small or too large, i.e.~the obstacles do not vanish or touch each other (or the boundary of $Y$).

\item  For the microscopic initial data $\ueps^\mathrm{in}\in L^{\infty}(\oe)$,  a $C > 0$ exists such that 
\begin{align*}
\Vert \ueps^\mathrm{in} \Vert_{L^{\infty}(\oe)} \le C 
\end{align*}
uniformly with resepct to $\e$. 
Further, there exists $u^\mathrm{in} \in L^2(\Omega)$ such that $\ueps^\mathrm{in} \rightwts{2} u^\mathrm{in}$.
\item  \label{AssumptionInitialValueReps} The initial radii $\reps^\mathrm{in} \in L^{\infty}(\Omega)$ are constant on every cell $\e (Y + {\bf k})$, ${\bf k} \in K_{\e}$. Furthermore, there exist $R_{\min}, R_\mathrm{max} \in \R$, $0 < R_\mathrm{min} < R_\mathrm{max} < \frac12$, such that almost everywhere in $\Omega$
\begin{align*}
\e R_{\min} \le \repsin \le \e R_\mathrm{max}.
\end{align*}
We also assume that 
\begin{align*}
    \e^{-1} \reps^\mathrm{in} \rightarrow R_0^\mathrm{in} \qquad \mbox{in } L^2(\Omega).
\end{align*}

\item\label{Assumptionf} For the reaction rate in the transport equation, we assume $f \in C^{0,1}(\R)$. In particular, there holds the following growth condition: 
\begin{align}
    |f(u)|\le C(1 + |u|)
\end{align}
for a constant $C>0$ and for all $u\in \R$.

\item\label{ass:h} For the force term in the Stokes equations, we assume $h_{\e} \in L^{\infty}((0,T),L^2(\Oe))^n$ and $h_{\e} \rightwts{p,2} h_0$ with $h_0 \in L^{\infty}((0,T),L^2(\Omega))^n$ for all $p \in [1,\infty)$.

\item For the external pressure, we assume $p_\mathrm{b} \in L^{\infty}((0,T),H^{\frac12}(\partial \Omega))$. In particular, there exists an extension (we use the same notation) $p_\mathrm{b} \in L^{\infty}((0,T),H^1(\Omega))$.

\item The diffusion coefficient $D\in \R^{n\times n}$ is symmetric and positive definite.
\end{enumerate}

\begin{remark}
Note that the transformation of Problem P$_\e$ would yield $h_\e(t,x) \coloneqq h(t, \psi_\e(t,x))$. For sake of simplicity, we assume that $h_\e$ is a given sequence. Nevertheless, we want to stress
that the following analysis 
also 
holds for $h_\e(t,x) \coloneqq h(t, \psi_\e(t,x))$, 
where 
$h \in L^\infty((0,T),C^{0,1}(\Omega))$. 
Thereby, the Lipschitz regularity with respect to $\Omega$ is only necessary for the proof of Lemma \ref{lem:Lip_L2}.
The two-scale convergence of $ h(t, \psi_\e(t,x)) \to h(t,x)$ was shown in \cite[Theorem 3.8]{AA23}. 
\end{remark}

\subsection{Properties of the domain transformation}
In what follows, we show several $\e$-scaled a priori estimates for the transformation $\psieps$, which was defined in \eqref{eq:def:psi_eps}, as well as its Lipschitz-continuous dependence on the radii. We will employ these properties for proving the existence and uniqueness of a microscopic solution of \eqref{eq:WeakForm:Microproblem} and for the derivation of $\e$-independent a priori estimates of the solution.
Moreover, in Lemma \ref{lem:LemConvData}, we show that several two-scale convergence results for the transformation $\psieps$ follow if $R_\e$ converges. We will use this convergence in order to pass to the two-scale homogenisation limit.
The direct two-scale limit problem as well as the homogenised problem of \ref{lem:LemConvData} contain coefficients which depend on the limit of $\psieps$ for $\e \to 0$. Nevertheless, we show in the Subsection \ref{subsec:BAcktrafoCellProblems} how the cell problems and effective tensors can be transformed back on reference cells, the domains of which depend on their macroscopic position and time. Thus, the homogenised problem and the cell problems become independent of the transformations $\psieps$ and its limit.
Having this in mind, the transformation $\psieps$ is reduced to a tool for proving existence and uniqueness of the microscopic problem, as well as the homogenisation. 
Such diffeomorphisms have been explicitly constructed in \cite{GahnPop2023Mineral} and \cite{WIEDEMANN2023113168}. For the above choice of the transformation, we verify these properties in Lemma \ref{lem:CellTrafo}, Lemma \ref{lem:EstimatesEpsTrafo} and Lemma \ref{lem:LemConvData}.

In order to enhance the understanding of the $\e$-scaled diffeomorphisms $\psieps$, we note that \eqref{eq:def:psi_eps} and \eqref{eq:def:psi} yield
\begin{align}
\begin{aligned}\label{eq:psieps=x+psi}
    \psi_\e(t,x) &= x +\e \psi\left(\e^{-1}\reps(t,x);\{\tfrac{x}{\e}\}\right) - \{\tfrac{x}{\e}\}
\\
&
=\e [\tfrac{x}{\e}] + \e \psi\left(\e^{-1}\reps(t,x);\{\tfrac{x}{\e}\}\right).
\end{aligned}
\end{align}
Thus, it can be seen that $\psieps$ is nothing but a union of diffeomorphisms $\psi\left(\e^{-1}\reps(t,x); \cdot\right)$ for every $\e$-scaled cell. Hence, the bijectivity of $\psi\left(\e^{-1}\reps(t,x); \cdot\right)$ from $Y^*$ onto $Y^*_{\e^{-1}\reps(t,x)}$ can be transfered into the bijectivity of $\psi_\e(t,\cdot)$ from $\Oe$ onto $\Oe(t)$.

Moreover, the $\e$-scaled estimates for $\psieps$ can be obtained from a corresponding results for $\psi\left(\e^{-1}\reps(t,x); \cdot \right)$ which is independent of $R_\e$.

\begin{lemma}[Cell transformation]\label{lem:CellTrafo}
Let $\chi \in C^\infty([0,\infty);[0,1])$ be monotonically decreasing with $\chi(\lambda) = 1$ for $\lambda \leq \rmax$ and $\chi(\lambda) = 0$ for $\lambda \geq (\rmax +0.5) /2$. Then, $\psi(R,\cdot)$ as defined in \eqref{eq:def:psi} is a diffeomorphism from $Y^*$ onto $Y^*_R$ for every
$R \in [\rmin,\rmax]$ such that 
\begin{align*}
\psi(R;Y^*) &= Y_R^* &&\textrm{ for every } R \in [\rmin,\rmax],
\\
\psi(R; x ) &= x &&\textrm{ for every } R \in [\rmin,\rmax], x \in Y \setminus Y_{(\rmax + 0.5)/2}^* ,
\\
\norm{\psi(R; \cdot)}{C^2(\overline{Y})} &\leq C  &&\textrm{ for every } R \in [\rmin,\rmax],
\\
\det(D_x \psi(r;x) ) &\geq c_J &&\textrm{ for every }  R \in [\rmin,\rmax], x \in Y^*,
\\
R &\mapsto \psi(R; \cdot) &&\in C^\infty([\rmin,\rmax], C^2(Y^*)^n)).
\end{align*}
\end{lemma}
\begin{proof}
After translation, we can assume without loss of generality that $\m=0$ and $Y^{\ast} = \left(-\frac12,\frac12\right) \setminus \overline{B_{R_\mathrm{max}}(0)}$ (and, similarly, we redefine $Y_R^{\ast}$). First, we consider $\psi $ as a function from $\R^n\setminus \{0\}$ into $\R^n$. With $\alpha(\lambda):= \lambda + (R - \rmax ) \chi(\lambda)$, for $\lambda \geq 0$, we can write $\psi $ in the following way for $y \in \R^n\setminus\{0\}$
\begin{align}\label{RepresentationPsi}
    \psi(y) = \alpha(|y|) \dfrac{y}{|y|}.
\end{align}
Since $\chi$ is decreasing and $R-\rmax\leq 0$, we obtain that $\alpha$ is strictly increasing. Hence, with $\alpha(\rmax)= R$ and $\lim_{\lambda \to \infty} \alpha(\lambda) = \infty$, we get that $\alpha:[\rmax,\infty) \rightarrow [R,\infty)$ is bijective. This implies that $\psi: \R^n \setminus {B_{\rmax}(0)} \rightarrow \R^n\setminus {B_R(0)}$ is bijective. In fact, for $z \in \R^n$ with $|z|\geq R$, there exists $\lambda_z \geq \rmax$ such that $\alpha (\lambda_z) = |z|.$ For $y:= \lambda_z \frac{z}{|z|}$, we get from $\eqref{RepresentationPsi}$ that $\psi(y) = z$ and, therefore, $\psi$ is surjective. For $y_1,y_2 \in \R^n \setminus B_{\rmax}(0)$ with $\psi(y_1) = \psi(y_2)$, we obtain (taking the absolute values), using again $\eqref{RepresentationPsi}$, that $\alpha(|y_1|) = \alpha(|y_2|)$, and the bijectivity of $\alpha$ gives $|y_1|=|y_2|$. In particular, this implies $y_1 = y_2$, so $\psi: \R^n \setminus {B_{\rmax}(0)} \rightarrow \R^n\setminus {B_R(0)}$ is injective. Finally, since $\psi(y) = y $ for $y\geq \frac12 (\rmax + 0.5)$, we obtain that $\psi: Y^{\ast} \rightarrow Y_R^{\ast}$ is bijective.
\end{proof}
We note that the construction of $\psi(R;\cdot)$ by \eqref{eq:def:psi} does in fact not give a diffeomorphism from $Y$ onto $Y$. In \cite{WIEDEMANN2023113168}, a diffeomorphism with the same properties from $Y$ onto $Y$ was constructed. 
Further, the corresponding cell displacement $\widecheck{\psi}(R, y)$ vanishes at $\partial Y$ and, thus, $\widecheck{\psi}(R; \cdot) \in C^\infty_{\per}(\overline{Y^*};\R^n)$ for all $R \in[\rmin, \rmax]$. 
Next, we transfer the radius-independent estimates of Lemma \ref{lem:CellTrafo} into $\e$-uniform estimates for $\psi_\e$ with the following lemma.
\begin{lemma}\label{lem:EstimatesEpsTrafo}
Let $\reps \in W^{1,\infty}((0,T),L^2(\oe))\cap L^{\infty}((0,T)\times \oe)$ be constant on every micro cell with $\|\partial_t \reps\|_{L^{\infty}((0,T)\times \oe)} \leq C_g$ and $\e^{-1} \reps \in [\rmin, \rmax]$. Let $\psi_\e$ be defined via \eqref{eq:def:psi_eps}
with its Jacobian matrix $F_\e = \partial \psieps = \nabla \psieps^\top$ and Jacobian determinant $J = \det(\Feps)$. Moreover, let $\Deps, \Aeps$ be given by \eqref{def:Deps}, \eqref{def:Aeps}. 
Then, there exists constants $C, \alpha >0$ such that  almost everwhere in $(0,T)$
\begin{align}\label{eq:psi-id}
&\e \norm{\psi_\e - \operatorname{id}_{\Oe}}{L^\infty(\oe)} + \norm{\Feps}{L^\infty(\oe)} + \norm{A_\e}{L^\infty(\oe)}   \leq C,
\\
&
\e \norm{\partial_x \Feps}{L^\infty(\oe)} 
+ \e \norm{D_x J_\e}{L^\infty(\oe)} +  \e \norm{\partial_x A_\e}{L^\infty(\oe)}  \leq  C,
\\\label{eq:dt_psi<C}
&\e^{-1} \norm{\partial_t \psi_\e}{L^\infty(\Oe) }+ \norm{\partial_x \partial_t \psi_\e}{L^\infty(\Oe) }+ \norm{\partial_t J_\e}{L^\infty(\Oe) } \leq C,
\\\label{eq:J>c_J}
&J_\e(t,x) \geq c_J,
\\
\label{eq:coercivityFeps}
&\zeta^\top D_\e\zeta  = \zeta^\top J_\e F_\e^{-1}  D F_\e^{-\top} \zeta \geq \alpha \norm{\zeta}{}^2  \textrm{ for all }\zeta \in \R^n.
\end{align}
Moreover, we obtain for every $v \in H^1_{\Ge}(\Oe)$
\begin{align}
\label{eq:Sym-coercivityFeps}
&(\e^2 e_\e(v), \Aeps^\top \nabla v) \geq \alpha \norm{\nabla v}{L^2(\Oe)}^2,
\end{align}
\begin{align}\label{eq:Inf-Sup}
\inf\limits_{q \in L^2(\Oe)} \sup\limits_{v \in H^1_\Ge(\Oe)} \frac{(q, \div(\Aeps v))_{\Oe}}{\e\norm{\nabla v}{L^2(\Oe)}\norm{q}{L^2(\Oe)}} \geq \beta.
\end{align}
\end{lemma}
\begin{proof}
The estimates \eqref{eq:psi-id}--\eqref{eq:dt_psi<C} follow by elementary calculations via scaling the estimates from Lemma \ref{lem:CellTrafo}, see also \cite{GahnPop2023Mineral,WIEDEMANN2023113168} for more details. The coercivity \eqref{eq:coercivityFeps} can be concluded from \eqref{eq:psi-id} and \eqref{eq:J>c_J} as in \cite[Proposition~4.1]{AA23}. The coercivity \eqref{eq:Sym-coercivityFeps} as well as the inf--sup condition \eqref{eq:Inf-Sup} can be concluded as in \cite{wiedemann2021homogenisation}.
\end{proof}

\begin{lemma}\label{lem:EstimateShiftsReps}  
Let $\Reps^1, \Reps^2$ fulfill the assumptions of Lemma \ref{lem:EstimatesEpsTrafo} and let $\psieps^i$, $i\in \{1,2\}$, be given by \eqref{eq:def:psi_eps} and $\Feps^i, \Aeps^i, \Deps^i$, accordingly.
Then, the following estimates  hold almost everywhere in $(0,T)$:
\begin{align*}
&\e^{-1}\norm{\psieps^1 -\psieps^1}{L^\infty(\Oe)} 
+
\norm{\Feps^1 -\Feps^2}{L^\infty(\Oe)} \leq \e^{-1} C\norm{R_\e^1 -\Reps^2}{L^\infty(\Oe)},
\\
& \norm{\Aeps^1 -\Aeps^2}{L^\infty(\Oe)}+
\norm{\Deps^1 -\Deps^2}{L^\infty(\Oe)}
+
\norm{\Jeps^1 -\Jeps^2}{L^\infty(\Oe)}\leq  \e^{-1}C
\norm{\Reps^1 -\Reps^2}{L^\infty(\Oe)} ,
\\
&\e\norm{\partial_x (\Aeps^1 -\Aeps^2)}{L^\infty(\Oe)}\leq \e^{-1}C\norm{\Reps^1 -  \Reps^2}{L^\infty(\Oe)},
\\
& \e^{-1}\norm{\partial_t \psieps^1 -\partial_t \psieps^2}{L^2( \Oe)} \leq  \e^{-1}C\norm{\Reps^1 -\Reps^2}{L^2(\Oe)}  +  \e^{-1}C\norm{\partial_t(\Reps^1 -\Reps^2)}{L^2(\Oe)} ,
\\
&\norm{\nabla\partial_t  \psieps^1 -\nabla \partial_t \psieps^2}{L^2(\Oe)}\leq \e^{-1}C\norm{\Reps^1 -\Reps^2}{L^2(\Oe)}  + \e^{-1} C\norm{\partial_t(\Reps^1 -\Reps^2)}{L^2(\Oe)} ,
\\
& \norm{\partial_t  \Jeps^1 -\partial_t \Jeps^2}{L^2(\Oe)}\leq \e^{-1} C\norm{\Reps^1 -\Reps^2}{L^2(\Oe)}  +\e^{-1} C\norm{\partial_t(\Reps^1 -\Reps^2)}{L^2(\Oe)} .
\end{align*}
\end{lemma}
\begin{proof}
We note that Lemma \ref{lem:CellTrafo} provides  $R \mapsto \psi(R; \cdot) \in C^\infty([\rmin,\rmax], C^2(Y^*)^n))$. 
Thus,  $\psi$ as well as all its derivatives depend Lipschitz continuously on $R$. After rescaling, this can be carried over to $\psieps$ via \eqref{eq:psieps=x+psi}, which yields the Lipschitz estimates above. For the Lipschitz estimate of $D_\e$, the Lipschitz estimate of $F^{-1}_\e$ is used in addition. This, estimate can be obtained by the Lipschitz estimate of $F^{-1}_\e$ and the uniform bound $J_\e \geq c_J$ from below. 
We note that, for the Lipschitz estimate of $A_\e = J_\e F_\e^{-1}$, this is not necessary since $A_\e$ is the adjugate of $F_\e$, and, thus, is nothing but a polynomial with respect to the entries of $F_\e$.
\end{proof}

\section{Main results and macroscopic model}
\label{SectionMainResults}

Let us summarise the main results of the paper in two theorems, which will be proved in the subsequent sections. The first is the existence and uniqueness of a weak microscopic solution of the problem $\eqref{MicroModelFixedDomain}$:

\begin{theorem}\label{MainThmExistenceUniqueness}
There exists a unique weak solution of the microscopic problem $\eqref{MicroModelFixedDomain}$.
\end{theorem}

In a second step, we show suitable compactness results for the microscopic solutions and pass to the limit $\e \to 0$. We show that the limit functions are a weak solution of a macroscopic model with homogenised coefficients. The triple $(u_0,R_0,p_0)$ with the macroscopic concentration $u_0: (0,T)\times \Omega \rightarrow \R$, the macroscopic function of local radii $R_0:(0,T)\times \Omega \rightarrow [R_\mathrm{min},R_\mathrm{max}]$, and the Darcy pressure $p_0:(0,T)\times \Omega \rightarrow \R$ solves the following problem:
We define the evolving reference element and surface 
\begin{align}
    Y^{\ast}(t,x):= Y^*_{R_0(t,x)}=Y \setminus \overline{B_{R_0(t,x)}(\m)}, \qquad \Gamma(t,x):= \partial B_{R_0(t,x)} (\m),
\end{align} 
which allows to define the porosity $\theta$  via
\begin{align}\label{Def:Porosity}
    \theta(t,x):= | Y^{\ast}(t,x)|.
\end{align}
Between the porosity $\theta$ and the specific surface $\Gamma(t,x)$, we have the relation $\partial_t \theta(t,x) = |\Gamma(t,x)|  \partial_t R_0 $.
The transport equation describing the evolution of $u_0$ is given by 
\begin{subequations}\label{MacroModel}
\begin{align}
\label{MacroTransportPDE}
 \partial_t (\theta u_0) - \nabla \cdot \left( D^{\ast} \nabla u_0 - u_0 v^{\ast}\right) &= \theta f(u_0) + \partial_t \theta \rho &\mbox{ in }& (0,T)\times \Omega,
 \\
 \label{MacroTransportBC}
 u_0 &= 0 &\mbox{ on }& (0,T)\times \partial \Omega,
 \\
 \label{MacroTransportIC}
   (\theta u_0) (0) &= \theta^\mathrm{in} u^\mathrm{in} &\mbox{ in }& \Omega,
\end{align}
with the homogenised diffusion coefficient $D^{\ast}:(0,T)\times \Omega \rightarrow \R^{n\times n}$ defined in \eqref{Def:HomDiffCoeff_fixed} respectively $\eqref{Def:HomDiffCoeff}$, and the Darcy velocity $v^{\ast}:(0,T)\times \Omega \rightarrow \R^n$  given in $\eqref{eq:DarcyLawV}$.
The evolution of the local radii in every macroscopic point, $R_0$, is described by
\begin{align}
    \label{MacroODE:R0}
    \partial_t R_0 &= g(u_0,R_0) &\mbox{ in }& (0,T)\times \Omega,
    \\
    \label{MacroODE_IC}
    R_0(0) &= R^\mathrm{in} &\mbox{ in }& \Omega,
\end{align}
and the Darcy pressure $p_0:(0,T)\times \Omega \rightarrow \R$ is given by 
\begin{align}
\label{DarcyLawPDE}
    - \nabla \cdot \left(K^{\ast}( h_0 - \nabla p_0)\right) &= -\partial_t \theta &\mbox{ in }& (0,T)\times \Omega,
    \\
    \label{DarcyLawBC}
    p_0 &= p_\mathrm{b} &\mbox{ on }& (0,T)\times \partial \Omega,
\end{align}
with the permeability tensor $K^{\ast}:(0,T)\times \Omega \rightarrow \R^{n\times n}$  defined in $\eqref{Def:PermeabilityTensorFixed}$ respectively $\eqref{eq:K:moving_cell}$. We emphasise that following relation between the Darcy velocity $v^{\ast}$ and the Darcy pressure $p_0$ holds:
\begin{align}\label{eq:DarcyLawV}
    v^{\ast} = K^{\ast}(h_0 - \nabla p_0).
\end{align}
\end{subequations}
We note that, in our special situation where the evolving reference elements are perforated with balls, the homogenised diffusion coefficient $D^{\ast}$ and the permeability tensor $K^{\ast}$ are in fact scalar-valued, see Lemma \ref{Lem:Scalar_hom_coeff} in the appendix.

A weak solution of the macroscopic model $\eqref{MacroModel}$ is defined as follows:
\begin{definition}\label{Def:Weak_macro}
We call $(u_0,R_0,p_0)$ a weak solution of the macroscopic model $\eqref{MacroModel}$ if 
\begin{align*}
    u_0 &\in L^2((0,T),H^1_{\partial \Omega} (\Omega)) \quad \mbox{ with } \partial_t (\theta u_0) \in L^2((0,T),H^1_0(\Omega)'),
    \\
    R_0 &\in W^{1,\infty}((0,T),L^2(\Omega)) \cap L^{\infty}((0,T)\times \Omega)  \quad \mbox{ with } \partial_t R_0 \in L^{\infty}((0,T)\times \Omega),
    \\
    p_0  - p_\mathrm{b} &\in L^{\infty}((0,T),H^1_0(\Omega)),
\end{align*}
with $(\theta u_0) (0) = \theta^\mathrm{in} u^\mathrm{in}$ (see Remark \ref{rem:RegularityU0}) such that the following variational equations are fulfilled almost everywhere in $(0,T)$: For all $\phi \in H^1_0(\Omega)$, we have
\begin{align}\label{MacroEqu:VarTransport}
    \langle \partial_t (\theta u_0), \phi \rangle_{H^{-1}(\Omega),H_0^1(\Omega)} + \int_{\Omega} [ D^{\ast}\nabla u_0 - u_0 v^{\ast} ] \cdot \nabla \phi \,\mathrm{d}x = \int_{\Omega} (\theta f + \partial_t \theta \rho)  \phi \,\mathrm{d}x,
\end{align}
while for all $\xi \in H^1_0(\Omega)$, we have
\begin{align*}
   \int_{\Omega} K^{\ast} (h - \nabla p_0) \cdot \nabla \xi \,\mathrm{d}x = \int_{\Omega} -\partial_t \theta \xi \,\mathrm{d}x.
\end{align*}
Furthermore,
\begin{align}\label{eq:limitODE}
    \partial_t R_0 = g(u_0,R_0)
\end{align}
almost everywhere in $(0,T)$.

\end{definition}

\begin{remark}\label{rem:RegularityU0}
We immediately obtain that a weak solution in the sense of Definition \ref{Def:Weak_macro} fulfills $\theta u_0 \in C^0([0,T],L^2(\Omega))$ and, therefore, the initial condition $(\theta u_0) (0) = \theta^\mathrm{in} u^\mathrm{in}$ makes sense.  For $R_0^\mathrm{in} \in H^1(\Omega)$, the spatial regularity of $u_0$ and $R_0^\mathrm{in}$ can be transfered onto the ODE \eqref{eq:StrongFormR:Transformed} and we obtain 
$R_0 \in C([0,T];H^1(\Omega)\cap L^\infty(\Omega))$ and, thus, ${\theta, \theta^{-1} \in C([0,T];H^1(\Omega) \cap L^\infty(\Omega))}$.
Moreover, the uniform $L^\infty((0,T)\times \Omega)$ bound of $u_\e$ yields  $u_0 \in L^\infty((0,T) \times \Omega)$ and, thus, $\nabla (\theta u_0) = \nabla \theta u_0 + \theta \nabla u_0 \in L^2((0,T) \times \Omega)$. Then, we can conclude $\theta u_0 \in C([0,T];L^2(\Omega))$ and, thus, $u_0 = \theta^{-1}\theta u_0  \in C([0,T];L^2(\Omega))$.
Then, the initial condition \eqref{MacroTransportIC} can be replaced by $u_0(0) = u^\mathrm{in}$.
\end{remark}

The second main result of the paper is the convergence of the microscopic solutions to solutions of the macroscopic model just introduced.

\begin{theorem}\label{MainTheoremMacroModel}
Let $(\ueps,\weps, \peps,\reps)$ be the unique weak solution of the microscopic problem $\eqref{MicroModelFixedDomain}$. Then, up to a subsequence, these sequences converge to the limit functions $(u_0,w_0,p_0,R_0)$  in the sense of Proposition \ref{prop:Comapctness}, which is a weak solution of the macroscopic model $\eqref{MacroModel}$ in the sense of Definition \ref{Def:Weak_macro}
\end{theorem}

\section{Existence and uniqueness for the microscopic model} \label{sec:Micromodel:Existence+Uniqueness}

In this section, we show the existence and uniqueness of weak solution $(\ueps,\weps,\peps,\reps)$ of the microscopic model \eqref{eq:weakForm:u_eps}--\eqref{eq:weak_w_eps_fixed_domain}. In order to show the existence of a solution, we employ a fixed-point argument. Thereby, we show not only the continuity of the operator, which is necessary in order to show the existence of a solution, but also its Lipschitz continuity. Thus, we can conclude the uniqueness of the fixed point and, thereby, the uniqueness of the solution of \eqref{eq:weakForm:u_eps}--\eqref{eq:weak_w_eps_fixed_domain}.

\subsection{Construction of the fixed-point operator and its properties}
In order to show the existence of a weak solution of the microscopic model $\eqref{MicroModelFixedDomain}$ in the sense of Definition \ref{def:DefinitionWeakSolutionMicroModel}, we apply a fixed-point argument as follows. 
For $\beta \in \left(\frac12 ,1\right) $, we define the operator  $\mathcal{F}_1: L^2((0,T) , H^{\beta}(\oe) )\rightarrow L^{\infty}((0,T)\times \oe) \cap H^1((0,T),L^{\infty}(\oe))$  by  $\mathcal{F}_1(\hueps) := \reps$, where $\reps $ solves the problem
\begin{subequations}\label{ODELinearization}
\begin{align}\label{eq:def:F1} 
    \partial_t \reps &=\e \fint_{\geps\left(\e \bfxe\right)} g\left(\hueps(t,z),\frac{\reps(t,z)}{\e}\right) \,\mathrm{d}\sigma_z    &\mbox{ in }& (0,T)\times \oe,
    \\\label{eq:def:F1_initial} 
    \reps(0) &= \reps^\mathrm{in} &\mbox{ in }& \oe.
\end{align}
\end{subequations}
In Lemma \ref{LemmaF1welldefined}, we will see that the operator is well-defined. In particular, $\reps$ is constant on every microcell $\e (Y^{\ast} + k)$ for $k \in K_{\e}$ with values in $[R_\mathrm{min},R_\mathrm{max}]$, and the time derivative is essentially bounded. 

For the linearization of the Stokes problem, we first introduce the set 
\begin{align*}
    Y_{\e}:= \big\{ \reps& \in L^{\infty}((0,T)\times \oe) \cap H^1((0,T),L^{\infty}(\oe)) \, : 
    \\
    &\|\partial_t \reps\|_{L^{\infty}((0,T)\times \Oe)} \le \e C_g, \reps|_{\e(Y^{\ast} + k)} \mbox{ constant for every } k \in K_{\e}, 
    \\
    &\e^{-1} \reps(t,x) \in [R_\mathrm{min},R_\mathrm{max}] \mbox{ for almost every } (t,x) \in (0,T)\times \oe\big\}.
\end{align*}
Note that solutions of $\eqref{ODELinearization}$ are actually elements of $Y_{\e}$ as shown in Lemma \ref{LemmaF1welldefined} below. Now,  we define $\mathcal{F}_2: Y_\e \rightarrow L^2((0,T);H^1_\Ge(\Oe)^n)$  as the solution of the linearised Stokes problem, i.e.\ $\mathcal{F}_2(\hreps):= \weps$, where  $\weps$ is the unique weak solution of 
\begin{align}\label{eq:def:F2}
\begin{aligned}
    -\e^2 \div(\hAeps \heeps(w_\e)) + \hAeps^\top \nabla \qeps  &= \hJeps\hheps +\e^2 \div( \heeps(\partial_t \psieps)) - \hAeps^\top \nabla p_\mathrm{b} &\mbox{ in }& (0,T)\times \oe,
    \\
    \nabla \cdot (\hAeps \veps ) &= - \nabla \cdot (\hAeps \partial_t \hpsieps )  &\mbox{ in }& (0,T)\times \oe,
    \\
    \veps &= -\partial_t \hReps \nu &\mbox{ on }& (0,T)\times \geps,
    \\
    \left(-\e^2 \heeps(w_\e)  +  \qeps I    \right) \hAeps^\top\nu &=  \e^2 \heeps(\partial_t \hpsieps) \hAeps^\top\nu  &\mbox{ on }& (0,T)\times \partial \Omega,
\end{aligned}
\end{align}
where $\heeps, \hAeps, \hFeps, \hJeps$ and $\widehat{h}_\e$ are defined in the same way as $e_\e, \Aeps, \Feps, \Jeps$ and $h_\e$ but with $\hReps$ instead of $\Reps$.
We emphasise that the operator is well defined since $\hreps$ is constant on every micro-cell and, therefore, the trace of $\partial_t \hreps$ on $\geps$ exists and is an element of $L^{\infty}((0,T)\times \geps)$.
Moreover, we note that the velocity fields which solve  \eqref{eq:def:F2} are elements of $Z_\e$, where
\begin{align}
\begin{aligned}\label{Def:Z_eps}
Z_\e \coloneqq \{(\reps , \weps ) &\in Y_{\e} \times L^2((0,T),H^1_{\geps}(\oe)^n)) \, : 
\\
&\, \nabla \cdot (A_{\e} (\weps + \partial_t \psieps)) = 0 , \|\weps\|_{L^\infty((0,T);H^1_{\geps}(\oe)^n))}  \leq C_h\}
\end{aligned}
\end{align}
where $C_h$ depends on $C_g$ and $h$. The existence of such an upper bound $C_h$ is given by Lemma \ref{lem:Stokes} below.

Finally, we define the operator 
\begin{align*}
  \mathcal{F}_3:  Z_\e  \times L^2((0,T),H^{\beta}(\Oe)) \rightarrow L^2((0,T),H^{\beta}(\Oe))
\end{align*}
 via $\mathcal{F}_3(\hreps,\hweps,\hueps):= \ueps$, where $\ueps$ is a weak solution of
\begin{align}\label{eq:def:F3}
\begin{aligned}
    \partial_t (\hjeps \ueps) - \nabla \cdot \left( \hDeps \nabla \ueps  - \ueps \hAeps \hweps  \right) &= \hjeps f(\hueps) &\mbox{ in }& (0,T)\times \oe,
    \\
    -(\hDeps \nabla \ueps - \ueps\hAeps \hweps) \cdot \nu &= \rho\hjeps g(\hueps,\e^{-1}\hreps)  &\mbox{ on }& (0,T)\times \geps,
    \\
    \ueps &= 0 &\mbox{ on }& (0,T)\times  \partial \Omega,
    \\
   \ueps(0)&= \uepsin &\mbox{ in }& \oe,
\end{aligned}
\end{align}
which is analogously defined as in \eqref{eq:weakForm:u_eps}.
Here, the data $\hjeps$, $\widehat{D}_{\e}$ and $\hweps$ are again defined in the same way as $\jeps$, $D_{\e}$ and   $\weps$ with $\hReps$ instead of $\Reps$.  Now, we are able to define the fixed-point operator for the whole system:
\begin{align*}
    \mathcal{F}: L^2((0,T),H^{\beta}(\oe)) \rightarrow L^2((0,T),H^{\beta}(\oe)),
\qquad \mathcal{F}(\hueps) := \mathcal{F}_3\left(\mathcal{F}_1(\hueps),\mathcal{F}_2\left(\mathcal{F}_1(\hueps)\right),\hueps\right).
\end{align*}

In what follows, we show that the operators $\mathcal{F}_i$, $i=1,2,3$, are well defined, and, therefore, so is the fixed-point operator $\mathcal{F}$.

\begin{lemma}\label{LemmaF1welldefined}
The operator $\mathcal{F}_1$ is well defined and the range of $\mathcal{F}_1$ is in $Y_{\e}$.
\end{lemma}
\begin{proof}
Let $\hueps \in L^2((0,T);H^\beta(\Oe))$. Due to the trace embedding $H^{\beta}(\oe)\hookrightarrow L^2(\geps)$,  the right-hand side of \eqref{eq:def:F1} is well defined. Then, for every $x \in \e(k +Y)$ (with $k \in K_{\e}$), the mapping 
\begin{align*}
  (t, r) \mapsto   \e \fint_{\geps\left(\e \bfxe\right)} g\left(\hueps(t,z),\e^{-1} r \right) \,\mathrm{d}\sigma_z
\end{align*}
is constant with respect to $x$, globally Lipschitz continuous with respect to $r$ and measurable with respect to $t$. Hence,  Carath\'eodory’s existence theorem yields the existence and uniqueness of a solution $\reps(\cdot_t, x) \in W^{1,1}(0,T)$ of \eqref{eq:def:F1}--\eqref{eq:def:F1_initial} for every $x \in \Oe$, and $\reps(\cdot_t,x)$ is constant with respect to $x$ on every microcell $\e(k + Y)$ (here, we also used that $\reps^\mathrm{in}$ is constant on every microcell). Then, we immediately obtain together with the essential boundedness of $g$ that $\reps \in L^{\infty}((0,T)\times \oe) \cap H^1((0,T),L^2(\oe))$ and also $\partial_t  \reps \in L^{\infty}((0,T) \times \oe)$. More precisely, we have for almost every $(t,x) \in (0,T)\times \oe$
\begin{align}\label{eq:EssentBound:dtR}
    |\partial_t \reps(t,x) | \le \e \|g\|_{L^{\infty}(\R^2)} \le \e C_g.
\end{align}
Moreover, \eqref{eq:g_Bound_below}--\eqref{eq:g_Bound_above} ensure $\e^{-1} \reps(t,x) \in [\rmin, \rmax]$ for every $t \in (0,T)$ and almost every $x \in \Oe$. Note also that, for every $t \in (0,T)$, the range $\reps(t,\cdot)$ lies in a finite-dimensional subspace of $L^{\infty}(\oe)$ and, therefore, $\reps$ is measurable as a function from $(0,T)$ to $L^{\infty}(\oe)$.
\end{proof}

The following Lemma provides a kind of Lipschitz estimate for the operator $\mathcal{F}_1$.

\begin{lemma}\label{lem:LipschitzContnuityF1}
Let $\ueps^j \in L^2((0,T) \times \geps)$, $j=1,2$, and $\reps^j $ be given by
\eqref{eq:StrongFormR:Transformed} for $\ueps = \ueps^j$ (not necessaryly with identical initial conditions).
Then, if $\reps^1(0) = \reps^2(0)$ it holds for almost every $t \in (0,T)$ that 
\begin{align*}
    \|\reps^1 - \reps^2\|_{L^{\infty}((0,t) \times \Omega)} &\le C_\e \|\ueps^1 - \ueps^2 \|_{L^1((0,t)\times \geps)},
    \\
    \|\partial_t \reps^1 -  \partial_t\reps^2\|_{L^2((0,t),L^{\infty}(\Omega))} &\le C_{\e}\|\ueps^1 - \ueps^2\|_{L^2((0,t),L^1(\geps))}
\end{align*}
for a constant $C_{\e}$ (depending on $\e$). In particular, for $\ueps^j \in L^2((0,T),H^1(\oe))$, we obtain for every $\theta >0$ the existence of a constant $C(\theta)>0$ such that the right-hand sides in the inequalities above can be further estimated by
\begin{align}\label{eq:TraceInequality}
    \e \|\ueps^1 - \ueps^2\|_{L^2((0,t),L^1(\geps))} \leq C(\theta) \|\ueps^1 - \ueps^2\|_{L^2((0,t) \times \oe)} + \theta \|\nabla \ueps^1 - \nabla \ueps^2 \|_{L^2((0,t)\times \oe)}.
\end{align}
We also have the $\e$-uniform Lipschitz estimate 
\begin{align}
\begin{aligned}\label{Lip_Estimate_reps}
   \e^{-1} \|\reps^1 - \reps^2& \|_{L^{\infty}((0,T),L^2(\Omega))} 
   \\
   &\le C \left(\sqrt{\e} \|\ueps^1 - \ueps^2\|_{L^2((0,T)\times \geps)} + \e^{-1} \| \reps^1(0) - \reps^2(0)\|_{L^2(\Omega)}\right)
\end{aligned}
\end{align}
for a constant $C>0$ independent of $\e$, which implies that $\mathcal{F}_1$ is Lipschitz continuous.
\end{lemma}
\begin{proof}
A proof can be found in \cite[Lemma 14]{GahnPop2023Mineral} and in the proof of \cite[Proposition 4]{GahnPop2023Mineral}. 
\end{proof}
Next, we analyze the operator $\mathcal{F}_2$ and show that it is Lipschitz continuous in particular. 

\begin{lemma}\label{lem:Stokes}
The operator $\mathcal{F}_2$ is well defined and $(\Reps,\mathcal{F}_2(\Reps)) \in Z_\e$ for all $\Reps \in Y_\e$. 
\end{lemma}
\begin{proof}
In \cite[Theorem3.1]{wiedemann2021homogenisation}, the existence of a unique solution $(w_\e(t), q_\e(t)) \in H^1_{\geps}(\Oe)^n \times L^2(\Oe)$ of \eqref{eq:def:F2} for a.e.~$t\in (0,T)$ was shown under the assumptions \eqref{eq:psi-id}--\eqref{eq:J>c_J} on $\hpsieps(t)$ and the assumptions  $\hheps(t) \in L^2(\Omega)^n$ and $\partial_t \psieps(t) \in H^1(\Oe)^n$. Moreover, this pointwise argumentation was transferred to the existence and uniqueness of a solution $(w_\e, q_\e) \in L^\infty((0,t);H^1_{\geps}(\Oe)^n) \times L^\infty((0,T);L^2(\Oe))$ of \eqref{eq:def:F2}, where the $L^\infty(0,T)$-boundedness of the solution arises from the  $L^\infty(0,T)$-boundedness of $\hheps$ (see \ref{ass:h}) and $\partial_t \psieps$, which is given by Lemma \ref{lem:EstimatesEpsTrafo}.

In order to conclude $(\Reps,\mathcal{F}_2(\Reps)) \in Z_\e$, we note that the divergence condition is obviously fulfilled. The existence of an upper bound $C_h$ such that $\|\mathcal{F}_2(\Reps) \|_{L^\infty((0,T);H^1_{\geps}(\Oe)^n)} \leq  C_h$ for all $\Reps \in Y_\e$ is shown in \cite[Theorem3.1]{wiedemann2021homogenisation} as well.
\end{proof}

In what follows, for two given functions $\phi^1$ and $\phi^2$, we denote its difference by
\begin{align*}
\delta \phi := \phi^1  - \phi^2.
\end{align*}

\begin{lemma}\label{lem:Lip_L2}
The operator $\mathcal{F}_2$ is locally Lipschitz continuous. More precisely, for every  ball $B \subset Y_\e$, there exists a constant $L_{B,\e}>0$ such that 
\begin{align}
    \begin{aligned}\label{eq:Lip:F2_L_infinity}
    \big\|\mathcal{F}_2(\hreps^1) - \mathcal{F}_2(\hreps^1)\big\|_{L^2(\tOe)} + \norm{\nabla (\mathcal{F}_2(\hreps^1) - \mathcal{F}_2(\hreps^1))}{L^2(\tOe)} \\
\leq L_{B,\e} \left(\norm{\hreps^1 - \hreps^2}{L^{\infty}((0,t),L^2(\Oe))} +    \norm{\partial_t \hreps^1 - \partial_t \hreps^2}{L^2( \tOe)} \right)
    \end{aligned}
\end{align}
for all $\hReps^1, \hReps^2 \in B$ and every $t \in (0,T)$.
\end{lemma}
\begin{proof}
We proof this Lemma by applying the generic Lemma \ref{lem:SaddlePointLpschitzContinuity}.
Within this proof, we use the shorthand notation $V = H^1_{\Ge}(\Oe)$ with the norm $\norm{v}{V} \coloneqq \norm{\nabla v}{L^2(\Oe)}$ and $Q = L^2(\Oe)$. Then, we define almost everywhere in $(0,T)$ 
\begin{align*}
a_{\hReps}(v,w) &:= \tfrac{1}{2}\e^2 (\widehat{e}_{\e}(v), \hAeps^\top \nabla w)_{Q} = \tfrac{1}{2}\e^2 (\hFeps^\top \nabla v + (\hFeps^\top \nabla v)^\top, \hAeps^\top \nabla w)_{Q},
\\
b_{\hReps}(v,q) &:= - (q, \div(\hAeps  v))_{Q} = - (q,\hAeps : \nabla v)_Q ,
\\
f_{\hReps}(v) &:= (\hjeps h_\e, v)_{Q}-a_{\hReps}(\partial_t \hpsieps,v) - (\hAeps  \nabla p_\mathrm{b},  v)_{Q} ,
\\
g_{\hReps }(q) &:= b(\partial_t\hpsieps , q ) =
(q , \div(\hAeps  \partial_t \hpsieps ))_{Q} = (q , \partial_t \hJeps)_{Q}
\end{align*}
for $v,w \in V$ and $q \in Q$.
We note that $\hReps^1, \hReps^2 \in Y_\e$ and, thus, we obtain with the H\"older inequality, the bounds for $\hFeps$ and $\Aeps$ from Lemma \ref{lem:EstimatesEpsTrafo} and the Lipschitz estimates of Lemma  \ref{lem:EstimateShiftsReps}
\begin{align}\label{eq:Lip:BilA}
\begin{aligned}
 |a_{\hReps^1 }(v,w) -a_{\hReps^2 }(v,w)|
=&
\tfrac{1}{2}\e^2 (\delta \hFeps  \nabla v + (\delta \hFeps   \nabla v)^\top, \hAeps^1  \nabla w)_{Q}
\\
&+
\tfrac{1}{2}\e^2 (\hFeps^{2\top} \nabla v + ( \hFeps^{2 \top}  \nabla v)^\top, \delta\hAeps^\top\nabla w)_{Q}
\\
\leq& C_\e
\left(\norm{\delta \hFeps}{L^\infty(\Oe)}
+
\norm{\delta \hAeps}{L^\infty(\Oe)}
\right) \norm{v}{V} \norm{w}{V}.
\\
\leq& C_\e \|\delta\hReps\|_{L^2(\Oe)}.
\end{aligned}
\end{align}
By employing the boundedness of $\hAeps$ and $\partial \Aeps$ from Lemma \ref{lem:EstimatesEpsTrafo},
\begin{align}
\begin{aligned}\label{eq:Lip:BilB}
    |b_{\hReps^1}(v,q) -b_{\hReps^2}(v,q)| &= | (q , \delta \hAeps : \nabla v)_Q| \le \|\delta \hAeps\|_{L^{\infty}(\Oe)} \|q\|_Q  \|v\|_{V}
    \\
    &\le C_{\e} \|\delta \reps \|_{L^{\infty}(\Oe)} \|q\|_Q \|v\|_V
\end{aligned}
\end{align}
for all $v \in V$ and $q \in Q$.

Taking into account the H\"older inequality, the Lipschitz continuity of $h$, the boundedness of $\hAeps$,  $\hFeps$ and $\hheps^1$,
\eqref{eq:Lip:BilA}, the Poincar\'e inequality for $H^1_{\Ge}(\Oe)$ and Lemma \ref{lem:EstimatesEpsTrafo}, we obtain
\begin{align*}
|f_{\hReps^1}(v) -  f_{\hReps^2}(v)| \leq& |(\delta\hjeps \widehat{h}_\e, v)_{L^2(\tOe)}|
+ |a_{\hReps^1}(\delta \partial_t \hpsieps,v)| 
\\
&+ |a_{\hReps^1}(\partial_t \hpsieps^2,v)- a_{\hReps^2}(\partial_t \hpsieps^2,v)| | 
+
| (\delta \hAeps  \nabla p_\mathrm{b},  v)_{Q}|
\\
\leq& 
C \norm{\delta \widehat{J}_{\e}}{L^\infty(\Oe)}\norm{v}{Q} 
+
C \norm{\delta \partial_t \hpsieps}{V} \norm{v}{V}
\\
&+
C_\e \|\delta\hReps\|_{L^2(\Oe)} \norm{\partial_t \hpsieps^2}{V}
\norm{v}{V}
+
C_\e \norm{\delta\hAeps}{L^\infty(\Oe)} \norm{v}{Q}
\\
\le& C_{\e} \left[ \left(1 + \|\partial_t \hReps^2 \|_{L^2(\Oe)} \right)\|\delta \hReps\|_{L^{\infty}(\Oe)} + \|\delta \partial_t \hReps\|_Q \right] \|v\|_V.
\end{align*}
Employing Lemma  \ref{lem:EstimatesEpsTrafo}, we obtain
\begin{align*}
    |g_{\hReps^1}(q) -g_{\hReps^2}(q) | &= |(q, \delta \partial_t \hjeps )_Q|\le \|q\|_Q \|\delta \partial_t \jeps \|_Q
    \le C_{\e} \| \delta \partial_t \hReps\|_Q \|q\|_Q.
\end{align*}
Thus, we obtain overall
\begin{align}\label{eq:LipDataF2}
\begin{aligned}
&\norm{a_{\hReps^1} -a_{\hReps^2}}{\Bil}
+
\norm{b_{\hReps^1} -b_{\hReps^2}}{\Bil}
\leq
C_\e \norm{\hReps^1 -\hReps^2}{L^2(\Oe)}
\\
&\norm{f_{\hReps^1} -f_{\hReps^2}}{V'}
+
\norm{g_{\hReps^1} -g_{\hReps^2}}{Q'}
\leq
C_\e M
\norm{\hReps^1 -\hReps^2}{L^2(\Oe)}
+ C_\e \norm{\partial_t (\hReps^1 -\hReps^2)}{L^2(\Oe)}
\end{aligned}
\end{align}
for $M = 1 + \|\partial_t \hReps^1\|_{L^2(\Oe)}+\|\partial_t \hReps^2\|_{L^2(\Oe)}$.
By similar estimates, it can be easily observed that 
\begin{align*}
&\norm{a_{\hReps}}{\Bil} + \norm{b_{\hReps}}{\Bil} \leq C_\e,
\\
&\norm{f_{\hReps}}{V'}+ \norm{g_{\hReps}}{Q'} \leq C_\e M
\end{align*}
for $\hreps^1, \hreps^2 \in B$. 
Moreover, \eqref{eq:Sym-coercivityFeps} and \eqref{eq:Inf-Sup} yield constants $\alpha_0, \beta_0 >0$ such that 
\begin{align*}
a_{\hReps}(v,v) \geq \alpha_0 \norm{v}{V}^2
\end{align*}
for all $v \in V$ and 
\begin{align*}
\inf\limits_{q \in Q} \sup\limits_{v \in V} \frac{b_{\hReps}(v,q)}{\norm{v}{V}\norm{q}{Q}} \geq \beta_0.
\end{align*}
Thus, $a_{\hReps}, b_{\hReps}, f_{\hReps}$ and $g_{\hReps}$ fulfill the assumptions of Lemma \ref{lem:SaddlePointLpschitzContinuity} and we can transfer the Lipschitz continuity of the data (see \eqref{eq:LipDataF2}) onto the solution which yields almost everywhere in $(0,T)$
\begin{align*}
\big\|\mathcal{F}_2(\hreps^1) - \mathcal{F}_2(\hreps^1)&\big\|_{L^2(\Oe)} + \norm{\nabla (\mathcal{F}_2(\hreps^1) - \mathcal{F}_2(\hreps^1))}{L^2(\Oe)} 
\\
&\leq
L  C_\e M \norm{\hReps^1 -\hReps^2}{L^2(\Oe)} + C_\e \norm{\partial_t (\hReps^1 -\hReps^2)}{L^2(\Oe)} .
\end{align*}
Integrating with respect to time from $0$ to $t \in (0,T)$ and employing the H\"older inequality yields
\begin{align*}
\|\mathcal{F}_2(\hreps^1)& - \mathcal{F}_2(\hreps^1)\|_{L^2((0,t) \times \Oe)} + \norm{\nabla (\mathcal{F}_2(\hreps^1) - \mathcal{F}_2(\hreps^1))}{L^2((0,T) \times \Oe)}  
\\
&\leq
C_\e \norm{M}{L^2(0,T)}
\norm{\hReps^1 -\hReps^2}{L^\infty((0,T);L^2(\Oe))}
+
C_\e \norm{\partial_t (\hReps^1 -\hReps^2)}{L^2((0,T) \times \Oe)} 
\end{align*}
and, thus, \eqref{eq:Lip:F2_L_infinity}.
We note that the proof also holds for more general $\Reps^1,\Reps^2$, even if they do not fulfil the bounds $\partial_t\Reps^1, \partial_t\Reps^1  \in L^\infty((0,T)\times \Oe)$.
\end{proof}

\begin{lemma}\label{LemmaF3}
The operator $\mathcal{F}_3$ is well defined and continuous.
\end{lemma}
\begin{proof}
Existence follows by the Galerkin method and similar a priori estimates as in Lemma \ref{LemmaAprioriEstimates} in Section \ref{SectionAprioriEstimates}.

For the continuity of the operator, we consider sequences $(\hreps^k,\hweps^k) \in Z_{\e}$, see $\eqref{Def:Z_eps}$, and $\hueps^k \in L^2((0,T),H^{\beta}(\Oe))$, such that for $k\to \infty$
\begin{align*}
    \hreps^k &\rightarrow \hreps &\mbox{ in }& L^{\infty}((0,T)\times \oe) \cap H^1((0,T),L^{\infty}(\oe)),
    \\
    \hweps^k &\rightarrow \hweps &\mbox{ in }& L^2((0,T),H^1_{\geps}(\oe)^n),
    \\
    \hueps^k &\rightarrow \hueps &\mbox{ in }& L^2((0,T),H^{\beta}(\Oe)).
\end{align*}
From Lemma \ref{lem:EstimateShiftsReps}, we also get the strong convergence of $\hjeps^k$, $\hDeps^k$, and $\hAeps^k$ in $L^{\infty}((0,T)\times \oe)$ to $\hjeps$, $\hDeps$, and $\hAeps$.
From the a priori estimates in Lemmas \ref{LemmaAprioriEstimates} and \ref{LemmaLinftyBoundUeps} below, we obtain for $\ueps^k := \mathcal{F}_3(\hreps^k,\hweps^k)$ and a constant $C_{\e}>0$ (depending on $\e$) the estimate
\begin{align*}
  \|\partial_t \ueps^k\|_{L^2((0,T),H_{\partial \Omega}^1(\Oe)')} +  \|\ueps^k\|_{L^2((0,T),H^1(\oe))}  \le C_{\e}.
\end{align*}
Hence, there exists $\ueps \in L^2((0,T),H^1_{\partial\Omega}(\oe))\cap H^1((0,T),H^1_{\partial \Omega}( \oe)')$ such that, up to a subsequence, $\ueps^k$ converges weakly to $\ueps$ in $L^2((0,T),H^1(\oe))$ and strongly in $L^2((0,T),H^{\beta}( \Oe))$ by the Aubin--Lions lemma and, by continuity of the trace operator, also in the space $L^2((0,T)\times \geps)$. Further, $\partial_t \ueps^k$ converges weakly to $\partial_t \ueps$ in $L^2((0,T),H^1_{\partial \Omega}(\Oe)')$.
Testing the weak formulation of the linearised problem $\eqref{eq:def:F3}$ with $\phi(x) \psi(t)$ with $\phi \in H^1_{\partial \Omega}(\oe)$ and $\psi \in C_0^{\infty}([0,T))$, we get
\begin{align*}
    - \int_0^T&\langle \hjeps^k \ueps^k ,\phi \rangle_{H^1_{\partial \Omega}(\Oe)',H^1_{\partial \Omega}(\Oe)} \psi \,\mathrm{d}x \,\mathrm{d}t + \int_0^T \int_{\oe } [\hDeps^k \nabla \ueps^k - \ueps^k \hAeps^k \hweps^k ] \cdot \nabla \phi \psi \,\mathrm{d}x \,\mathrm{d}t 
    \\
    &= \int_0^T \int_{\oe } \hjeps^k f(\hueps^k) \phi \psi \,\mathrm{d}x \,\mathrm{d}t - \rho \int_0^T \int_{\geps } \hjeps^k g\left(\hueps^k,\e^{-1} \hreps^k\right) \phi \psi \,\mathrm{d}\sigma \,\mathrm{d}t\\ & \quad + \int_{\oe }\hjeps^k(0) \uepsin \phi \psi(0) \,\mathrm{d}x.
\end{align*}
Passing to the limit $k \to \infty $, we obtain
\begin{align*}
     - \int_0^T& \langle \hjeps\ueps ,\phi \rangle_{H^1_{\partial \Omega}(\Oe)',H^1_{\partial \Omega}(\Oe)} \psi' \,\mathrm{d}x \,\mathrm{d}t + \int_0^T \int_{\oe } [\hDeps \nabla \ueps - \ueps \hAeps \hweps ] \cdot \nabla \phi \psi \,\mathrm{d}x \,\mathrm{d}t 
    \\
    &= \int_0^T   \int_{\Oe} \hjeps f(\hueps)\phi\psi \,\mathrm{d}x \,\mathrm{d}t - \rho \int_0^T \int_{\geps } \hjeps g\left(\hueps,\e^{-1} \hreps\right) \phi \psi \,\mathrm{d}\sigma \,\mathrm{d}t\\ & \quad + \int_{\oe }\hjeps(0) \uepsin \phi \psi(0) \,\mathrm{d}x.
\end{align*}
By density, we obtain $\mathcal{F}_3(\hreps,\hweps,\hueps) = \ueps$. The initial condition $\ueps(0) = \ueps^\mathrm{in}$ can be obtained by a similar argument as above and using an integration by parts in time. This implies the continuity of $\mathcal{F}_3$.
\end{proof}

\begin{remark}\label{RemarkGalerkinMethod}\
   The operator $\mathcal{F}_3$ is even Lipschitz continuous. This can be shown by similar arguments as in the proof of \cite[Theorem 3]{GahnPop2023Mineral} and the proof of Proposition \ref{prop:uniqueness} below. In both proofs, an estimate for a fixed point of $\mathcal{F}$ are given. However, it is straightforward to generalise these arguments to show the Lipschitz continuity. Here, we restrict to the continuity of $\mathcal{F}_3$, which is enough to guarantee the existence of a solution. For uniqueness of the weak solution, we can use the results from the proof of \cite[Theorem 3]{GahnPop2023Mineral} and, therefore, avoid repeating the same arguments, which shortens the uniqueness proof.
\end{remark}

\subsection{Existence of a solution for the micoscopic model}
\begin{proposition}\label{TheoremExistenceMicroModel}
There exists a weak solution $(\ueps,\reps,\weps,\peps)$ of the microscopic problem $\eqref{MicroModelFixedDomain}$.
\end{proposition}
\begin{proof}
We apply Sch\"afer's fixed-point theorem (see e.g.\ \cite{EvansPartialDifferentialEquations}) to the operator $\mathcal{F}$. For this, we have to show that $\mathcal{F}$ is compact and continuous and that the set
\begin{align*}
    \mathcal{X}:= \left\{ \ueps \in L^2((0,T),H^{\beta}(\oe))\, : \, \ueps = \lambda \mathcal{F}(\ueps) \, \mbox{ for some } 0 \le \lambda \le 1 \right\}
\end{align*}
is bounded. Let $(\hueps^k)_{k\in \N} \subset L^2((0,T),H^{\beta}(\oe))$ be bounded and $\ueps^k := \mathcal{F}(\hueps^k)$. By similar arguments as in the proof of Lemma \ref{LemmaAprioriEstimates} in Section \ref{SectionAprioriEstimates} below, we can show that
\begin{align*}
    \|\partial_t \ueps^k\|_{L^2((0,T),H_{\partial\Omega}^1(\oe)')} + \|\ueps^k\|_{L^2((0,T),H^1(\oe))} \le C_{\e}
\end{align*}
for a constant $C_{\e}>0$ independent of $k$ (but depending on $\e$). Hence, the Aubin--Lions lemma, see \cite{Lions}, implies the existence of a convergent subsequence in $L^2((0,T),H^{\beta}(\oe))$ and, therefore, $\mathcal{F}$ is compact.

The continuity of the operator $\mathcal{F}$ follows from the continuity of the operators $\mathcal{F}_i$, $i=1,2,3$, see Lemma \ref{LemmaF1welldefined}, \ref{lem:Lip_L2}, and \ref{LemmaF3}.

It remains to show the boundedness of $\mathcal{X}$. We only have to consider $\ueps = \lambda \mathcal{F}(\ueps)$ for $\ueps \in L^2((0,T),H^{\beta}(\oe))$ and $\lambda \in (0,1]$. Hence, we have to replace $\ueps$ in the microscopic model $\eqref{MicroModelFixedDomain}$ with $\frac{\ueps}{\lambda}$. Then, the claim follows again by using similar arguments as in the proof of Lemma \ref{LemmaAprioriEstimates} by multiplying the right-hand side of the weak equation $\eqref{eq:weakForm:u_eps}$ by $\lambda$. Now, Sch\"afer's fixed-point theorem implies the existence of a fixed point and, therefore, the existence of a microscopic solution.
\end{proof}
\subsection{Uniqueness of a solution for the micoscopic model}

Now, we show uniqueness of the weak microscopic solution. In \cite[Theorem 3]{GahnPop2023Mineral}, uniqueness was established for the problem without convection. Hence, here we generalise this result, where the crucial new ingredient is the Lipschitz estimate from Lemma \ref{lem:Lip_L2}.

\begin{proposition}\label{prop:uniqueness}
The weak solution $(\ueps,\reps,\weps,\peps)$ of the microscopic problem $\eqref{MicroModelFixedDomain}$ obtained by Proposition \ref{TheoremExistenceMicroModel} is unique.
\end{proposition}
\begin{proof}

Consider two microscopic solutions $(\ueps^j,\reps^j,\weps^j,\peps^j)$ for $j=1,2$. Since every problem associated to the operators $\mathcal{F}_i$, $i=1,2,3$, admits a unique solution, it follows that $\ueps^j$ are fixed points of $\mathcal{F}$. Hence, we have $\ueps^j = \mathcal{F}(\ueps^j)$, $\reps^j = \mathcal{F}_1(\ueps^j)$ and $\weps^j= \mathcal{F}_2(\reps^j)$. The triple $(\ueps,\weps,\reps)$ fulfills the a priori estimates of Section \ref{SectionAprioriEstimates}. This can be shown by the same arguments as in the proofs of Lemmas \ref{LemmaAprioriEstimates} and \ref{LemmaLinftyBoundUeps}. In particular, the estimates from Lemmas \ref{lem:EstimatesEpsTrafo}, \ref{lem:EstimateShiftsReps} and \ref{lem:LipschitzContnuityF1} are valid, as well as the inequality $\eqref{eq:Lip:F2_L_infinity} $.

Now, as in the proof of \cite[Theorem 3]{GahnPop2023Mineral} (see the last inequality in the proof in particular) with only slight modifications for the nonlinear right-hand sides with $f$ and $g$ in the transport equation, we obtain, by subtracting the weak equations for $\ueps^1$ and $\ueps^2$ and testing with $\delta \ueps$, for almost every $t \in (0,T)$ and a constant $c>0$ 
\begin{align}
\begin{aligned}\label{eq:lipschitz_estimate}
    \|\delta \ueps(t)\|_{L^2(\oe)} + c\|\delta & \nabla \ueps\|_{L^2((0,t)\times \oe)}  
    \\
    &\le C_{\e}  \|\delta \ueps\|_{L^2((0,t)\times \oe)} + \left( \delta(\ueps A_{\e} \weps), \nabla \delta \ueps \right)_{(0,t)\times \oe},
\end{aligned}
\end{align}
where $C_{\e}>0$ is a constant depending on $\e$ ($C_{\e}$ is a generic constant in what follows) and $A_{\e}^j$ is defined as in $\eqref{def:Aeps}$.  It remains to estimate the convective term on the right-hand side:
\begin{align*}
    \big( \delta&(\ueps A_{\e}  \weps), \nabla \delta \ueps \big)_{(0,t)\times \oe} 
    \\
    &= \left(\ueps^1 A_{\e}^1 \delta \weps,\nabla \delta \ueps\right)_{(0,t)\times \oe} + \left(\ueps^1 \delta A_{\e} \weps^2,\nabla \delta \ueps\right)_{(0,t)\times \oe} + \left(\delta \ueps   A_{\e}^2 \weps^2,\nabla\delta \ueps\right)_{(0,t)\times \oe}
    \\
    &=: I_1 + I_2 + I_3.
\end{align*}
First of all, using the essential bound for $A_{\e}^1$, see Lemma \ref{lem:EstimatesEpsTrafo}, and the essential bound for $\ueps^1$ from Lemma \ref{LemmaLinftyBoundUeps} in Section \ref{SectionAprioriEstimates} below, we obtain
\begin{align*}
I_1 &\le C_{\e} \|\ueps^1\|_{L^{\infty}((0,t) \times \oe))}\|\delta \weps \|_{L^{2}((0,t) \times \oe)} \|\nabla \delta \ueps\|_{L^2((0,t)\times \oe)}
\\
&\le C_{\e} \|\delta \weps\|_{L^{2}((0,t)\times\oe)} \|\nabla \delta \ueps\|_{L^2((0,t)\times \oe)}.
\end{align*}
Using estimate \eqref{eq:Lip:F2_L_infinity} from Lemma \ref{lem:Lip_L2}, as well as Lemma \ref{lem:LipschitzContnuityF1}, we obtain for arbitrary $\theta>0$  the existence of $C_{\e}(\theta)$ such that 
\begin{align*} 
I_1 &\le C_{\e} \left(\norm{\delta R_\e}{L^\infty((0,t) \times \Oe)} + \norm{\delta \partial_t R_\e}{L^2((0,t) \times \Oe)} \right)  \|\nabla \delta \ueps\|_{L^2((0,t)\times \oe)}
\\
&\leq C_\e \left\| \delta \ueps \right\|_{L^2((0,t),L^1(\geps))}  \|\nabla \delta \ueps\|_{L^2((0,t)\times \oe)}
\\
&\leq C_{\e}(\theta) \|\delta \ueps \|_{L^2((0,t) \times \oe)}^2 + \theta \|\delta \nabla \ueps \|_{L^2((0,t)\times \oe)}^2.
\end{align*}
Next, using again Lemma \ref{lem:EstimateShiftsReps} and \ref{lem:LipschitzContnuityF1} and the a priori bounds for $\ueps^1$ and $\weps^2$, we get with the H\"older inequality
\begin{align*}
I_2 &\le C_{\e}\|\ueps^1\|_{L^{\infty}((0,t) \times \oe))} \|\delta A_{\e}\|_{L^{\infty}((0,t)\times \oe)} \|\weps^2\|_{L^{\infty}((0,t), L^2(\oe))} \|\nabla \delta \ueps \|_{L^2((0,t)\times \oe)}
\\
&\le C_{\e} \|\delta \ueps\|_{L^1((0,t)\times \Ge)} \|\nabla \delta \ueps \|_{L^2((0,t)\times \oe)}
\\
&\le C_{\e}(\theta) \|\delta \ueps \|_{L^2((0,t)\times \oe )}^2 + \theta \|\delta \nabla \ueps \|_{L^2((0,t)\times \oe)}^2.
\end{align*}
Finally, for $I_3$, we use integration by parts, the zero boundary condition of $\ueps$ on $\partial \Omega$ resp.\ $\weps $ on $\geps$ as well as $\partial_t \jeps = \nabla \cdot (A_{\e}\weps)$ and Lemma \ref{lem:EstimatesEpsTrafo} to obtain
\begin{align*}
    I_3 = \frac12  (\delta \ueps \partial_t \jeps^2 , \delta \ueps )_{(0,t)\times \oe} \le C_{\e} \|\delta \ueps\|_{L^2((0,t)\times \oe)}^2.
\end{align*}
 Altogether, we obtain (perhaps after some redefinition of $\theta$)
\begin{align*}
\big(\delta & (\ueps A_{\e} \weps), \nabla \delta \ueps\big)_{(0,t)\times \oe)} 
\\
&\le C_{\e}(\theta)\left(\|\delta \ueps \|^2_{L^2((0,t)\times \oe)} + \|\delta \hueps \|^2_{L^2((0,t),H^{\beta}(\oe))} \right) + \theta \|\delta \nabla \ueps \|^2_{L^2((0,t)\times \oe)}.
\end{align*}
Plugging this estimate into $\eqref{eq:lipschitz_estimate}$ and choosing $\theta$ small enough, such that we can absorb the gradient term on the right-hand side by the left-hand side, we obtain 
\begin{align*}
    \|\delta \ueps(t)\|_{L^2(\oe)} + c\|\delta \nabla \ueps\|_{L^2((0,t)\times \oe)}  \le C_{\e} \|\delta \ueps \|_{L^2((0,t) \times \oe)} .
\end{align*}
Application of the Gronwall inequality leads to the desired result.
\end{proof}

\section{A priori estimates for the microscopic solution}
\label{SectionAprioriEstimates}

In this section, we show uniform a priori estimates for  the weak solution $(\ueps,\weps,\reps)$ of the microscopic problem $\eqref{MicroModelFixedDomain}$. These estimates are the basis for the two-scale compactness results stated in Proposition \ref{prop:Compactness_dt} and Proposition \ref{prop:Comapctness} of Section \ref{SubsectionCompactness} below, which allow us to pass to the limit $\e \to 0$ in the weak formulation \eqref{eq:WeakForm:Microproblem}.  The proof follows the following scheme: The uniform bound for $\reps$ is a direct consequence of the essential boundedness of $g$. Therefore, $\reps$ and $\partial_t \reps$ can be estimated independently of $\ueps$. The bounds for $\reps$ allow to control the velocity $\weps$ and the pressure $\peps$. Finally, we are able to estimate the concentration $\ueps$.  

\begin{lemma}\label{LemmaAprioriEstimates}
The weak solution $(\ueps,\weps,\peps\reps)$   of the microscopic problem $\eqref{MicroModelFixedDomain}$ fulfills the following a priori estimates
\begin{align}
    \label{AprioriEstimatesReps}
     \|\reps\|_{L^{\infty}((0,T)\times \oe) } + \|\partial_t \reps \|_{L^{\infty}((0,T)\times \oe)} &\le C \e,
     \\
     \label{AprioriEstimatesWeps}
     \|\weps \|_{L^\infty((0,T),L^2(\Oe))}+ \e \|\nabla \weps \|_{L^\infty((0,T),L^2(\Oe))} + \|\peps\|_{L^{\infty}((0,T),L^2(\oe))} &\leq C,
     \\
     \label{AprioriEstimatesUeps}
     	\e \|\partial_t \ueps\|_{L^2((0,T), H^1(\Oe)')} + 
	\|\ueps\|_{L^\infty((0,T),L^2(\Oe))}+ \|\nabla \ueps\|_{L^2((0,T)\times \oe)} &\leq C
	\\
	\label{AprioriEstimatesJepsUeps}
	\|\partial_t (J_\e \ueps)\|_{L^2((0,T),H^1_{\partial\Omega}(\Oe)')} + \|\jeps \ueps\|_{L^{\infty}((0,T),L^2(\oe))} +  \e \|\nabla (\jeps \ueps) \|_{L^2((0,T)\times \oe)} &\leq C
\end{align}

\end{lemma}
\begin{proof}
\noindent\textit{Bound for $\reps$:} We start with the bounds for the radius $\reps$.  This follows immediately from the essential boundedness of $g$: For for almost every $(t,x) \in (0,T)\times \oe$, we have
\begin{align*}
    |\partial_t \reps(t,x) | \le \e \|g\|_{L^{\infty}(\R^2)} \le \e C_g.
\end{align*}
This also implies the $L^\infty$-bound for $\reps$ (which was already shown in Lemma \ref{LemmaF1welldefined}).
\\
\noindent\textit{Bound for $\weps$ and $\peps$:} 
Due to the bound $\eqref{AprioriEstimatesReps}$ for $\reps$ showed above, we can apply Lemma \ref{lem:EstimatesEpsTrafo}. Now, the results follow directly from \cite[Theorem 3.1]{wiedemann2021homogenisation}.
\\
\noindent\textit{Bound for $\ueps$:} Here, we use similar arguments as in the proof of \cite[Lemma 6]{WIEDEMANN2023113168}, so we skip some details. The main new part is the convective term. 
We divide the proof into two steps, showing the estimate for 
$\|\ueps\|_{L^\infty((0,T),L^2(\Oe))}$ and $\|\nabla \ueps\|_{L^2((0,T)\times \oe)}$ first (which then directly also implies the estimate for 
 $\e \|\nabla (\jeps \ueps) \|_{L^2((0,T)\times \oe)}$):\\
\noindent{\textit{Step A:}} We begin by testing the weak equation for $\ueps$  $\eqref{eq:weakForm:u_eps}$ with $\ueps$ to obtain almost everywhere in $(0,T)$
\begin{align*}
 \langle \partial_t &(\jeps \ueps), \ueps \rangle_{H^1(\oe)',H^1(\oe)} + \int_{\oe} D_{\e} \nabla \ueps \cdot \nabla \ueps \,\mathrm{d}x = 
 \\
 &\int_{\oe} \ueps A_{\e} \weps \cdot \nabla \ueps \,\mathrm{d}x + \int_{\oe } \jeps f(\ueps) \ueps \,\mathrm{d}x - \rho \int_{\Ge} g\left(\ueps,\e^{-1} \reps\right) \jeps \ueps \,\mathrm{d}\sigma.
\end{align*}
Using the uniform coercivity of $D_{\e}$, see $\eqref{eq:coercivityFeps}$, we get for almost every $t \in (0,T)$
\begin{align*}
\frac12 \frac{\mathrm{d}}{\mathrm{d}t} &\left\| \sqrt{\jeps} \ueps (t) \right\|_{L^2(\oe)}^2 + \alpha \|\nabla \ueps \|^2_{L^2(\oe)}
\\
&\le -  \frac12 \int_{\oe} \partial_t \jeps |\ueps|^2 \,\mathrm{d}x  + \int_{\oe} \ueps A_{\e} \weps \cdot \nabla \ueps \,\mathrm{d}x\\ &\quad + \int_{\oe } \jeps f(\ueps) \ueps \,\mathrm{d}x - \rho \int_{\Ge} g\left(\ueps,\e^{-1} \reps\right) \jeps \ueps \,\mathrm{d}\sigma. 
\end{align*}
Using the uniform bounds for $\reps$, $\jeps$ (see Lemma \ref{lem:EstimatesEpsTrafo}), the growth condition on $f$ from assumption \ref{Assumptionf} and the boundedness of $g$ from assumption \ref{eq:def:C_g} together with the trace inequality, we obtain for arbitrary $\theta>0$ a constant $C_{\theta}>0$ such that 
\begin{align*}
   \int_{\oe } \jeps & f_{\e} \ueps \,\mathrm{d}x - \rho \int_{\Ge} g\left(\ueps,\e^{-1} \reps\right) \jeps \ueps \,\mathrm{d}\sigma 
   \\
   &\le C_{\theta} \left(1 + \|\ueps\|^2_{L^2(\oe)} \right) + \theta \e^2 \|\nabla \ueps\|_{L^2(\oe)}^2.
\end{align*}
For the convective term, we use integration by parts and $\nabla \cdot (A_{\e}\weps) = -\partial_t \jeps$ (see Remark \ref{RemarksTrafoVelocity})  to obtain
\begin{align}\label{Eq:Adv_Term_Apriori}
-  \frac12 \int_{\oe} \partial_t \jeps |\ueps|^2 \,\mathrm{d}x  + \int_{\oe} \ueps A_{\e} \weps \cdot \nabla \ueps \,\mathrm{d}x = \frac12 \int_{\partial \oe } |\ueps|^2 A_{\e}\weps \cdot \nu \,\mathrm{d}\sigma = 0.
\end{align}
For the last equality, we used $\weps = 0$ on $\Ge$ and $\ueps= 0$ on $\partial \Omega$.
Choosing $\theta >0$ small enough, integration by parts and Gronwall's inequality imply 
\begin{align*}
	\|\ueps\|_{L^\infty((0,T),L^2(\Oe))}+ \|\nabla \ueps\|_{L^2((0,T)\times \oe)} \leq C.
\end{align*}
Now, from the properties of $\jeps $ and $\nabla \jeps$ in Lemma \ref{lem:EstimatesEpsTrafo}, we get
\begin{align*}
 \e \|\nabla (\jeps \ueps) \|_{L^2((0,T)\times \oe)} \leq C.
\end{align*}
\noindent{\textit{Step B:}}
Now, we show the estimate for $\partial_t (\jeps \ueps)$. 
For this, we use the uniform bound for $\ueps$ in $L^{\infty}((0,T)\times \oe)$, which is shown in Lemma \ref{LemmaLinftyBoundUeps} below, for the proof of which we only need the regularity of the weak solution but no $\e$-uniform a priori bounds. However, we emphasise that for the proof of the existence of a weak microscopic solution we used the Galerkin method, see Lemma \ref{LemmaF3}, and for the Galerkin approximation we have no (uniform) $L^\infty$-bound (the proof of Lemma \ref{LemmaLinftyBoundUeps} is not valid for the Galerkin approximation). We also obtain the regularity for $\partial_t (\jeps \ueps)$ without the $L^\infty$-bound for $\ueps$, see Remark \ref{Rem:RegTimeUeps} below. 

We test $\eqref{eq:weakForm:u_eps}$ with $\phi \in H^1(\oe)$ such that $\|\phi\|_{H^1(\oe)} \le 1$. We only estimate the convective term. The estimate for the other terms follows straightforwardly. 
The $L^{\infty}$-bound for $\ueps$ is obtained from Lemma \ref{LemmaLinftyBoundUeps} using the essential bound for $A_{\e}$ (see Lemma \ref{lem:EstimatesEpsTrafo}) and the H\"older inequality for almost every $t \in (0,T)$. More concretely, we obtain $\ueps(t) \in L^{\infty}(\oe)$ and
\begin{align*}
   \|\ueps(t)\|_{L^{\infty}(\oe)} \le \|\ueps\|_{L^{\infty}((0,T)\times \oe)} \le C 
\end{align*}
for almost every $t \in (0,T)$.
(We emphasise that, in general, the mapping $t \mapsto \|\ueps(t)\|_{L^{\infty}(\oe)}$, in other words $\ueps:(0,T)\rightarrow L^{\infty}(\oe)$, is not measurable.) Furthermore, we obtain for almost every $t \in (0,T)$
\begin{align*}
    \int_{\oe} \left( \ueps A_{\e} \weps \cdot \nabla \phi \right)(t,x) \,\mathrm{d}x &\le  C\|\ueps(t)\|_{L^{\infty}(\oe)} \|\weps(t)\|_{L^2(\oe)} \|\nabla \phi\|_{L^2(\oe)}
    \\
    &\le C.
\end{align*}
Here, we also used the $L^\infty$-bound with respect to time of $\weps$.  Hence, we obtain for almost every $t \in (0,T)$ that
\begin{align*}
    \left|\left\langle \partial_t (\jeps \ueps)(t),\phi \right\rangle\right| \le C \left( 1 + \|\nabla \ueps(t)\|_{L^2(\oe)} \right)
\end{align*}
and, therefore, 
\begin{align*}
    \left\|\partial_t (\jeps\ueps)(t)\right\|_{H^1_{\partial\Omega} (\oe)'}\le C \left( 1 + \|\nabla \ueps(t)\|_{L^2(\oe)} \right).
    \end{align*}
    Using the a priori bound for $\nabla \ueps$, we obtain further
    \begin{align*}
        \left\|\partial_t (\jeps\ueps) \right\|_{L^2((0,T),H^1_{\partial \Omega}(\oe)')} \le C.
    \end{align*}
Now, the estimate for $\partial_t \ueps $ follows from the product rule and the bound for $\e \nabla \jeps$ in Lemma \ref{lem:CellTrafo}. A similar argument is valid for $\jeps\ueps$ and $\nabla(\jeps\ueps)$. Thus, the proof is complete except for the verification of Lemma \ref{LemmaLinftyBoundUeps} which is undertaken below.
\end{proof}
\begin{remark}\label{Rem:RegTimeUeps}
To obtain the $\e$-uniform bound for $\partial_t (\jeps \ueps)$ in Lemma \ref{LemmaAprioriEstimates}, we need the uniform $L^{\infty}$-bound for $\ueps$ from Lemma \ref{LemmaLinftyBoundUeps} below. However, in the Galerkin method in the proof of Lemma \ref{LemmaF3} (where we are not allowed to use Lemma \ref{LemmaLinftyBoundUeps}), we can argue in the following way to obtain a bound for $\partial_t (\jeps \ueps)$ (not uniformly with respect to $\e$): From the \textit{a priori} bounds already obtained for $\weps$ and $\ueps$ in Step A of the proof above, we get
\begin{align*}
    \left|\int_{\Oe}\ueps \Aeps \weps \cdot \nabla \phi \,\mathrm{d}x \right| &\le C_{\e} \|\ueps\|_{L^4(\Oe)}\|\weps\|_{L^4(\Oe)}\|\nabla \phi\|_{L^2(\Oe)}
    \\
    &\le C_{\e} \|\ueps\|_{H^1(\Oe)} \|\weps \|_{H^1(\Oe)}
\end{align*}
for a constant $C_{\e}$ depending on $\e$.
Now, arguing as in the proof above, we obtain with the $L^\infty((0,T),H^1(\Oe)^n)$-bound for $\weps$ and the $L^2((0,T),H^1(\Oe))$-bound for $\ueps$ the estimate
\begin{align*}
    \|\partial_t (\jeps \ueps)\|_{L^2((0,T),H^1_{\partial\Omega} (\Oe)')} \le C_{\e}.
\end{align*}
\end{remark}

We now show the essential bound for $\ueps$, which holds uniformly with respect to the homogenisation parameter $\e$. Further, under some additional assumptions on the data, the non-negativity of the concentration $\ueps$ also follows.

\begin{lemma}\label{LemmaLinftyBoundUeps}
The solution $\ueps$ is  uniformly essentially bounded. More precisely, there exists a constant $C>0$ (independent of $\e$) such that
\begin{align}
    \|\ueps \|_{L^{\infty}((0,T)\times \oe)} \le C.
\end{align}
Under the additional assumptions on the data
\begin{align}
    f(u)(u)^- \le |(u)^-|^2, \qquad g(u,R)(u)^- \geq 0, \qquad \ueps^\mathrm{in} \geq 0
\end{align}
for every $R,u\in \R$, where $(\cdot)^- = \min\{\cdot,0\}$, the solution $\ueps$ is non-negative.
\end{lemma}

\begin{proof}
We define for $\omega \geq 0$ the function $W:= \mathrm{e}^{-\omega t} \ueps$ and for $k \geq 1$  the functions $W_k:= W - k$ and $W_k^+ := (W - k)^+$, where $(\cdot)^+:= \max(\cdot,0)$. Choosing $\mathrm{e}^{-\omega t} W_k^+$ as a test function in $\eqref{eq:weakForm:u_eps}$, we obtain
\begin{align}
\begin{aligned}\label{AuxiliaryEquationLinfty}
    \int_0^t \langle & \partial_t (\jeps \ueps) , \mathrm{e}^{-\omega s} W_k^+\rangle \,\mathrm{d}s + (D_{\e} \nabla \ueps , \mathrm{e}^{-\omega s} \nabla W_k^+ )_{(0,t)\times \Oe}
    \\
    =& \left( A_{\e}\weps \ueps , \mathrm{e}^{-\omega s } \nabla W_k^+ \right)_{(0,t)\times \Oe} + (\jeps f(\ueps), \mathrm{e}^{-\omega s} W_k^+ )_{(0,t)\times \Oe} 
    \\&+ \rho\left( \Jeps g(\ueps, \e^{-1} \Reps), \mathrm{e}^{-\omega s} W_k^+\right)_{(0,t)\times \geps}.
\end{aligned}
\end{align}
Using the identity 
\begin{align*}
    \mathrm{e}^{-\omega t} \partial_t (\jeps \ueps) = \partial_t \left[\jeps \left(\mathrm{e}^{-\omega t} \ueps - k \right)\right] + \omega \mathrm{e}^{-\omega t} \jeps \ueps + k \partial_t \jeps 
\end{align*}
in $L^2((0,T);H^1_{\partial \Omega} (\Oe)')$ and similar arguments as in \cite[Lemma 3.2]{Wachsmuth2016}, we get (for $k $ large enough)
\begin{align*}
    \int_0^t \langle \partial_t (\jeps \ueps) , \mathrm{e}^{-\omega s }W_k^+ \rangle \,\mathrm{d}s &= \int_0^t \langle \mathrm{e}^{-\omega s }\partial_t (\jeps \ueps) , W_k^+ \rangle \,\mathrm{d}s 
    \\
    &= \int_0^t \langle \partial_t \left[\jeps W_k \right], W_k^+ \rangle \,\mathrm{d}s + \left( \omega \mathrm{e}^{-\omega s} \jeps \ueps , W_k^+\right)_{(0,t)\times \Oe} + \left(k \partial_t \jeps ,W_k^+ \right)_{(0,t)\times \Oe} 
    \\
    &= \frac12 \left\Vert \sqrt{\jeps} W_k^+(t) \right\Vert_{L^2(\Oe)}^2 - \frac12 \underbrace{ \left\Vert \sqrt{\jeps} W_k^+(0) \right\Vert_{L^2(\Oe)}^2}_{= 0} + \frac12 \int_0^t \int_{\Oe}  \partial_t \jeps W_k W_k^+ \,\mathrm{d}x \,\mathrm{d}s
    \\
    &\quad +  \left( \omega \mathrm{e}^{-\omega s} \jeps \ueps , W_k^+\right)_{(0,t)\times \Oe} + \left(k \partial_t \jeps ,W_k^+ \right)_{(0,t)\times \Oe} .
\end{align*}
For the second term in $\eqref{AuxiliaryEquationLinfty}$, we obtain with the $\e$-uniform coercivity of $D_{\e}$ from $\eqref{eq:coercivityFeps}$
\begin{align*}
    (D_{\e} \nabla \ueps , \mathrm{e}^{-\omega s} \nabla W_k^+ )_{(0,t)\times \Oe } &= (D_{\e} \nabla W_k , \nabla W_k^+)_{(0,t)\times \Oe } = (D_{\e} \nabla W_k^+ , \nabla W_k^+ )_{(0,t)\times \Oe }
    \\
    &\geq c_0 \Vert \nabla W_k^+ \Vert^2_{L^2((0,t)\times \Oe)}.
\end{align*}
Hence, we obtain 
\begin{align}
\begin{aligned}
    \frac12 &\left\Vert \sqrt{\jeps} W_k^+(t) \right\Vert_{L^2(\Oe)}^2 + c_0 \Vert \nabla W_k^+ \Vert^2_{L^2((0,t)\times \Oe)}
    \\
    \le& \left( A_{\e}\weps \ueps , \mathrm{e}^{-\omega s } \nabla W_k^+ \right)_{(0,t)\times \Oe} + (\jeps f(\ueps), \mathrm{e}^{-\omega s} W_k^+ )_{(0,t)\times \Oe} 
    \\
    &+ \rho \left( \Jeps g(\ueps, \e^{-1} \Reps), \mathrm{e}^{-\omega s} W_k^+\right)_{(0,t)\times \geps}
    - \frac12 \int_0^t \int_{\Oe}  \partial_t \jeps W_k W_k^+ \,\mathrm{d}x \,\mathrm{d}s 
    \\
    &-  \left( \omega  \jeps W , W_k^+\right)_{(0,t)\times \Oe} - \left(k \partial_t \jeps ,W_k^+ \right)_{(0,t)\times \Oe} 
    \eqqcolon \sum_{l=1}^6 I_{\e}^6.
\end{aligned}
\end{align}
For the convective term $I_{\e}^1$, we have
\begin{align*}
    I_{\e}^1 &= \left( A_{\e}\weps W , \nabla W_k^+ \right)_{(0,t)\times \Oe } 
    \\
    &= \left(A_{\e}\weps  W_k^+ , \nabla W_k^+ \right)_{(0,t)\times \Oe } + k \left( A_{\e}\weps ,\nabla W_k^+ \right)_{(0,t)\times \Oe }
    \\
    &=: A_{\e}^1 + A_{\e}^2.
\end{align*}
For the first term, we use $\nabla \cdot (A_{\e} \weps) = \partial_t \jeps $ almost everywhere in $\oe$ and $A_{\e} \weps\cdot \nu =0$ almost everywhere on $\Ge$ as well as the zero boundary condition of $\ueps $ on $\partial \Omega $, which is then also valid for $W_k^+$, to obtain with integration by parts 
\begin{align*}
    A_{\e}^1 &= \frac12 \left(A_{\e}\weps  ,\nabla (W_k^+)^2 \right)_{(0,t)\times \Oe }
    = -\frac12 \left(\nabla \cdot \left( A_{\e}\weps \right) ,(W_k^+)^2 \right)_{(0,t)\times \Oe }
    \\
    &= \frac12 \left(\partial_t \jeps , (W_k^+)^2 \right)_{(0,t)\times \Oe } 
    = - I_{\e}^4.
\end{align*}
For the term $A_{\e}^2$, we obtain with similar arguments
\begin{align*}
    A_{\e }^2 = k(\partial_t \jeps , W_k^+)_{(0,t)\times \Oe } = - I_{\e}^6.
\end{align*}
Now, we note that $f(u) \leq C(1+ u) \leq C(1 + W_k^+ + k)$ for $u \geq 0$ and 
$W_k^+ = 0$ for $u \leq 0$ almost everywhere.
Thus, we obtain 
\begin{align}
  |f(u)W_k^+|  \leq  C(1 + W_k^+ + k) W_k^+.
\end{align}
Using additionally the essential boundedness of $\Jeps$ as well as the Hölder and Young inequalities, we obtain for $I_\e^2$ 
\begin{align*}
 \vert I_{\e}^2 \vert &\le \int\limits_{(0,t) \times \Oe} C(1 + W_k^+ + k) W_k^+ \,\mathrm{d}x \,\mathrm{d}s
\\
&\le
C(1+k)   \Vert W_k^+ \Vert_{L^1((0,t)\times \Oe)} 
+
 C \Vert W_k^+ \Vert^2_{L^2((0,t)\times \Oe)}
 \\
&\le C(1 + k^2) \int_0^t \int_{\{W(t) > k \} } \,\mathrm{d}x \,\mathrm{d}s
 +
 C \Vert W_k^+ \Vert^2_{L^2((0,t)\times \Oe)}.
\end{align*}
For the boundary term $I_{\e}^3$, we obtain with the essential boundedness of $\Jeps$ and $g$ as well as the scaled trace inequality for $\theta > 0 $  (with $\e \leq 1$)
\begin{align*}
    \vert I_{\e}^3\vert &\le \e C C_g \Vert \mathrm{e}^{-\omega s} W_k^+ \Vert_{L^1((0,t)\times \geps)} 
    \le C(\theta) \Vert W_k^+ \Vert_{L^1((0,t)\times \Oe)} + \e \theta \Vert \nabla W_k^+ \Vert_{L^1((0,t)\times \Oe)}
    \\
    &\le C(\theta) \int_0^t \int_{\{W(t)> k \} } \,\mathrm{d}x \,\mathrm{d}s + C(\theta) \Vert W_k^+ \Vert_{L^2((0,t)\times \Oe)}^2 + \theta \Vert \nabla W_k^+ \Vert_{L^2((0,t)\times \Oe)}^2.
\end{align*}
For $I_{\e}^5$, we get
\begin{align*}
    I_{\e}^5 &= - \omega (\jeps W_k^+,W_k^+)_{(0,t)\times \Oe } + k(\jeps , W_k^+)_{(0,t)\times \Oe }
    \\
    &\le -\omega J_\mathrm{min} \Vert W_k^+ \Vert_{L^2((0,t)\times \Oe)}^2 + C \left(k^2 \int_0^t \int_{\{W(t) > k\}} \,\mathrm{d}x \,\mathrm{d}s + \Vert W_k^+ \Vert_{L^2((0,t)\times \Oe)}^2 \right).
\end{align*}
Combining all the results, we obtain
\begin{align*}
    \frac12 &\left\Vert\sqrt{\jeps} W_k^+(t)\right\Vert_{L^2(\oe)}^2 + c_0 \Vert \nabla W_k^+\Vert_{L^2((0,t)\times \Oe)}^2
    \\
    &\le C(\theta) k^2 \int_0^t \int_{\{W(t) > k\} } \,\mathrm{d}x \,\mathrm{d}s + \left[C(\theta) - \omega J_\mathrm{min} \right] \Vert W_k^+ \Vert^2_{L^2((0,t)\times \Oe)} + \theta \Vert \nabla W_k^+ \Vert_{L^2((0,t)\times \Oe)}^2.
\end{align*}
For $\theta $ small enough, the term including the gradient on the right-hand side can be absorbed by the left-hand side. Hence, choosing $\omega > \frac{C(\theta)}{J_\mathrm{min}}$ in the inequality above (for fixed $\delta$), we obtain for almost every $t \in (0,T)$
\begin{align}
    \Vert W_k^+ \Vert_{L^2(\Oe)}^2 + \|\nabla W_k^+\|_{L^2((0,t)\times \Oe)}^2 \le C k^2 \int_0^t \int_{\{W(t) > k\} } \,\mathrm{d}x \,\mathrm{d}s.
\end{align}
Now, the result follows from \cite[II Theorem 6.1]{Ladyzenskaja} implying 
\begin{align*}
    \ueps \le C
\end{align*}
for a constant $C>0$ independent of $\e$. Repeating these arguments with $-\ueps$ instead of $\ueps$, we obtain 
\begin{align*}
    -\ueps \le C,
\end{align*}
which gives the essential boundedness of $\ueps$ uniformly with respect to $\e$.

It remains to show the non-negativity. For this, we choose $(\ueps)^- := \min(\ueps,0)$ as a test function in $\eqref{eq:weakForm:u_eps}$ and obtain (by a similar calculation as above for $W_k^+$)
\begin{align*}
    \frac12 \frac{\mathrm{d}}{\mathrm{d}t}\|(\ueps)^- \|^2_{L^2(\Oe)} + c_0 \|&\nabla (\ueps)_-\|^2_{L^2(\Oe)} 
    \\
    &\le \int_{\Oe} \jeps f(\ueps) (\ueps)^- \,\mathrm{d}x - \rho \int_{\geps } \jeps g\left(\ueps,\e^{-1}\reps\right) (\ueps)^- \,\mathrm{d} \sigma
    \\
    &\le C \|(\ueps)^-\|^2_{L^2(\Oe)}.
\end{align*}
Now, the Gronwall inequality implies the non-negativity of the microscopic concentration $\ueps$.
\end{proof}

\section{Derivation of the macroscopic model}\label{sec:Homogenisation}

In this section, we show that the microscopic solution $(\ueps,\weps,\peps,\reps)$ converges in a suitable sense to the solution of the macroscopic model $\eqref{MacroModel}$ formulated in Section \ref{SectionMainResults} for $\e \to 0$. In a first step, we establish compactness results obtained from the {a priori} estimates in Section \ref{SectionAprioriEstimates} and, in a second step, we pass to the limit $\e \to 0$ in the weak formulation of the microscopic problem in Definition \ref{def:DefinitionWeakSolutionMicroModel}.

\subsection{Compactness results for the microscopic solutions}
\label{SubsectionCompactness}
Based on the {a priori} estimates from Section \ref{SectionAprioriEstimates}, we derive compactness results for the microscopic solutions. As a suitable convergence concept for the treatment of problems including oscillating coefficients and domains, we use the two-scale convergence, see Appendix \ref{SubsectionTwoScaleConvergence}. Since we are dealing with functions on perforated domains, we have to extend these functions to the whole domain. While for the fluid velocity $\weps$ it is natural to use the zero extension, the situation is more complicated for the pressure $\peps$ and the concentration $\ueps$. For the pressure, we use an extension from Tartar given in \cite{SanchezPalencia1980}, see also \cite{lipton1990darcy}. For the concentration, we would loose spatial regularity for the zero extension and, therefore, we use an extension operator from \cite{CioranescuSJPaulin}, see also \cite{Acerbi1992} for more general domains. Namely, there exists an extension operator $E_{\e}: H^1(\oe) \rightarrow H^1(\Omega)$ such that
\begin{align*}
    \|E_{\e} \phieps\|_{L^2(\Omega)} &\le C \|\phieps\|_{L^2(\oe)},
    \\
    \|E_{\e}\phieps \|_{H^1(\Omega)} &\le C \|\phieps\|_{H^1(\oe)}
\end{align*}
for all $\phieps \in H^1(\oe)$ 
with a constant $C>0$ independent of $\e$. For time dependent functions, we apply this operator pointwise in $t$. We shortly write
\begin{align}\label{DefExtensionUeps}
    \tueps := E_{\e} \ueps \in L^2((0,T),H^1(\Omega))
\end{align}
and we obtain
\begin{align}\label{EstimateExtensionUeps}
    \|\tueps \|_{L^2((0,T),H^1(\Omega)) } + \|\tueps \|_{L^{\infty} ((0,T),L^2(\Omega))} \le C.
\end{align}
However, we emphasise that this extension operator in general does not preserve the regularity for the time derivative and, therefore, we loose control for the time variable of the microscopic sequence $\tueps$, which is important to obtain strong two-scale compactness results.
The strong two-scale convergence of $\ueps$ was shown for the same a priori estimates in \cite[Proposition 5]{GahnPop2023Mineral} and with additional estimates for the differences of shifts with respect to time in \cite{WIEDEMANN2023113168}. Here, we simplify the proof of \cite{GahnPop2023Mineral} by showing the following general strong two-scale compactness result. The Proposition will then be applied to $\jeps$ and $\ueps$, so we keep the same notation, but we slightly generalise the assumptions for the sake of the proposition.
\begin{proposition}\label{prop:Compactness_dt}
Let $n\geq 2$ and $ V_{\e} \subset H^1(\oe)$ dense in $L^2(\oe)$. 
Let $\ueps \in L^2((0,T),V_{\e})$ and $\jeps \in L^{\infty}((0,T)\times \oe)$ with $\partial_t (\jeps \ueps )\in L^2((0,T),V_{\e}')$ with respect to the usual Gelfand triple $V_{\e}\subset L^2(\oe) \subset V_{\e}'$. Further, let $\jeps \geq c_0 >0 $ for a constant $c_0>0$ independent of $\e$ and 
\begin{align*}
\Vert \partial_t (\jeps \ueps) \Vert_{L^2((0,T),V_{\e}')} + \Vert \ueps \Vert_{L^2((0,T),H^1(\oe))} &\le C.
\end{align*}
Moreover, assume there exists $\kappa(h)$ with $\kappa(h) \to 0$ for $h\to 0$ and $h>0$ such that 
\begin{align}\label{ControlShiftsJeps}
\Vert \jeps(\cdot + h) - \jeps \Vert_{L^{\infty}((0,T-h),L^s(\oe))} \le \kappa(h)
\end{align}
with $s>1$ arbitrary for $n=2$ and $s = \frac{n}{2}$ for $n\geq 3$.
Then, there exists $\kappa_0(h)$ with $\kappa_0(h) \rightarrow 0$ for $h\to 0$ such that
\begin{align*}
\Vert \ueps (\cdot + h ) - \ueps \Vert_{L^2((0,T-h),L^2(\oe))} \le \kappa_0(h).
\end{align*}
In particular, there exists $u_0 \in L^2((0,T)\times \Omega)$ (even in $L^2((0,T),H^1(\Omega))$) such that, up to a subsequence, $\chi_{\oe}\ueps $ converges strongly in the two-scale sense to $\chi_{Y^{\ast}}u_0$.
\end{proposition}
\begin{proof}
For a function $\phi: (0,T) \times U \rightarrow \R$ with $U\subset \R^n$, we define the shift with respect to time  for $0 < h \ll 1$ and $t \in (0,T-h)$ by
\begin{align*}
\delta_h \phi (t):= \delta_h \phi(t,\cdot_x) := \phi(t+h) - \phi(t).
\end{align*}
Since $\jeps \geq c_0 > 0 $ almost everywhere in $(0,T)\times \oe$, it holds for almost every $t \in (0,T-h)$ that
\begin{align*}
c_0 \Vert \delta_h \ueps(t)\Vert_{L^2(\oe)}^2 \le \int_{\oe } \jeps(t+h) (\delta_h \ueps (t))^2 \,\mathrm{d}x.
\end{align*}
Using the identity 
\begin{align*}
\jeps(t+h) (\delta_h \ueps(t))^2 = \delta_h(\jeps \ueps)(t) \delta_h \ueps(t) - \delta_h \jeps(t) \ueps(t) \delta_h \ueps(t),
\end{align*}
we obtain
\begin{align*}
c_0\Vert \delta_h \ueps \Vert_{L^2((0,T-h)\times \oe)}^2 &\le  \int_0^{T-h}\int_{\oe} \delta_h (\jeps\ueps) (t) \delta_h \ueps(t) \,\mathrm{d}x \,\mathrm{d}t\\ &\quad  - \int_0^{T-h} \int_{\oe} \delta_h \jeps(t) \ueps(t) \delta_h \ueps(t) \,\mathrm{d}x\,\mathrm{d}t
\\
&=: A_{\e,h}^1 + A_{\e,h}^2.
\end{align*}
Let us estimate the two terms separately. We start with the second term and restrict to the case $n\geq 3$ (the case $n=2$ follows by similar arguments). Using the Sobolev embedding $H^1(\oe)\hookrightarrow L^{2^{\ast}}(\oe)$ with $2^{\ast} : = {\frac{2n}{n-2}}$, we obtain $\ueps \in L^2((0,T), L^{2^{\ast}}(\oe))$  such that
\begin{align*}
    \left\| \ueps \right\|_{L^2((0,T),L^{p^{\ast}}(\oe))} \le C \|\ueps\|_{L^2((0,T),H^1(\oe))} \le C
\end{align*}
with a constant $C>0$ independent of $\e$. More precisely, this constant is independent of $\e$ owing to the existence of the extension operator $E_{\e}$ defined at the beginning of Section \ref{SubsectionCompactness}.
Now, the H\"older inequality implies 
\begin{align*}
\vert A_{\e,h}^2\vert &\le \int_0^{T-h} \Vert \delta_h \jeps(t)\Vert_{L^{\frac{n}{2}}(\oe)} \Vert \ueps(t) \Vert_{L^{2^{\ast}}(\oe)} \left( \Vert \ueps(t+h) \Vert_{L^{2^{\ast}}(\oe)} + \Vert \ueps(t)\Vert_{L^{2^{\ast}}(\oe)} \right) \,\mathrm{d}t
\\
&\le 2 \Vert \ueps \Vert^2_{L^2(0,T),L^{2^{\ast}}(\oe))} \Vert \delta_h \jeps \Vert_{L^{\infty}((0,T),L^{\frac{n}{2}}(\oe))}
\\
&\le C \kappa(h).
\end{align*}
For the first term $A_{\e,h}^1$, we have (using the natural embedding in the Gelfand triple $V_{\e}\hookrightarrow L^2(\oe) \hookrightarrow V_{\e}'$) 
\begin{align*}
A_{\e,h}^1 &=\int_0^{T-h} \langle \delta_h (\jeps \ueps)(t) , \delta_h \ueps(t) \rangle_{V_{\e}',V_{\e}} \,\mathrm{d}t
\\
&\le \underbrace{\Vert \delta_h (\jeps \ueps) \Vert_{L^2((0,T-h),V_{\e}')} }_{ \le h \Vert \partial_t (\jeps \ueps)\Vert_{L^2((0,T),V_{\e}')} } \Vert \delta_h \ueps(t) \Vert_{L^2((0,T),H^1(\oe))}
\\
&\le C h.
\end{align*}
Altogether, we obtain 
\begin{align*}
\Vert \delta_h \ueps\Vert_{L^2((0,T-h)\times \oe)}\le C \sqrt{h + \kappa(h)} =: \kappa_0(h)
\end{align*}
and, obviously, $\kappa_0(h) \rightarrow 0$ for $h\to 0$.

It remains to show the strong two-scale convergence of $\ueps$. The proof follows similar ideas as in \cite[Lemma 10]{Gahn}. 
Using the extension $\tueps$ of $\ueps$ defined in $\eqref{DefExtensionUeps}$, we obtain with $\eqref{EstimateExtensionUeps}$
\begin{align*}
\left\Vert \int_{t_1}^{t_2} \tueps \,\mathrm{d}t\right\Vert_{H^1(\Omega)} \le C \Vert \tueps \Vert_{L^2((0,T),H^1(\Omega))} \le C .
\end{align*}
Then, noting the compact embedding $H^1(\Omega) \hookrightarrow L^2(\Omega)$, we have that $\int_{t_1}^{t_2} \tueps \,\mathrm{d}t $ is relatively compact in $L^2(\Omega)$. Further, it holds that (the extension operator is pointwise defined in time)
\begin{align*}
\Vert \delta_h \tueps \Vert_{L^2((0,T-h),L^2(\Omega))} \le C \Vert \delta_h \ueps \Vert_{L^2((0,T-h),L^2(\oe))} \le C \kappa_0(h) \overset{h\to 0}{\longrightarrow} 0.
\end{align*}
Therefore, by the Kolmogorov--Simon compactness theorem, see \cite[Theorem 1]{Simon}, $\tueps$ converges strongly in $L^2((0,T)\times \Omega)$. This gives the desired result.
\end{proof}

\begin{remark}\ 
\begin{enumerate}[label = (\roman*)]
\item Proposition \ref{prop:Compactness_dt} remains valid for arbitrary perforated domains $\oe$ as long there exists a linear extension operator $E_{\e}^{\ast}: H^1(\oe)\rightarrow H^1(\Omega)$ such that 
\begin{align*}
    \|E_{\e}^{\ast}\ueps \|_{H^1(\Omega)} \le C \|\ueps \|_{H^1(\oe)} \quad\mbox{ for all } \ueps \in H^1(\oe).
\end{align*}
In particular, the result is valid for $\oe = \Omega$ and gives (besides the strong two-scale convergence) the strong convergence in $L^2((0,T)\times \Omega)$ in this case.

\item Under the additional assumption $\ueps \in L^{\infty}((0,T),L^2(\oe))$ with 
\begin{align*}
    \|\ueps\|_{L^{\infty}((0,T),L^2(\oe))} \le C
\end{align*}
for a constant $C>0$ independent of $\e$, we can replace $\eqref{ControlShiftsJeps}$ with 
\begin{align*}
    \|\jeps(\cdot + h)  - \jeps\|_{L^2((0,T-h),L^{\infty}(\oe))} \le \kappa (h)
\end{align*}
and obtain the same result (using slightly different arguments for the term $A_{\e,h}^2$ in the proof).
\end{enumerate}
\end{remark}
In the following proposition, we give the compactness results for the microscopic solutions. The results for $(\ueps,\reps)$ have already been established in \cite{GahnPop2023Mineral} for the same \textit{a priori} estimates as in Lemma \ref{LemmaAprioriEstimates} and the Lipschitz estimate for $\reps$ from Lemma \ref{lem:LipschitzContnuityF1} (with slightly more regularity assumptions on $\reps(0)$). Similar compactness results were also obtained in \cite{WIEDEMANN2023113168}, where an additional estimate for the differences of the shifts of $\ueps$ with respect to time was shown for the strong convergence of $\ueps$, and the strong convergence of $\reps$ was shown by a direct comparison with the macroscopic radius $R_0$ using the microscopic and macroscopic ODE.

For $\ueps $, we use the extension from $\eqref{DefExtensionUeps}$. Further, since $w_\e \in L^\infty((0,T);H^1_\Ge(\Oe))$, we can extend it by zero to a function $w_\e \in L^\infty((0,T);H^1(\Omega))$ in order to state the convergence.
Moreover, in order to derive not only a weak but also strong convergence result for the pressure, we extend it as follows (see \cite{lipton1990darcy}):
\begin{align}
Q_\e (t,x) \coloneqq \begin{cases}
    q_\e(t,x) &\textrm{ if } x \in \Oe,
    \\
    \frac{1}{|Y^*|} \int\limits_{k + \e Y^*} q_\e(t, z) \,\mathrm{d}z &\textrm{ if } x \in k + \e (Y\setminus Y^*) \textrm{ for } k\in K_\e.
\end{cases}
\end{align}
Note that this extension coincides with the extension from Tartar in \cite{SanchezPalencia1980}.

\begin{proposition}\label{prop:Comapctness}
 The microscopic solutions $(\ueps,\weps, \peps,\reps)$ fulfill the following convergence results up to a subsequence:
\begin{enumerate}[label = (\roman*)]
\item\label{CompactnessUeps} There exist $u_0 \in L^2((0,T),H^1_0(\Omega))\cap  L^{\infty}((0,T),L^2(\Omega))$ and $u_1 \in L^2((0,T)\times \Omega ,H_{\per}^1(Y^{\ast})/\R)$ such that 
\begin{align*}
    \tueps &\rightarrow u_0 \quad \mbox{ in } L^2((0,T)\times \Omega),
    \\
    \chi_{\oe}\nabla \ueps &\rightwts{2} \chi_{Y^{\ast}}\left(\nabla u_0 + \nabla_y u_1\right) ,
    \\
    \chi_{\oe} \ueps &\rightsts{2}  \chi_{Y^{\ast}}u_0,
    \\
    \ueps|_{\geps} &\rightsts{2}  u_0 \quad \mbox{ on } \geps .
\end{align*}

\item\label{CompactnessReps} There exists $R_0 \in H^1((0,T),L^2(\Omega)) \cap L^{\infty}((0,T)\times \Omega)$ with $\partial_t R_0 \in L^{\infty}((0,T)\times \Omega)$ such that for all $p\in [1,\infty)$
\begin{align}
\label{ConvergenceReps}
    \e^{-1} \reps &\rightarrow R_0 &\mbox{ in }& L^p((0,T)\times \Omega),
    \\
\label{ConvergenceDtReps}
    \e^{-1} \partial_t \reps &\rightarrow \partial_t R_0 &\mbox{ in }& L^p((0,T)\times \Omega).
\end{align}
These convergences are also valid in the strong two-scale sense in $\Omega$ and $\geps$. Further, it holds almost everywhere in $(0,T)\times \Omega$ that
\begin{align*}
    \partial_t R_0 = g(u_0,R_0). 
\end{align*}
\item\label{CompactnessWepsPeps} There exists $w_0 \in L^\infty((0,T);L^2(\Omega;H^1_\#(Y)))$ with 
\begin{align}\label{eq:Compactness:w_0:P1}
    w_0&= 0 &\mbox{ in } & (0,T)\times \Omega \times (Y\setminus Y^*),
    \\\label{eq:Compactness:div_yw_0}
    \div_y(A_0 w_0 ) &= 0  &\mbox{ in } & (0,T)\times \Omega \times Y,
    \\\label{eq:Compactness:div_xw_0}
    \div_x \left( \int\limits_{Y^{\ast}} A_0 w_0  \right) &= - \partial_t \theta &\mbox{ in } & (0,T)\times \Omega \times Y,
\end{align}
such that for the zero extension of $\weps$ to $\Omega$ it holds  for every $p \in [1,\infty)$ that
\begin{align}\label{eq:Compactness:w_0}
    \weps  &\rightwts{p,2} w_0 ,
    \\\label{eq:Compactness:nablaw_0}
    \e \nabla \weps &\rightwts{p,2} \nabla_y w_0 ,
    \\\label{eq:Compactness:nablaSw_0}
    \e e_{\e}(\weps) &\rightwts{p,2} e_0(w_0),
\end{align}
where $e_0(w_0) = \tfrac{1}{2}(F_0^{-\top}\nabla_y w_0 + \nabla_y w_0^\top F_0^{-1})$  and $F_0$ is defined in \eqref{eq:def:F0} below.

Moreover, there exists $Q_0 \in L^\infty((0,T);H^1_0(\Omega))$ such that for every $p \in [1, \infty)$
\begin{align}\label{eq:Compactness:Q_e}
   Q_\e &\rightarrow Q_0 &\mbox{ in } & L^p((0,T);L^2(\Omega)),
   \\\label{eq:Compactness:P_e}
   Q_\e + p_\mathrm{b}=:P_\e &\rightarrow P_0:= Q_0 + p_\mathrm{b} &\mbox{in }& L^p((0,T),L^2(\Omega)).
\end{align}
\end{enumerate}
\end{proposition}

\begin{proof}
We start with \ref{CompactnessUeps}. The existence of $u_0$ and $u_1$ and the  weak two-scale convergence results follow from the {a priori} estimates in Lemma \ref{LemmaAprioriEstimates} and the two-scale compactness results from Lemma \ref{BasicTwoScaleCompactness} in the appendix. Using again the {a priori} estimates from Lemma \ref{LemmaAprioriEstimates}, the strong convergence of $\tueps$ follows from Proposition \ref{prop:Compactness_dt} if we show $\eqref{ControlShiftsJeps}$. However, this is a direct consequence of Lemma \ref{lem:EstimateShiftsReps} (with $\reps^1 = \reps$ and $\reps^2:= \reps(\cdot_t + h,\cdot_x)$) and the Lipschitz continuity of $\reps$.

Next, we consider the convergence results in \ref{CompactnessReps} for $\reps$. The first result $\eqref{ConvergenceReps}$ was shown in \cite[Proposition~4]{GahnPop2023Mineral}, and in the strong two-scale sense 
on the surface $\geps$ in \cite[Corollary 7]{GahnPop2023Mineral} but with slightly more regularity assumptions on the initial values $\reps^\mathrm{in}$ and $R_0^\mathrm{in}$. In particular, it was assumed that $R_0^\mathrm{in} \in H^1(\Omega)$. Here, we dropped this assumption and only assume (as in \cite{WIEDEMANN2023113168}) the strong convergence of $\e^{-1} \reps^\mathrm{in}$ in $L^2(\Omega)$. However, since the proof still follows the same lines as in \cite[Proposition~4]{GahnPop2023Mineral}, we only sketch the main ideas. We use the standard Kolmogorov compactness result, see for example \cite{Brezis}, and since $\reps$ is bounded in $L^2(\Omega)$ it is enough to control the shifts in time and space. In what follows, we extend $\reps$ and $\ueps$ by zero outside $\Omega$, which is possible due to the zero boundary condition of $\ueps$ on $\partial \Omega$. (It is also possible to consider shifts within $\Omega$ but then the argumentation becomes more technical.) The essential bound of $\partial_t \reps$ implies for all $h>0$ and $t \in (0,T-h)$ 
\begin{align*}
   \e^{-1} \|\reps(\cdot_t + h , \cdot_x) -\reps \|_{L^{\infty}((0,T-h)\times \Omega)} \le C h.
\end{align*}
For the spatial variable, it is enough to consider shifts with $\e l$ for $l\in \Z^n$ (see  also \cite[(42)]{GahnPop2023Mineral}). From the Lipschitz estimate $\eqref{Lip_Estimate_reps}$ in Lemma \ref{lem:LipschitzContnuityF1}, we obtain, using the trace inequality (see Lemma \ref{Lem:TraceInequality} in the appendix), the strong convergence of $\e^{-1} \reps^\mathrm{in}$ in $L^2(\Omega)$ (which guarantees a bound for the shifts), the mean value theorem ($\ueps$ is extended via $\tueps$ to the whole domain $\Omega$) and the {a priori} bound for $\nabla \ueps$ already obtained above,
\begin{align*}
    \e^{-1}\|\reps(\cdot_t , \cdot_x + \e l) - \reps\|_{L^{\infty}((0,T),L^2(\Omega))} \le C\left( \e + \kappa(|\e l|)\right)
\end{align*}
with $\kappa(|\e l |) \rightarrow 0$ for $|\e l| \to 0$. Hence, the Kolmogorov compactness result implies the desired result.

The convergence of the time derivative $\partial_t \reps$ was shown in \cite[Proposition 6]{GahnPop2023Mineral} but we give a much simpler argument for the compactness here. (In \cite[\textcolor{blue}{...}]{WIEDEMANN2023113168} the macroscopic equation was used.) By a change of coordinates, we obtain from the ODE $\eqref{eq:weakForm:r_eps}$ for $\reps$ 
\begin{align*}
   \e^{-1} \partial_t \reps(t,x) = \frac{1}{|\Gamma|} \int_{\Gamma} g\left(\teps(\ueps)(t,x,z) , \teps (\e^{-1} \reps)(t,x,z)\right) \,\mathrm{d}\sigma_z
\end{align*}
for almost every $(t,x) \in (0,T)\times \Omega$.
Using the strong two-scale convergence of $\ueps$ and $\e^{-1} \reps$ on $\geps$, we obtain
\begin{align*}
    \teps(\e^{-1}\reps) &\rightarrow R_0 &\mbox{ in }& L^p((0,T)\times \Omega \times \Gamma),
    \\
    \teps(\ueps) &\rightarrow u_0 &\mbox{ in }& L^2((0,T)\times \Omega \times \Gamma).
\end{align*}
Hence, up to a subsequence, these convergence also hold pointwise almost everywhere in $(0,T)\times \Omega \times \Gamma$. Then, for almost every $(t,x)\in (0,T)\times \Omega \times \Gamma$, we get
\begin{align}
    \e^{-1} \partial_t \reps (t,x) \rightarrow \frac{1}{|\Gamma|} \int_{\Gamma} g(u_0(t,x),R_0(t,x)) \,\mathrm{d}\sigma_z = g(u_0(t,x),R_0(t,x)).
\end{align}
The dominated convergence theorem of Lebesgue now implies the convergence in $L^p((0,T)\times \Omega)$.  It remains to show the convergence of the traces of $\partial_t \reps$.  Since $\partial_t \reps$ is constant on every microscopic cell, we also have that $\teps (\partial_t \reps)(t,x,\cdot_y) $ is constant for almost every $(t,x)$. Hence, using the trace inequality, we obtain (with $\nabla_y( \teps(\e^{-1}\partial_t \reps ) - g(u_0,R_0) ) = 0$)
\begin{align*}
    \|\teps(\e^{-1} \partial_t \reps ) - g(u_0,R_0)\|_{L^p((0,T)\times \Omega \times \Gamma)} \le C \|\teps(\e^{-1} \partial_t \reps) - g(u_0,R_0)\|_{L^p((0,T)\times \Omega \times Y^{\ast})} \rightarrow 0.
\end{align*}

Last, we consider the convergence results in \ref{CompactnessWepsPeps}.
The compactness results \eqref{eq:Compactness:w_0} and \eqref{eq:Compactness:nablaw_0} with the properties \eqref{eq:Compactness:w_0:P1}, \eqref{eq:Compactness:div_yw_0}, \eqref{eq:Compactness:div_xw_0} for the limit function $w_0$ are shown in \cite[Lemma 4.9]{wiedemann2021homogenisation} and \cite[Corollary 5.2]{wiedemann2021homogenisation}.
Moreover, \eqref{eq:Compactness:nablaSw_0} can be concluded from \eqref{eq:Compactness:nablaw_0} and the strong two-scale convergence of $F_\e^{-1}$, which is given by Lemma \ref{lem:LemConvData} below, by a standard argument.
The strong convergence of $Q_\e$ was shown in \cite[Lemma 4.8]{wiedemann2021homogenisation} similarly as in \cite{All89Stokes}  (where different boundary conditions were considered).

We note that the compactness results of \cite{wiedemann2021homogenisation} are formulated pointwise with respect to time but can easily be adapted to the time-dependent two-scale convergence.
\end{proof}

As an immediate consequence of the strong convergence results for $\reps$ and $\partial_t \reps$, we obtain strong two-scale compactness for the coefficients in the transformed equations.
Therefore, we define the limit transformation $\psi_0$ and its Jacobians with respect to the microscopic variable by (see also $\eqref{eq:def:psi}$ for the definition of $\psi$)
\begin{align*}
\check{\psi}_0(t,x,y) &\coloneqq \check{\psi}(R_0(t,x), y) =  (R_0(t,x)-\rmax) \tfrac{y-m}{\|y-m\|} \chi(\| y-\m \|),
\\
\psi_0(t,x,y) &\coloneqq  y + \check{\psi}_0(t,x,y) = y + (R_0(t,x)-\rmax) \tfrac{y-m}{\|y-m\|} \chi(\| y-\m \|).
\end{align*} 
Thus, $\psi_0(t,x,\cdot)$ is a diffeomorphism from $Y^*$ onto $Y^*(t,x)$ for a.e.~$(t,x) \in (0,T) \times \Omega$ with inverse $\psi_0^{-1}(t,x, y ) =\psi^{-1}(R_0(t,x), y)$ for $y \in Y^*(t,x)$.
Moreover, we define analogously to the $\e$-scaled setting
\begin{align}
\label{eq:def:F0}
F_0 \coloneqq \nabla_y \psi_0^{\top}, \qquad
J_0 \coloneqq \det(F_0), \qquad
A_0 \coloneqq J_0  F_0^{-1}, \qquad
D_0 \coloneqq J_0 F_0^{-1} D F_0^{-\top}.
\end{align}
In the following lemma, we show the   strong two-scale convergence for the derivatives of $\psieps$. Since the coefficients $F_\e, A_\e, D_\e$ and $J_\e$ in the weak form \eqref{eq:WeakForm:Microproblem} can be expressed as rational functions of the  derivatives of $\psieps$, we can conclude the the strong two-scale convergence for them in Lemma \ref{lem:LemConvData} afterwards.

\begin{lemma}\label{lem:StrongTwoScaleConvPsi}
Let $\reps$ be the subsequence from Proposition \ref{prop:Comapctness} and $\psieps$ be given by \eqref{eq:def:check_psi_eps} and $\check{\psieps}$ be given by \eqref{eq:def:psi_eps}.
Then, it holds 
\begin{align*}
\e^{|\alpha|-1} \partial_{x_{\alpha}} \widecheck{\psieps} &\rightsts{p} \partial_{y_{\alpha}} \widecheck{\psi_0}, 
\\
\e^{|\alpha|} \partial_{x_\alpha} \nabla \psieps  &\rightsts{p}  \partial_{y_\alpha} \nabla_y \psi_0,
\\
\e^{|\alpha|-1} \partial_t\partial_{x_{\alpha}} \psieps =  \e^{|\alpha|-1} \partial_t\partial_{x_{\alpha}} \widecheck{\psieps} &\rightsts{p} \partial_t\partial_{y_{\alpha}} \widecheck{\psi_0} =  \partial_t\partial_{y_{\alpha}} \psi_0 
\end{align*}
for all $\alpha \in  \N_0^n$.
\end{lemma}
\begin{proof}
We follow the argumentation of \cite{WIEDEMANN2023113168}.
We note that $R_\e \rightsts{p} R_0$  and $\partial_t R_\e \rightsts{p} \partial_t R_0$ yield the pointwise convergences ${\teps R_\e(t,x) \to R_0(t,x)}$ and ${\partial_t \teps R_\e(t,x) \to \partial_t R_0(t,x)}$ for a.e.~$(t,x)\in (0,T) \times \Omega$ for a subsequence. By employing the continuity of $r \mapsto \partial_{y_{\alpha}}\widecheck{\psi}(r,y)$ and  $r \mapsto \partial_{y_{\alpha}}\partial_t\widecheck{\psi}(r,y)$ for $\alpha \in \N_0^n$ , we obtain
\begin{align*}
\e^{|\alpha| - 1} \teps \partial_{x_\alpha} \check{\psi}_\e(t,x,y)
&= \e^{- 1} \partial_{y_\alpha}\teps\check{\psi}_\e(t,x,y)
= 
\partial_{y_\alpha} (\check{\psi}(\teps \Reps(t,x), \{[x] + \e y\}) )
\\&
= 
\partial_{y_\alpha} \check{\psi}(\teps \Reps(t,x), y) \to \partial_{y_\alpha} \check{\psi}(R_0(t,x), y)  =
\partial_{y_\alpha} \check{\psi}_0(t,x, y)
\end{align*}
and
\begin{align*}
\e^{|\alpha| - 1} \teps \partial_t \partial_{x_\alpha} \check{\psi}_\e(t,x,y)
&= \e^{- 1} \partial_{y_\alpha}\teps \partial_t\check{\psi}_\e(t,x,y)
= 
\partial_{y_\alpha} \partial_t (\check{\psi}(\teps \Reps(t,x), \{[x] + \e y\}))
\\
&= 
\partial_{y_\alpha} \partial_t(\check{\psi}(\teps \Reps(t,x), y) )
=
\partial_{y_\alpha}  \partial_R\check{\psi}(\teps \Reps(t,x), y) \partial_t \teps \Reps(t,x)
\\
&\to \partial_{y_\alpha} \partial_R \check{\psi}(R_0(t,x), y)  \partial_t R_0 (t,x)=
\partial_{y_\alpha}  \partial_t \check{\psi}_0(t,x, y)
\end{align*}
for a.e.~$(t,x,y)\in (0,T) \times \Omega \times Y$.
Then, we obtain $    \e^{|\alpha| - 1} \teps \partial_{x_\alpha} \check{\psi}_\e \to \partial_{y_\alpha} \check{\psi}_0$ and $ \e^{|\alpha| - 1} \teps \partial_{x_\alpha} \partial_t\check{\psi}_\e \to \partial_{y_\alpha} \partial_t\check{\psi}_0$  by the Lebesgue convergence theorem,
which yields
$\e^{|\alpha| - 1} \partial_{x_\alpha} \check{\psi}_\e 
\rightsts{p}
\partial_{y_\alpha} \check{\psi}_0$ and $\e^{|\alpha| - 1} \partial_{x_\alpha} \partial_t\check{\psi}_\e 
\rightsts{p}
\partial_{y_\alpha} \partial_t \check{\psi}_0$ .
Since this argumentation holds for every arbitrary subsequence, it holds for the whole sequence.
Since $F_\e = \partial_x \widecheck{\psieps}+ \1$ and $F_0 = \partial_y \widecheck{\psi_0}+ \1$, we can transfer the convergence to $F_\e$.
Moreover, it can be easily oserved that $\partial_t \psieps = \partial_t \widecheck{\psieps}$ as well as $\partial_t \psi_0 = \partial_t \widecheck{\psi_0}$.
\end{proof}

\begin{lemma}\label{lem:LemConvData}
Let $\reps$ be the subsequence from Proposition \ref{prop:Comapctness}. Then, it holds for every $p\in [1,\infty)$ that (for $i=1,\ldots,n$) 
\begin{align}
\begin{aligned}\label{eq:StrongTwoSCaleCoeff}
&\Feps \rightsts{p} F_0, \quad
\Jeps \rightsts{p} J_0, \quad
J_\e^{-1} \rightsts{p} J_0^{-1}, \quad
F_\e^{-1} \rightsts{p} F_0^{-1},
\\
&D_{\e} \rightsts{p} D_0, \quad
\Aeps \rightsts{p} A_0,\quad
\Aeps^{-1} \rightsts{p} A_0^{-1}, \quad
\e \partial_{x} \Aeps \rightsts{p} \partial_{y} A_0, \\
&\e \partial_{x} \Aeps^{-1} \rightsts{p} \partial_{y} A_0^{-1}, \quad \partial_t \jeps \rightsts{p} \partial_t J_0, \quad
\e^{-1} \partial_{x} \partial_t \psieps \rightsts{p} \partial_t \partial_{y} \psi_0.
\end{aligned}
\end{align}
The convergence of $\jeps$ is also valid on the surface $\geps$. Further, we obtain 
\begin{align}\label{TSconvergenceEepsWeps}
    \e e_{\e}(\weps) \rightwts{p,2} e_0(w_0).
\end{align}
\end{lemma}
\begin{proof}
The convergences of $F_\e, J_\e$ and $D_\e$ were shown in 
\cite[Corollary~8]{GahnPop2023Mineral} using the structure of $\psi_\e$ and the strong convergence of $R_\e$. In \cite[Lemma~3.3]{AA23}, \cite[Lemma~10]{WIEDEMANN2023113168} and \cite[Lemma~4.6]{wiedemann2021homogenisation}, a more abstract argumentation was used in order to provide additionally the strong convergences for $J_\e^{-1}$, $F_\e^{-1}$, $A_\e$, $A_\e^{-1}$, $\e  \partial_x A_\e$ and $\e \partial_x A_\e^{-1}$ from the first two convergences of Lemma \ref{lem:StrongTwoScaleConvPsi}. The strong two-scale convergence of $J_\e$ follows with the two-scale convergence of $\e^{-1} \partial_t \partial_x \psi_\e$ by the same argumentation.
For sake of completeness, we recap this argumentation briefly. 
First, we note that the convergence of $\Feps = \nabla \psieps^\top$ is given by Lemma \ref{lem:StrongTwoScaleConvPsi}.
Since $J_\e = \det(F_\e)$ is a polynomial with respect to the entries of $F_\e$, the strong two-scale convergence can be transferred on $J_\e$. Having the bound for $J_\e\geq c_J$ from below, we can transfer the strong two-scale convergence of $J_\e$ to $J_\e^{-1}$.

In the next step, we notice that all other terms on the left-hand side in \eqref{eq:StrongTwoSCaleCoeff} can be formulated as polynomials with respect to the entries of $\nabla \psieps$, $\e \partial_x \partial_x \psieps$, $\partial_t\nabla \psieps$, $\e \partial_t \partial_x \partial_x \psieps$ and $J_\e^{-1}$. These convergences are provided by Lemma \ref{lem:StrongTwoScaleConvPsi} and, thus, can be transferred to these polynomials which yields \eqref{eq:StrongTwoSCaleCoeff}.

The strong two-scale convergence of $J_\e$ on the surface $\geps$ follows with the same argumentation and using the unfolding operator for periodic surfaces in the proof of Lemma \ref{lem:StrongTwoScaleConvPsi}.

It remains to show the convergence of $\e e_{\e}(\weps)$ in \eqref{TSconvergenceEepsWeps}.
From the strong two-scale convergence of $\Feps^{-1}$ in $L^p$ for every $p\in [1,\infty)$ and the weak two-scale convergence of $\e \nabla w_\e$ in $L^2$, we obtain together with the characterization via the unfolding operator that  $\e e_{\e}(\weps) \rightwts{p,q} e_0(w_0)$ for all $p \in [1, \infty)$ and $q\in [1,2)$. Then, the essential boundedness of $F_\e$ yields that $\e e_\e (w_\e)$ is bounded in ${L^p((0,T);L^2(\Oe))}$ and, thus, $\e e_{\e}(\weps) \rightwts{p,2} e_0(w_0)$.
\end{proof}

Finally, we have the following weak compactness result for the time derivative of $\jeps \ueps$:
\begin{corollary}\label{Cor:WeakConDtJepsueps}
Let $\widehat{\jeps \ueps}$ denote the zero extension of $\jeps \ueps$. Then, it holds up to a subsequence that
\begin{align}
    \partial_t \left(\widehat{\jeps \ueps}\right) \rightharpoonup \partial_t (\theta u_0) \quad \mbox{ weakly in } L^2((0,T),H^{-1}(\Omega)),
\end{align}
where the porosity $\theta$ is defined in $\eqref{Def:Porosity}$.
\end{corollary}
\begin{proof}
This proof was shown in \cite[p.133--134]{GahnPop2023Mineral} for slightly different boundary conditions and can be easily adapted to our setting. We emphasise that in our notations we have
\begin{align*}
    \theta(t,x) = |Y^{\ast}(t,x)| = \int_{Y^{\ast}(t,x) } \,\mathrm{d}y = \int_{Y^{\ast} } J_0(t,x,y) \,\mathrm{d}y.
\end{align*}
\end{proof}

\subsection{Proof of Theorem \ref{MainTheoremMacroModel}}

Now, we are able to prove our main result Theorem \ref{MainTheoremMacroModel}, \ie~pass to the limit $\e \to 0$ in the weak formulation of the microscopic model.
The  ODE for $R_0$ in $\eqref{MacroODE:R0}$ -- $\eqref{MacroODE_IC}$ follows directly from Proposition \ref{prop:Comapctness}\ref{CompactnessReps}. The Darcy law in $\eqref{DarcyLawPDE}$ -- $\eqref{DarcyLawBC}$ was derived in \cite{wiedemann2021homogenisation} for a given evolution of the microscopic domain (with the same properties of $\psieps$). Hence, the results are also valid in our setting. However, for the sake of completeness, we will sketch here the essential steps and also use a slightly different proof with an orthogonality argument, see Lemma \ref{Lem:TwoPressureDecomp} in the appendix, to obtain the two-pressure Stokes model, which was also used by Allaire in \cite{Hornung1997}. Further, throughout the proof, we work here on the fixed reference element and additionally formulate  the cell problems and the Darcy velocity on the fixed cell. Finally,   the derivation of the  macroscopic transport equation $\eqref{MacroTransportPDE}$ -- $\eqref{MacroTransportIC}$ follows similar lines as for pure diffusion without convection in \cite[Section 5.2]{GahnPop2023Mineral} and, therefore, we will focus  only on the convective term in our proof below.

We start with the derivation of the Darcy law, where we mainly follow the ideas of Allaire in \cite{Hornung1997}, and first obtain  a generalised  two-pressure Stokes model generalised for evolving microdomains. For this, we introduce the space
\begin{align}
\begin{aligned}\label{SpaceV}
    \mathcal{V}:= \bigg\{ \eta \in  &L^2( \Omega , H_{\per,0}^1(Y^{\ast})^n) \, : \, 
    \\ 
    & \nabla_y \cdot \eta = 0 \mbox{ in } \Omega \times Y^{\ast}, \, \nabla_x \cdot \int_{Y^{\ast}} \eta \,\mathrm{d}y = 0 \mbox{ in }  \Omega \bigg\}
\end{aligned}
\end{align}
with $H_{\per,0}^1(Y^{\ast}):= \left\{ v \in  H_{\per}^1(Y^{\ast}) \, : \, v = 0 \mbox{ on } \Gamma\right\}.$
We also define the subspace 
\begin{align*}
    \mathcal{V}^{\infty}:= \mathcal{V} \cap C^{\infty}(\overline{\Omega}, H_{\per,0}^1(Y^{\ast})^n).
\end{align*}
The space $\mathcal{V}^{\infty}$ is a dense subspace of $\mathcal{V}$. For $\eta \in \mathcal{V}^{\infty}$ and $\psi \in C_0^{\infty}(0,T)$, we choose 
\begin{align}\label{TestFunctionEtaeps}
    \eta_{\e}(t,x):= A_{\e}^{-1} \psi(t)\eta\left(x,\fxe\right),
\end{align}
as a test function in $\eqref{eq:weak_w_eps_fixed_domain}$
to obtain after integration with respect to time and using $A_{\e} : \nabla \eta_{\e } = \psi \nabla \cdot \left(\eta \left(\cdot_x , \frac{\cdot_x}{\e} \right)  \right)=\psi \nabla_x \cdot \eta \left(\cdot_x , \frac{\cdot_x}{\e} \right) $
\begin{align}
\begin{aligned}\label{AuxEq:FluidMicroStart}
    \e^2 \int_0^T &\int_{\oe} \jeps e_{\e} (\weps) : e_{\e}(\eta_{\e}) \,\mathrm{d}x \,\mathrm{d}t - \int_0^T \int_{\oe } \qeps \nabla_x \cdot \eta\left(x,\fxe\right) \psi \,\mathrm{d}x \,\mathrm{d}t
    \\
    &= - \e^2 \int_0^T \int_{\oe} \jeps e_{\e}(\partial_t \psieps ) : e_{\e } (\eta_{\e }) \,\mathrm{d}x \,\mathrm{d}t + \int_0^T \int_{\oe } \left[ \jeps \Aeps^{-\top} h_{\e } - \nabla p_\mathrm{b} \right] \cdot \eta\left(x,\fxe\right) \psi \,\mathrm{d}x \,\mathrm{d}t.
\end{aligned}
\end{align}
From the strong two-scale convergence results in Lemma \ref{lem:LemConvData}  and the strong two-scale convergence of $\eta\left(\cdot,\frac{\cdot}{\e}\right)$, which follows from the oscillation lemma \cite[Lemma 1.3]{Allaire_TwoScaleKonvergenz}, we obtain that the unfolded sequence $\teps(\e e_{\e}( \eta_{\e}))$ converges strongly in $L^s((0,T)\times \Omega\times Y^{\ast})$ to $e_0(A_0^{-1} \eta)\psi$ for $s \in [1,2)$ and, therefore, (up to a subsequence) also pointwise almost everywhere in $(0,T)\times \Omega \times Y^{\ast}$. Further, it holds due to the essential bounds for $F_{\e}^{-1}$ and $\Aeps^{-1}$ (and its derivative) almost everywhere in $(0,T)\times \Omega \times Y^{\ast}$
\begin{align*}
    \left|\teps(\e e_{\e}(\eta_{\e}))\right| \le C \teps\left(\eta\left(\cdot,\dfrac{\cdot}{\e}\right)\right).
\end{align*}
Since the right-hand side converges strongly in $L^2((0,T)\times \Omega \times Y^{\ast})$, the dominated convergence theorem of Lebesgue implies $\e e_{\e}(\eta_{\e}) \rightsts{2} e_0(A_0^{-1}\eta) \psi$.
Further, Lemma \ref{lem:EstimatesEpsTrafo} implies that 
\begin{align*}
    \|e_{\e}(\partial_t \psieps)\|_{L^2((0,T)\times \oe)} \le C
\end{align*}
for a constant $C>0$ independent of $\e$. Hence, the first term on the right-hand side in $\eqref{AuxEq:FluidMicroStart}$ is of order $\e$ and vanishes for $\e \to 0$. Using the compactness results from Proposition \ref{prop:Comapctness} and Lemma \ref{lem:LemConvData}, we can pass to the limit $\e \to 0$ in $\eqref{AuxEq:FluidMicroStart}$ and obtain almost everywhere in $(0,T)$
\begin{align*}
  \int_{\Omega} \int_{Y^{\ast}} J_0 e_0(w_0) : e_0(A_0^{-1} \eta) dy dx  -  \int_{\Omega} \int_{Y^{\ast}} \left[J_0 A_0^{-\top}h_0 - \nabla p_\mathrm{b} \right] \cdot \eta dy dx  = 0.
\end{align*}
By density this result is valid for all $\eta \in \mathcal{V}$. The left-hand side defines a functional on $L^2(\Omega, H_{\per,0}^1(Y^{\ast})^n)$. Due to Lemma \ref{Lem:TwoPressureDecomp} in the appendix, there exist $P_0 \in L^{\infty}((0,T),H^1_0(\Omega))$ and $q_1 \in L^{\infty}((0,T)\times \Omega \times Y^{\ast})/\R$ such that for all $\eta \in L^2(\Omega, H_{\per,0}^1(Y^{\ast})^n)$ it holds almost everywhere in $(0,T)$ that
\begin{align}
\begin{aligned}\label{AuxEquationFluidMacro}
    \int_{\Omega}&\int_{Y^{\ast}} J_0 e_0(w_0) : e_0 (A_0^{-1} \eta) \,\mathrm{d}y \,\mathrm{d}x + \int_{\Omega}\int_{Y^{\ast}} \nabla P_0 \cdot \eta \,\mathrm{d}y \,\mathrm{d}x - \int_{\Omega}\int_{Y^{\ast}} q_1 \nabla_y \cdot \eta \,\mathrm{d}y \,\mathrm{d}x
    \\
    &= \int_{\Omega}\int_{Y^{\ast}} J_0 h_0 \cdot A_0^{-1} \eta \,\mathrm{d}y \,\mathrm{d}x  - \int_{\Omega}\int_{Y^{\ast}} \nabla p_\mathrm{b} \cdot \eta \,\mathrm{d}y \,\mathrm{d}x.
\end{aligned}
\end{align}
Repeating the arguments above, where we choose in the test function in $\eqref{TestFunctionEtaeps}$ the function $\eta$ as $\eta \in C_0^{\infty}(\Omega,C_{\per,0}^{\infty}(Y^{\ast})^n)$ with $\nabla_y \cdot \eta = 0$, and subtracting the resulting equation from $\eqref{AuxEquationFluidMacro}$, we obtain by a density argument for all $\eta \in L^2(\Omega,H_{\per,0}^1(Y^{\ast})^n)$ with $\nabla_y \cdot \eta =0$ almost everywhere in $(0,T)$
\begin{align*}
 \int_{\Omega}\int_{Y^{\ast}} \nabla P_0 \cdot \eta \,\mathrm{d}y + \int_{\Omega}\int_{Y^{\ast}} Q_0 \nabla_x \cdot \eta \,\mathrm{d}y \,\mathrm{d}x = 0.
\end{align*}
Using \cite[Lemma 2.10]{Allaire_TwoScaleKonvergenz}, we obtain $P_0 = Q_0 = p_0 - p_\mathrm{b}$ (where $p_0$ is the limit of $\peps$, see also Proposition \ref{prop:Comapctness}) and, in particular, $p_0 \in L^2((0,T),H^1(\Omega))$. Altogether, replacing $\eta $ $\eqref{AuxEquationFluidMacro}$ with $A_0\eta $ (which is still an admissible test function), we get almost everywhere in $(0,T)$
\begin{align}
\begin{aligned}\label{TwoPressureStokesWeakForm}
    \int_{\Omega}&\int_{Y^{\ast}} J_0 e_0(w_0) : e_0 (\eta) \,\mathrm{d}y \,\mathrm{d}x + \int_{\Omega}\int_{Y^{\ast}} A_0^{\top}  \nabla p_0 \cdot \eta \,\mathrm{d}y \,\mathrm{d}x - \int_{\Omega}\int_{Y^{\ast}} q_1 \nabla_y \cdot (A_0 \eta) \,\mathrm{d}y \,\mathrm{d}x
    \\
    &= \int_{\Omega}\int_{Y^{\ast}} J_0 h_0 \cdot \eta \,\mathrm{d}y \,\mathrm{d}x.
\end{aligned}
\end{align}
In other words, the triple $(w_0,p_0,q_1)$ is the unique weak solution of the two-pressure Stokes model
\begin{align}
\begin{aligned}\label{TwoPressureStokesFixedCell}
    -\nabla_y \cdot (A_0 e_0(w_0)) + A_0^{\top} \nabla_x p_0 + A_0^{\top } \nabla_y q_1 &= J_0 h_0 &\mbox{ in }& (0,T)\times \Omega \times Y^{\ast},
    \\
    -\nabla_y \cdot (A_0 w_0) &= 0 &\mbox{ in }& (0,T)\times \Omega \times Y^{\ast},
    \\
    w_0 &= 0 &\mbox{ on }& (0,T)\times \Omega \times \Gamma,
    \\
    -\nabla_x \cdot \int_{Y^{\ast}} A_0 w_0 dy &= 0 &\mbox{ in }& (0,T)\times \Omega,
    \\
    p_0 &= p_\mathrm{b} &\mbox{ on }& (0,T)\times \partial \Omega,
    \\
    w_0,\,q_1 \mbox{ is } Y\mbox{-periodic}.
   \end{aligned}
\end{align}
An elementary calculation shows that for all $\xi \in \R^n$ it holds almost everywhere in $(0,T)\times \Omega \times Y^{\ast}$ that
\begin{align}\label{eq:Trafo:RechteSeiteStokes}
    J_0 \xi = A_0^{\top } \left[\nabla_y \left((\psi_0 - y )\cdot \xi \right) \right] + A_0^{\top }\xi.
\end{align}
Hence, we can write the first equation in $\eqref{TwoPressureStokesFixedCell}$ as
\begin{align*}
    -\nabla_y \cdot (A_0 e_0(w_0)) + A_0^{\top }\nabla_y \left[q_1 - (\psi_0 -y) \cdot h_0 \right] = A_0^{\top } \left[h_0 - \nabla_x p_0\right].
\end{align*}
We emphasise that this substitution holds also for the weak formulation, where we note that $(\psi_0 - y) \cdot h_0$, $A_0^\top$ and the test functions are $Y$-periodic and that the test functions are zero on $\Gamma$.
This gives rise to introduce the following cell problems on the fixed reference element $Y^{\ast}$:  We denote for $i=1,\ldots,n$ by $(w_i,\pi_i) \in L^{\infty}((0,T) \times \Omega,H^1_{\per,0}(Y^*)^n) \times L^{\infty}((0,T) \times \Omega; L^2(Y^*)/\R)$ the unique weak solutions of the cell problems
\begin{align}
\begin{aligned}\label{CellProblemsFluidFixed2}
     -\nabla_y \cdot (A_0 e_0(w_i) ) + A_0^{\top} \nabla_y \pi_i &= A_0^{\top}e_i &\mbox{ in }& (0,T)\times \Omega \times Y^{\ast},
     \\
     \nabla_y \cdot (A_0 w_i) &= 0&\mbox{ in }& (0,T)\times \Omega \times Y^{\ast},
     \\
     w_i, \, \pi_i \mbox{ are } Y\mbox{-periodic}.
\end{aligned}
\end{align}
Since the two-pressure Stokes problem $\eqref{TwoPressureStokesFixedCell}$ admits a unique weak solution, we obtain from the linearity of the problem the following result:
\begin{corollary}
It holds that
\begin{align*}
    w_0 &= \sum_{i=1}^n (h_0 - \nabla_x p_0)_i w_i,
    \\
    q_1 &= (\psi_0 - y)\cdot h_0 + \sum_{i=1}^n (h_0 - \nabla_x p_0)_i \pi_i
\end{align*}
and the Darcy-velocity $v^{\ast}$ is given for almost every $(t,x)\in (0,T)\times \Omega$ by
\begin{align}\label{DefDarcyVelocity}
    v^{\ast} (t,x) := \int_{Y^{\ast}} A_0 w_0 \,\mathrm{d}y= K^{\ast} (h_0 - \nabla p_0),
\end{align}
with the permeability tensor $K^{\ast} \in L^{\infty}((0,T)\times \Omega)^{n\times n}$ defined by (for $i,j=1,\ldots,n$)
\begin{align}\label{Def:PermeabilityTensorFixed}
    K^{\ast}_{ij} &= \int_{Y^{\ast}} J_0 e_0(w_i) : e_0 (w_j) \,\mathrm{d}y = \int_{Y^{\ast}} A_0^{\top} e_i \cdot w_j \,\mathrm{d}y.
\end{align}
In particular, it holds that $\nabla \cdot v^{\ast} = -\partial_t \theta$ with the porosity $\theta$ defined in $\eqref{Def:Porosity}$ and, therefore, $p_0$ fulfills the Darcy law $\eqref{DarcyLawPDE}$ -- $ \eqref{DarcyLawBC}$. 
\end{corollary}

Next, we consider the transport equation and derive $\eqref{MacroEqu:VarTransport}$.
As a test function 
in the variational equation $\eqref{eq:weakForm:u_eps}$ we choose $\phi(t,x) = \phi_0(t,x) + \e \phi_1\left(t,x,\fxe\right)$ with $\phi_0 \in C_0^{\infty}(\Omega)$ and $\phi_1 \in C_0^{\infty}((0,T),C_0^{\infty}(\Omega, C_{\per}^{\infty}(Y)))$ to obtain after integration with respect to time:
\begin{align}
\begin{aligned}\label{AuxEquTransportLimitEps}
   \int_0^T &\left\langle \partial_t (\jeps \ueps) , \phi_0 + \e \phi_1 \left(t, \cdot_x ,\dfrac{\cdot_x}{\e}\right)\right\rangle_{H_{\partial \Omega}^1(\oe)', H^1_{\partial \Omega}(\oe)} \,\mathrm{d}t 
   \\
   &+ \int_0^T\int_{\oe} \left[D_{\e} \nabla \ueps - \ueps A_{\e} \weps \right] \cdot \left[ \nabla \phi_0 + \e \nabla_x \phi_1 \left(t,x,\fxe\right) + \nabla_y \phi_1\left(t,x,\fxe\right)\right] \,\mathrm{d}x \,\mathrm{d}t
   \\
   =& \int_0^T \int_{\oe} \jeps \feps \left[ \phi_0 + \e \phi_1\left(t,x,\fxe\right) \right] \,\mathrm{d}x \,\mathrm{d}t - \rho \int_0^T \int_{\geps} \jeps g\left(\ueps,\e^{-1} \reps\right) \left[ \phi_0 + \e \phi_1\left(t,x,\fxe\right) \right] \,\mathrm{d}\sigma \,\mathrm{d}t.
\end{aligned}
\end{align}
Using the convergence results from Proposition \ref{prop:Comapctness}, Lemma \ref{lem:LemConvData} and Corollary \ref{Cor:WeakConDtJepsueps}, we can pass to the limit $\e \to 0$. For more details on the treatment of the non-convective terms, we refer to \cite[Section 5.2]{GahnPop2023Mineral} and \cite[Theorem 17]{WIEDEMANN2023113168}, in particular for the boundary term.
    
Using integration by parts with respect to the microscopic variable $y$ and using $\nabla_y \cdot (A_0w_0 ) = 0$ (see Proposition \ref{prop:Comapctness}\ref{CompactnessWepsPeps}), we have 
\begin{align*}
    \int_0^T \int_{\oe } \ueps A_{\e}& \weps \cdot \left[ \nabla \phi_0 + \e \nabla_x \phi_1 \left(t,x,\fxe\right) + \nabla_y \phi_1\left(t,x,\fxe\right)\right] \,\mathrm{d}x \,\mathrm{d}t
    \\
    &\overset{\e\to 0}{\longrightarrow}  \int_0^T \int_{\Omega }\int_{Y^{\ast}} u_0(t,x) A_0(t,x,y)w_0(t,x,y) \cdot [\nabla \phi_0(t,x) + \nabla_y \phi_1(t,x,y)]\,\mathrm{d}y \,\mathrm{d}x \,\mathrm{d}t
    \\
    &= \int_0^T \int_{\Omega} u_0(t,x) v^{\ast}(t,x) \nabla \phi_0(t,x) \,\mathrm{d}x \,\mathrm{d}t  - \int_0^T \int_{Y^{\ast} } u_0(t,x) \underbrace{\nabla_y \cdot (A_0w_0)}_{=0} \phi_1(t,x,y) \,\mathrm{d}y \,\mathrm{d}x \,\mathrm{d}t
    \\
    &= \int_0^T \int_{\Omega} u_0 v^{\ast} \nabla \phi_0 \,\mathrm{d}x \,\mathrm{d}t.
\end{align*}
Hence, for $\e \to 0$, we obtain from $\eqref{AuxEquTransportLimitEps}$ the limit equation (neglecting the arguments of the functions, which should be clear from the context, for a simpler notation) 
\begin{align}
\begin{aligned}
\label{AuxEquationTransportMacro}
   \int_0^T \langle \partial_t & (\theta u_0) , \phi_0 \rangle_{H^{-1}(\Omega),H_0^1(\Omega)}  - \int_0^T \int_{\Omega} u_0 v^{\ast} \phi_0 \,\mathrm{d}x \,\mathrm{d}t
   \\
   &+ \int_0^T \int_{\Omega}\int_{Y^{\ast} } D_0\left[\nabla u_0 + \nabla_y u_1 \right] \cdot [\nabla \phi_0 + \nabla_y \phi_1]\,\mathrm{d}y \,\mathrm{d}x \,\mathrm{d}t 
   = \int_0^T \int_{\Omega } \theta f \phi_0 - \rho |\Gamma(t,x)| \partial_t R_0 \phi_0 \,\mathrm{d}x \,\mathrm{d}t,
\end{aligned}
\end{align}
with the porosity $\theta$ defined in $\eqref{Def:Porosity}$. In the last term, we can replace $\partial_t \theta = -|\Gamma(t,x)| \partial_t R_0$.
By density, this equation is valid for all $\phi_0 \in L^2((0,T),H_0^1(\Omega))$ and $\phi_1 \in L^2((0,T)\times \Omega , H_{\per}^1(Y^{\ast}))$. We emphasise that we also used the $L^\infty(0,T)\times \Omega)$ bound for $u_0$ here. By standard arguments from homogenisation theory, see for example \cite{Allaire_TwoScaleKonvergenz}, we obtain by choosing $\phi_0=0$ that $u_1$ admits a representation of the form
\begin{align*}
    u_1(t,x,y) = \sum_{i=1}^n \partial_{x_i} u_0(t,x) \chi_i(t,x,y)
\end{align*}
for almost every $(t,x,y) \in (0,T)\times \Omega \times Y^{\ast}$, where $\chi_i\in L^2((0,T)\times \Omega,H_{\per}^1(Y^{\ast})/\R)$ are the unique weak solutions of the cell problems
\begin{align}
\begin{aligned}
   -\nabla_y \cdot (D_0 (\nabla_y \chi_i + e_i )) &= 0 &\mbox{ in }& (0,T)\times \Omega \times Y^{\ast},
   \\
   -D_0(\nabla_y \chi_i +e_i ) \cdot \nu &= 0 &\mbox{ on }& (0,T)\times \Omega \times \Gamma,
   \\
   \chi_i \mbox{ is } Y\mbox{-periodic}, \, \int_{Y^{\ast}} & \chi_i \,\mathrm{d}y =0.
\end{aligned}
\end{align}
We define the homogenised diffusion coefficient $D^{\ast}\in L^{\infty}((0,T)\times \Omega)^{n\times n}$  for $i,j=1,\ldots,n$  and almost every $(t,x) \in (0,T)\times \Omega)$ by
\begin{align}\label{Def:HomDiffCoeff_fixed}
    D^{\ast}_{ij}(t,x):= \int_{Y^{\ast}} D_0(t,x) [\nabla_y \chi_i(t,x,y) + e_i ] \cdot [\nabla_y \chi_j(t,x,y) + e_j] \,\mathrm{d}y.
\end{align}
Note that that the right-hand side depends on the transformation $\psi_0$ and, therefore, on the radius $R_0$.

\subsection{Transformation of the cell problems back to the evolving domain}\label{subsec:BAcktrafoCellProblems}
Instead of the cell problems \eqref{CellProblemsFluidFixed2}, we can consider the cell problems in the evolving domain: find $(\tilde{w}_i, \tilde{\pi}_i) \in L^\infty((0,T) \times \Omega;H^1_{\per,0}(Y^*(t,x))) \times L^\infty((0,T) \times \Omega;L^2(Y^*(t,x))/\R)$ such that 
\begin{align}
\begin{aligned}
\label{eq:CellProblemStokes:MovingCell}
-\nabla_y \cdot (e (\tilde{w}_i) ) + \nabla_y \tilde{\pi}_i &= e_i &&\mbox{ in }& Y^{\ast}(t,x),
     \\
     \nabla_y \cdot \tilde{w}_i &= 0&&\mbox{ in }& Y^{\ast}(t,x),
     \\
    \tilde{w}_i, \, \tilde{\pi}_i \mbox{ are } Y\mbox{-periodic},   \,   \int_{Y^{\ast}(t,x)}& \tilde{\pi}(t,x,y)\,\mathrm{d}y = 0
    \end{aligned}
\end{align}
for almost every $(t,x) \in (0,T)\times \Omega$. (The Bochner spaces of vector-valued functions with range depending on the domain of definition can be defined via restriction, see for example \cite{meier2008note}.)
Using \eqref{eq:Trafo:RechteSeiteStokes}, it can be observed that $\tilde{w}_i(t,x,y) = w_i(t,x,\psi^{-1}_0(t,x,y))$, $\tilde{\pi}_i(t,x,y) = \pi_i(t,x,\psi^{-1}_0(t,x,y)) - (\psi^{-1}_0(t,x,y) -y)$ up to a constant depending on $t$ and $x$, where $(w_i, \pi_i)$ is the solution of \eqref{CellProblemsFluidFixed2}. 
Having this relation, the permeability tensor \eqref{Def:PermeabilityTensorFixed} can be transformed accordingly, which yields
\begin{align}\label{eq:K:moving_cell}
\begin{aligned}
K^*(t,x)_{ij} &=  \int\limits_{Y^*(t,x)}  e(\tilde{w}_i) : e(\tilde{w}_j) \,\mathrm{d}y = \int\limits_{Y^*(t,x)}  e_i \cdot \tilde{w}_j \,\mathrm{d}y,
\end{aligned}
\end{align}
where the second equality can be obtained with \eqref{eq:CellProblemStokes:MovingCell}.
With the periodicity and the divergence property, it can be shown that for $\varphi \in H^1_{\per,0}(Y^*)$ with $\div \varphi = 0$ it holds that
\begin{align*}
   \int\limits_{Y^*(t,x)} (\nabla \tilde{w}_i)^\top : \nabla \varphi \,\mathrm{d}y = 0.
\end{align*}
Thus, in \eqref{eq:CellProblemStokes:MovingCell}, the symmetric gradients $e_0(\tilde{w}_i)$ can be replaced by $\tfrac{1}{2} \nabla \tilde{w}_i$ and $K^*$ can be also expressed by 
\begin{align}
\begin{aligned}
K^*(t,x)_{ij} = \tfrac{1}{2}
\int\limits_{Y^*(t,x)}  \nabla \tilde{w}_i : \nabla \tilde{w}_j \,\mathrm{d}y.
\end{aligned}
\end{align}
By the same argumentation, the symmetric gradient can be replaced in \eqref{TwoPressureStokesFixedCell}, \eqref{CellProblemsFluidFixed2} and \eqref{Def:PermeabilityTensorFixed}.

Following \cite{AA23}, we can give a representation of $D^{\ast}$ independent of the transformation $\psi_0$ included in the diffusion coefficient $D_0$ as an integral over the evolving reference element $Y^{\ast}(t,x)$. More precisely, we have
\begin{align}\label{Def:HomDiffCoeff}
    D^{\ast}_{ij}(t,x) = \int_{Y^{\ast}(t,x)} D [\tilde{\chi}_i(t,x,y) + e_i ]\cdot [\tilde{\chi}_j(t,x,y) + e_j] \,\mathrm{d}y ,
\end{align}
where $\tilde{\chi}_i \in L^\infty((0,T)\times \Omega,H_{\per}^1(Y^{\ast}(t,x))/\R)$  are the solutions of the following cell problems on the evolving reference element $Y^{\ast}(t,x)$ for almost every $(t,x) \in (0,T)\times \Omega$:
\begin{align}
\begin{aligned}
   -\nabla_y \cdot (D (\nabla_y \tilde{\chi}_i + e_i)) &= 0 &\mbox{ in }& Y^{\ast}(t,x),
   \\
   -D(\nabla_y \tilde{\chi}_i + e_i )\cdot \nu &= 0 &\mbox{ on }& \Gamma(t,x),
   \\
   \tilde{\chi}_i \mbox{ is } Y\mbox{-periodic}, \, \int_{Y^{\ast}} & \tilde{\chi}_i(t,x,y) \,\mathrm{d}y = 0.
\end{aligned}
\end{align}
Between $\chi_i$ and $\tilde{\chi}_i$, the following relation holds for almost every $(t,x,y) \in (0,T)\times \Omega \times Y^{\ast}$ (up to a constant depending on $t$ and $x$): 
\begin{align*}
    \tilde{\chi}_i\left(t,x,\psi_0(t,x,y)\right) = \chi_i(t,x,y) + \left(y - \psi_0(t,x,y)\right) \cdot e_i.
\end{align*}
Now, we proceed with the macroscopic equation. Choosing $\phi_1 = 0$ in $\eqref{AuxEquationTransportMacro}$ and using the definition of $D^{\ast}$, we obtain by a standard calculation
\begin{align}
\begin{aligned}\label{AuxEquTransportMacroII}
     \int_0^T \langle \partial_t & (\theta u_0) , \phi_0 \rangle_{H^{-1}(\Omega),H_0^1(\Omega)}  - \int_0^T \int_{\Omega} u_0 v^{\ast} \phi_0 \,\mathrm{d}x \,\mathrm{d}t
   \\
   &+ \int_0^T \int_{\Omega}D^{\ast}\nabla u_0 \cdot \nabla \phi_0 \,\mathrm{d}x \,\mathrm{d}t 
   = \int_0^T \int_{\Omega } \theta f \phi_0 + \rho \partial_t \theta \phi_0 \,\mathrm{d}x \,\mathrm{d}t
\end{aligned}
\end{align}
for all $\phi_0 \in L^2((0,T), H_0^1(\Omega))$,
which is precisely the variational equation for $u_0$ in the weak formulation of the macroscopic model in Definition \ref{Def:Weak_macro}.  It remains to establish the initial condition $u_0(0) = u^\mathrm{in}$, see also Remark \ref{rem:RegularityU0}, where it suffices to show $(\theta u_0)(0) = \theta^\mathrm{in} u^\mathrm{in}$. This follows  by an integration by parts with respect to time in $\eqref{AuxEquTransportLimitEps}$, choosing $\phi_1=0$ and passing to the limit $\e \to 0$ and comparing the resulting equation with $\eqref{AuxEquTransportMacroII}$. For more details, we refer again to \cite[Section 5.2]{GahnPop2023Mineral}. Note that the homogenised equations in \cite{GahnPop2023Mineral} contain additional terms on the right- and left-hand side which cancel each other. In fact, the origins of these terms in the $\e$-scaled model already cancel each other, see also \cite{WIEDEMANN2023113168}, and they do not appear in the homogenisation process and the $\e$-scaled weak form here. Consequently, the homogenised equations of \cite{GahnPop2023Mineral} and \cite{WIEDEMANN2023113168}, who did not consider advection, coincide, whereas \eqref{AuxEquTransportMacroII} contains the additional advective transport term.
This completes the proof of Theorem \ref{MainTheoremMacroModel}.

\section*{Acknowledgement}
The research of I.~S.~Pop was supported by the Research Foundation-Flanders (FWO),
Belgium, through the Odysseus programme (project G0G1316N) and the German Research Foundation (DFG) through the SFB 1313 (project 327154368). 

\begin{appendix}

\section{Auxiliary results}\label{sec:AppendixA}

In this section, we summarise some technical results important for the treatment of our problem. We start with a two-pressure decomposition:
\begin{lemma}\label{Lem:TwoPressureDecomp}
Let the space $\mathcal{V}$ be defined as in $\eqref{SpaceV}$. We denote its annihilator in $L^2(\Omega, H_{\per,0}^1(Y^{\ast})^n)^{\prime} $ by $\mathcal{V}^{\perp}$. Then, it holds that
\begin{align*}
    \mathcal{V}^{\perp} = \big\{\nabla_x p_0 + \nabla_y p_1 \in L^2(\Omega, H_{\per,0}^1(Y^{\ast})^n)^{\prime}  \, :\, p_0 \in  H^1_0(\Omega), \, p_1 \in L^2(\Omega \times Y^{\ast})/\R \big\}. 
\end{align*}
\end{lemma}
\begin{proof}
Similar results can be found in \cite[Lemma 1.5]{Hornung1997} and \cite[Section 14]{chechkin2007homogenization}. For a detailed proof (for thin domains, which can be easily transferred to our situation), we refer to \cite[Lemma 5.3 and Remark 5.5]{fabricius2023homogenization}.
\end{proof}

Next, we show that, in the case of symmetrical inclusions in the perforated reference cell $Y^{\ast}$ (in our case simply balls), we get a simplified structure for the  homogenised diffusion coefficient and permeability tensor. More precisely,
in the special situation when the perforations are balls, we get that the homogenised diffusion coefficient $D^{\ast}$ and the permeability tensor $K^{\ast}$ in Section \ref{SectionMainResults} are scalars, which we formulate in the following Lemma for the permeability tensor (the proof follows similar lines for the homogenised diffusion coefficient). Here, we consider $Y^{\ast} := \left(-\frac12,\frac12\right)^n \setminus B_r(0)$ with $r \in \left(0,\frac12\right)$, and define the permeability tensor $K^{\ast} \in \R^{n\times n}$ for $i,j=1,\ldots n$ via 
\begin{align}
    K_{ij} := \int_{Y^{\ast}} w_j \cdot e_i \,\mathrm{d}y,
\end{align}
where $(w_j,\pi_j)$ solve
\begin{align}
\begin{aligned}\label{CellProblemK_diagonal}
    -\nabla_y \cdot (\nabla_y w_j) + \nabla_y \pi_j &= e_j &\mbox{ in }& Y^{\ast},
    \\
    \nabla_y \cdot w_j &= 0 &\mbox{ in }& Y^{\ast},
    \\
    w_j &= 0 &\mbox{ on }& \Gamma,
    \\
    w_j \mbox{ is } Y\mbox{-periodic}.
 \end{aligned}  
\end{align}

\begin{lemma}\label{Lem:Scalar_hom_coeff}
The permeability tensor $K^{\ast}$ fulfills $K_{ij}= 0$ for $i\neq j$ and $K_{ii}= K_{jj}$ for all $i,j=1,\ldots, n$. In other words, there exists $k>0$  such that $K = kE_n$.  
\end{lemma}
\begin{proof}
Let $R \in \R^{n \times n}$ be an orthonormal matrix. More precisely, we assume $R = R^{\top} = R^{-1}$. Later, we will choose  $R$ as a reflection at a hyperplane through the origin.
We define $\tw_j(y) := R w_j (Ry)$ and $\tilde{\pi}_j(y):= \pi_j(Ry)$. Hence, an elementary calculation shows that $(\tw_j , \tilde{\pi}_j)$  solves $\eqref{CellProblemK_diagonal}$ with $e_j$ replaced by $Re_j$.

Now, for fix $j\in \{1,\ldots,n\}$, we choose $R$ with $Re_j = -e_j $ and $R e_i = e_i$ for $i\neq j$, \ie~$Ry = (y_1,\ldots,y_{j-1} -y_j y_{j+1},\ldots,y_n)$. Then, $(-\tw_j,-\tilde{\pi}_j)$ is also a solution of $\eqref{CellProblemK_diagonal}$ and, therefore, $\tw_j = - w_j$. We get after transformation and using the definition of $\tw_j$ (using $|\det(R) |= 1$ and $R^{\top} = R=R^{-1}$):
\begin{align*}
    K_{ij} = \int_{Y^{\ast}} w_j \cdot e_i \,\mathrm{d}y = \int_{Y^{\ast}} R \tw_j \cdot e_i \,\mathrm{d}y =- \int_{Y^{\ast}} w_j \cdot R e_i \,\mathrm{d}y.
\end{align*}
This implies that $K$ is a diagonal matrix. 
In order to identify the diagonal elements $K_{ii}= K_{jj}$ for $ i \neq j$, we choose $R e_j = e_i$, $R e_i = e_j$ and $Re_l = e_l$  for $l \in \{1, \dots, n \} \setminus \{i,j\}$.
The same argumentation as above yields $\tilde{w}_j$ is a solution of \eqref{CellProblemK_diagonal} with right-hand side $e_i$. Therefore, we obtain $\tilde{w}_j = w_i$, which yields
\begin{align*}
K_{jj} = \int_{Y^{\ast}} w_j \cdot e_j \,\mathrm{d}y = \int_{Y^{\ast}} R \tw_j \cdot e_j \,\mathrm{d}y = \int_{Y^{\ast}} w_i \cdot R e_j dy = \int_{Y^{\ast}} w_i \cdot e_i \,\mathrm{d}y = K_{ii}.
\end{align*}
Hence, we get the desired result with $k:= K_{ii}$ (with $i\in \{1,\ldots,n\}$ arbitrary).
\end{proof}

We have the following well-known trace inequality for heterogeneous domains $\Oe$, which can be obtained by a standard decomposition argument.
\begin{lemma}\label{Lem:TraceInequality}
Let $p \in [1,\infty)$. For every $\theta>0$, there exists a constant $C(\theta)>0$ independent of $\e$ such that 
\begin{align*}
   \e^{\frac{1}{p}} \|\ueps\|_{L^p(\geps)}  \le C(\theta)\|\ueps\|_{L^p(\Oe)} + \theta \e \|\nabla \ueps\|_{L^p(\Oe)}
\end{align*}
holds for all $\ueps \in W^{1,p}(\Oe)$.
\end{lemma}

Finally, we show that the solution operator for a saddle-point problem is locally Lipschitz continuous. We formulate this result for an abstract saddle-point problem.
\begin{lemma}[Lipschitz continuity of saddle-point problems]\label{lem:SaddlePointLpschitzContinuity}
Let $V,Q$ be Banach spaces and
\begin{align}\label{eq:Def:Space_V_SaddlePoint}
&A \coloneqq \{ a \in \Bil(V,V) \mid \norm{a}{\Bil}< \infty \textrm{ and }  a(v,v) \geq  \alpha \norm{v}{V}^2  \textrm{ for an } \alpha >0 \textrm{ and all } v \in V \},
\\\label{eq:Def:Space_Q_SaddlePoint}
&B \coloneqq \{ b \in \Bil(Q,V) \mid \norm{b}{\Bil}< \infty \textrm{ and } \inf\limits_{q \in Q} \sup\limits_{v \in V} b(q,v) \geq \beta \textrm{ for a } \beta > 0\},  
\end{align}
where for $a \in A$ and $b \in B$
\begin{align}\label{eq:Def:Norm_SaddlePoint}
&\norm{a}{\Bil} \coloneqq\sup\limits_{v,w \in V\setminus \{0 \} } \frac{|a(v,w)|}{\norm{v}{V} \norm{w}{V}} ,
\\
&\norm{b}{\Bil} \coloneqq\sup\limits_{v \in V\setminus \{0 \}, q \in Q\setminus \{0 \} } \frac{|b(v,q)|}{\norm{v}{V} \norm{q}{Q}} .
\end{align}
Let $\mathcal{L} : A  \times B \times V' \times Q' \to  V \times Q$ be the solution operator of the following saddle-point problem with $\mathcal{L}(a,b,f,g) \coloneqq (v,p)$, where $(v,p) \in V \times Q$ is the unique solution of 
\begin{align}\label{eq:SaddlePointProb:1}
a(v, \cdot ) + b(q, \cdot) &= f  && \textrm{ in } V',
\\\label{eq:SaddlePointProb:2}
b(\cdot, v) &= g  && \textrm{ in } Q'.
\end{align}
Let $\alpha_0, \beta_0, C>0$ and
\begin{align*}
&A^{\alpha_0}_H \coloneqq \{ a \in A \mid \norm{a}{\Bil} \leq H \textrm{ and } a(v,v) \geq  \alpha_0 \norm{v}{V}^2  \textrm{ for all } v \in V \},
\\
&B^{\beta_0} \coloneqq \{b \in B \mid  \inf\limits_{q \in Q} \sup\limits_{v \in V} b(q,v) \geq \beta_0 \},
\\
&V'_M \coloneqq \{v' \in V' \mid \norm{v'}{V'} \leq M\},
\\
&Q'_M \coloneqq \{q' \in Q' \mid \norm{q'}{Q'} \leq M\}.
\end{align*}
Then, 
$\mathcal{L}$ is Lipschitz continuous on $A^{\alpha_0}_H \times B^{\beta_0} \times V'_M \times Q'_M$ and the Lipschitz constant depends only on $\alpha_0, \beta_0, C$, i.e.~there exists $L(\alpha_0, \beta_0, H)$ such that
\begin{align*}
\norm{\mathcal{L}(a_1,b_1,f_1,g_1)-\mathcal{L}(a_2,b_2,f_2,g_2)}{V \times Q} 
\\
\leq
L(\alpha_0, \beta_0, H) \left( M\left(\norm{a_1 -a_2}{\Bil} + \norm{b_1 -b_2}{\Bil}\right) +
\left(\norm{f_1 -f_2}{V'} +  \norm{g_1 -g_2}{Q'} \right) \right).
\end{align*}
Since the coercivity constant and the inf--sup constant depend continuously on the bilinear form, $L$ is locally Lipschitz continuous on $A  \times B \times V' \times Q'$ in particular.
\end{lemma}
\begin{proof}
For $i \in \{1,2\}$, let $(a_i,b_i,f_i,g_i)  \in A^{\alpha_0}_H \times B^{\beta_0} \times V'_M \times Q'_M$ and let $(v_i,p_i) \in V \times Q$ be the unique solution of the corresponding saddle-point problem \eqref{eq:SaddlePointProb:1}--\eqref{eq:SaddlePointProb:2}, i.e.~
\begin{align}\label{eq:SaddlePointProb:i:1}
a_i(v_i, \cdot ) + b_i(p_i, \cdot) &= f_i  && \textrm{ in } V',
\\\label{eq:SaddlePointProb:i:2}
b_i(\cdot, v_i) &= g_i  && \textrm{ in } Q'.
\end{align}
By applying the coercivity of $a_1$ and the linearity of the bilinear forms and the right-hand sides, we obtain 
\begin{align*}
\alpha_0\norm{\delta v}{V}^2 \leq& a_1(\delta v, \delta v) 
\\
=&
-\delta b(p_2,\delta v) - \delta g(\delta p) + \delta b(\delta p, v_2) + \delta f(\delta v) -\delta a(v_2,\delta v)
\\
\leq& \norm{\delta b}{\Bil} \norm{p_2}{Q} \norm{\delta v}{V}  + \norm{\delta g}{Q'} \norm{\delta p}{P} 
+ \norm{\delta b}{\Bil} \norm{\delta p}{P}\norm{v_2}{V}  
\\
&+ \norm{\delta f}{V'} \norm{\delta v}{V} + \norm{\delta a}{\Bil} \norm{v_2}{V} \norm{\delta v}{V}.
\end{align*}
We note that standard existence theory provides the following bounds for the solution $(v_i,p_i)$ of \eqref{eq:SaddlePointProb:i:1}--\eqref{eq:SaddlePointProb:i:2}
\begin{align}
\norm{v_i}{V} \leq \frac{1}{\alpha_0} \norm{f_i}{V'} + \frac{2\norm{a_i}{\Bil}}{\alpha_0 \beta_0} \norm{g_i}{Q'} \leq CM,
\\
\norm{p_i}{Q} \leq \frac{2 \norm{a_i}{\Bil}}{\alpha_0 \beta_0} \norm{f_i}{V'} + \frac{2 \norm{a_i}{\Bil}^2}{\alpha_0 \beta_0^2} \norm{g_i}{Q'} \leq CM,
\end{align}
where $C$ depends on $H$ but not on $M$.
Then, the boundedness of the bilinear forms and the right-hand sides yields the boundedness of the solutions and we obtain
\begin{align}\notag
\alpha_0\norm{\delta v}{V}^2 
\leq C M \Big(\norm{\delta b}{\Bil} \norm{\delta v}{V}  + \norm{\delta b}{\Bil} \norm{\delta p}{Q}  +\norm{\delta a}{\Bil} \norm{\delta v}{V}  \Big)
\\\label{eq:002}
+ C \Big( \norm{\delta g}{Q'} \norm{\delta p}{Q}   +\norm{\delta f}{V'} \norm{\delta v}{V}  \Big).
\end{align}
Moreover, employing the inf--sup estimate for $b_1$ yields 
\begin{align*}
\beta \norm{\delta p}{Q} \leq \norm{b_1( \cdot, \delta p)}{V'}
=
\norm{-a_1(\delta v,\cdot) - (\delta a)(v_2,\cdot) + \delta f - \delta b(\cdot,p_2)}{V'}
\\\leq
\norm{a_1}{\Bil} \norm{\delta v}{V} + \norm{\delta a}{\Bil}\norm{v_2}{V} + \norm{\delta f}{V'} + \norm{\delta b}{\Bil}\norm{p_2}{Q}.
\end{align*}
Using the boundedness of the data and of the solutions, we obtain
\begin{align}\label{eq:003}
\beta_0 \norm{\delta p}{Q} \leq C
\left(\norm{\delta v}{V} + \norm{\delta f}{V'} \right) +CM \left(  \norm{\delta a}{\Bil}+ \norm{\delta b}{\Bil} \right).
\end{align}
Inserting \eqref{eq:003} in \eqref{eq:002} and applying the Young inequality yields for arbitrary $\lambda >0$
\begin{align*}
\norm{\delta v}{V}^2 
\leq CM \big(\norm{\delta b}{\Bil} \norm{\delta v}{V}  + \norm{\delta a}{\Bil} \norm{\delta v}{V}  \Big) 
\\+ CM 
\norm{\delta b}{\Bil}
\left(\norm{\delta v}{V} + \norm{\delta f}{V'} \right) 
+
CM^2 \norm{\delta b}{\Bil}\left(  \norm{\delta a}{\Bil}+ \norm{\delta b}{\Bil} \right)
\\+
C\norm{\delta g}{Q'} 
\left(\norm{\delta v}{V} + \norm{\delta f}{V'} \right) +
CM \norm{\delta g}{Q'} \left(  \norm{\delta a}{\Bil}+ \norm{\delta b}{\Bil} \right)
\\
+\norm{\delta f}{V'} \norm{\delta v}{V}
\\
\leq 
C_\lambda M^2 (\norm{\delta b}{\Bil}^2 +  \norm{\delta a}{\Bil}^2
)
+
\lambda \norm{\delta v}{V}^2
+
C_\lambda \norm{\delta f}{V'}^2
+
C_\lambda\norm{\delta g}{Q'}^2.
\end{align*}
By choosing $\lambda$ small enough and taking the square, we obtain
\begin{align}\label{eq:004}
    \norm{\delta v}{V}
\leq
C M (\norm{\delta b}{\Bil} +  \norm{\delta a}{\Bil}
)
+
C \norm{\delta f}{V'}
+
C\norm{\delta g}{Q'}.
\end{align}
Inserting \eqref{eq:004} into \eqref{eq:003} yields
\begin{align}\label{eq:005}
\norm{\delta p}{Q} \leq
C M (\norm{\delta b}{\Bil} +  \norm{\delta a}{\Bil}
)
+
C \norm{\delta f}{V'}
+
C\norm{\delta g}{Q'}.
\end{align}
Thus, \eqref{eq:004}--\eqref{eq:005} imply the local Lipschitz continuity of $\mathcal{L}$.
\end{proof}

\section{Two-scale convergence and unfolding}
\label{SectionTwoScaleConvergence}
In this section, we briefly summarise the notion of two-scale convergence and the concept of the unfolding operator. These methods provide the basic techniques to pass to the limit $\e \to 0$ in the microscopic problem.

\subsection{Two-scale convergence}
\label{SubsectionTwoScaleConvergence}

We start with the definition of two-scale convergence, which was first introduced and analyzed in \cite{Nguetseng} and \cite{Allaire_TwoScaleKonvergenz}, see also \cite{LukkassenNguetsengWallTSKonvergenz}.  Here, we formulate the two-scale convergence for time-dependent functions. However, the aforementioned references can be easily extended to such problems since time acts as a parameter. For the initial data, we use the usual two-scale convergence for time-independent functions, which can be defined in the same way.

\begin{definition}\label{DefinitionTSConvergence}
A sequence $\ueps \in  L^q((0,T),L^p( \Omega))$ for $p,q\in [1,\infty)$ is said to converge  in the two-scale sense (in $L^q L^p$) to the limit function $u_0\in L^q((0,T);L^p( \Omega \times Y))$ if the following relation holds
\begin{align*}
\lim_{\e\to 0}\int_0^T \int_{\Omega}u_{\e}(t,x)\phi\left(t,x,\frac{x}{\e}\right)\,\mathrm{d}x\,\mathrm{d}t = \int_0^T\int_{\Omega}\int_Y u_0(t,x,y)\phi(t,x,y)\,\mathrm{d}y\,\mathrm{d}x\,\mathrm{d}t 
\end{align*}
for every $\phi \in L^{q'}((0,T);L^{p'}( \Omega, C_{\per}^0(Y)))$. In this case, we write $\ueps \rightwts{q,p} u_0$. For $q = p$, we use the notation $\ueps \rightwts{p} u_0$.

A  two-scale convergent sequence $u_{\e}$ convergences strongly in the two-scale sense to $u_0$ (in $L^q L^p$) if  
\begin{align*}
\lim_{\e\to 0}\|u_{\e}\|_{L^q((0,T), L^p( \Omega))} =\|u_0\|_{L^q((0,T), L^p( \Omega  \times Y))} ,
\end{align*}
and we write $\ueps \rightsts{q,p} u_0$ in this case. For $q = p$, we write $\ueps \rightsts{p} u_0$.
\end{definition} 

In \cite{AllaireDamlamianHornung_TwoScaleBoundary,Neuss_TwoScaleBoundary}, the method of two-scale convergence was extended to oscillating surfaces:
\begin{definition}\label{DefinitionTSKonvergenzBoundary}
A sequence of functions $\ueps  \in L^q((0,T),L^p(\Gamma_{\e}))$ for $p,q \in [1,\infty)$ is said to converge in the two-scale sense on the surface $\Gamma_{\e}$ (in $L^qL^p$) to a limit $u_0\in L^q((0,T),L^p(\Omega \times \Gamma))$ if 
\begin{align*}
\lim_{\e \to 0} \e \int_0^T\int_{\Gamma_{\e}} u_{\e}(t,x)\phi\left(t,x,\frac{x}{\e}\right)\,\mathrm{d}\sigma \,\mathrm{d}t = \int_0^T\int_{\Omega}\int_{\Gamma} u_0(t,x,y)\phi(t,x,y)\,\mathrm{d}\sigma_y\,\mathrm{d}x\,\mathrm{d}t\end{align*}
holds for every $\phi \in C\left([0,T]\times\overline{\Omega},C_{per}^0(\Gamma)\right)$.
We write $\ueps \rightwts{q,p} u_0$ on $\geps$ in this case. For $q = p$, we use the notation $\ueps \rightwts{p} u_0$ on $\geps$.

We say that a  two-scale convergent sequence $u_{\e}$ converges strongly in the two-scale sense on $\geps$ (in $L^pL^q$) if
\begin{align*}
\lim_{\e\to 0}  \e^{\frac{1}{p}}\|u_{\e}\|_{L^q((0,T),L^p( \Gamma_{\e}))} = \|u_0\|_{L^q((0,T),L^p(\Omega \times \Gamma))}
\end{align*}
additionally holds
and we write $\ueps \rightsts{q,p} u_0$ in this case. For $q = p$, we write $\ueps \rightsts{p} u_0$ on $\geps$.
\end{definition}

We have the following compactness results (see e.g.~\cite{Allaire_TwoScaleKonvergenz,LukkassenNguetsengWallTSKonvergenz,Neuss_TwoScaleBoundary}, with slight modification for time-dependent functions):
\begin{lemma}\label{BasicTwoScaleCompactness}\
For every $p,q \in (1,\infty)$ we have:
\begin{enumerate}
[label = (\roman*)]
\item For every bounded sequence $\ueps \in L^q((0,T),L^p(\Omega))$, there exists $u_0 \in L^q((0,T),L^p( \Omega \times Y))$ such that, up to a subsequence,
\begin{align*}
\ueps &\rightwts{q,p} u_0.
\end{align*}
\item For every bounded sequence $\ueps \in L^q((0,T),W^{1,p}(\Omega))$, there exist $u_0 \in L^q((0,T),L^p( \Omega))$ and $u_1 \in L^q((0,T), L^p( \Omega,W^{1,p}_{\per}(Y)/\R))$ such that, up to a subsequence,
\begin{align*}
\ueps &\rightwts{q,p} u_0,
\\
\nabla \ueps &\rightwts{q,p} \nabla_x u_0 + \nabla_y u_1 .
\end{align*}
\item For every sequence $\ueps \in L^q((0,T),W^{1,p}(\Omega))$ with $\ueps $ and $\e\nabla \ueps$ bounded in $L^q((0,T), W^{1,p}(\Omega))$, there exists $u_0 \in L^q((0,T), L^p(\Omega, W^{1,p}_{\per}(Y)))$ such that, up to a subsequence, it holds that
\begin{align*}
    \ueps &\rightwts{q,p} u_0,
    \\
    \e \nabla \ueps &\rightwts{q,p} \nabla_y u_0.
\end{align*}
\item For every sequence $\ueps \in L^q((0,T),L^p( \geps))$ with
\begin{align*}
\e^{\frac{1}{p}}\Vert \ueps\Vert_{L^q((0,T),L^p(\geps))} \le C,
\end{align*}
there exists $u_0 \in L^q((0,T),L^p(\Omega \times \Gamma))$ such that, up to a subsequence,
\begin{align*}
\ueps \rightwts{q,p} u_0 \quad \mbox{ on } \geps.
\end{align*}
\end{enumerate}

\end{lemma}

\subsection{The unfolding operator}
When dealing with nonlinear problems it is helpful to work with the unfolding method, which allows a characterization of weak and strong two-scale convergence, see Lemma \ref{LemmaAequivalenzTSKonvergenzUnfolding} below. We refer to \cite{CioranescuGrisoDamlamian2018} for a detailed investigation of the unfolding operator and its properties. For a perforated domain (here we also allow the case $Y^{\ast} = Y$, \ie~$\oe = \Omega$), we define the unfolding operator for $p ,q\in [1,\infty]$ by
\begin{align*}
\teps : L^q((0,T) ,L^p( \oe)) \rightarrow L^q((0,T),L^p( \Omega \times Y^{\ast})), \\
\teps(\ueps)(t,x,y)= \ueps\left(t,\e \left[\fxe\right] + \e y\right).
\end{align*}
In the same way, we define the boundary unfolding operator for the oscillating surface $\geps$ via
\begin{align*}
\teps : L^q((0,T), L^p( \geps)) \rightarrow L^q((0,T), L^p( \Omega \times \Gamma)), \\
\teps(\ueps)(t,x,y)= \ueps\left(t,\e \left[\fxe\right] + \e y\right).
\end{align*}
We emphasise that we use the same notation for the unfolding operator on $\oe$ and the boundary unfolding operator on $\geps$. We summarise some basic properties of the unfolding operator, see \cite{CioranescuGrisoDamlamian2018}, generalised in an obvious way to time-dependent functions with different integrability $p$ and $q$:

\begin{lemma}\label{LemmaPropertiesUnfoldingOperator}
Let $q,p \in [1,\infty]$.
\begin{enumerate}[label = (\roman*)]
\item   For $\ueps \in L^q((0,T),L^p(\oe))$, it holds that
\begin{align*}
\| \teps (\ueps) \|_{L^p((0,T), L^p(\Omega \times Y^{\ast}))} = \|\ueps\|_{L^q((0,T),L^p(\oe))}.
\end{align*}
\item For $\ueps \in L^q((0,T),W^{1,p}(\oe))$, it holds that
\begin{align*}
\nabla_y \teps (\ueps )= \e \teps( \nabla_x \ueps).
\end{align*} 
\item  For $\ueps \in L^q((0,T),L^p(\geps))$, it holds that
\begin{align*}
\|\teps (\ueps) \|_{L^q((0,T),L^p(\Omega \times \Gamma))} = \e^{\frac{1}{p}} \|\ueps\|_{L^q((0,T),L^p(\geps))}.
\end{align*}
\end{enumerate}
\end{lemma}

The following lemma gives a relation between the unfolding operator and two-scale convergence. Its proof is quite standard and we refer the reader to \cite{BourgeatLuckhausMikelic} and \cite{CioranescuGrisoDamlamian2018} for more details (see also \cite[Proposition 2.5]{visintin2006towards}).
\begin{lemma}\label{LemmaAequivalenzTSKonvergenzUnfolding}
Let $p,q \in (1,\infty)$.
\begin{enumerate}
[label = (\alph*)]
\item For   a sequence $\ueps \in L^q((0,T),L^p( \Omega))$, the following statements are equivalent:
\begin{enumerate}[label = (\roman*)]
\item $\ueps \rightwts{q,p} u_0$ ($\ueps \rightsts{q,p} u_0$)
\item $\teps (\ueps )\rightharpoonup u_0$  ($\teps(\ueps) \rightarrow u_0 $) in $L^q((0,T), L^p( \Omega \times Y))$.
\end{enumerate}
\item  For a sequence $\ueps \in L^q((0,T),L^p(\geps))$,  the following statements are equivalent:
\begin{enumerate}[label = (\roman*)]
\item $\ueps \rightwts{q,p} u_0$ ($\ueps \rightsts{q,p} u_0$) on $\geps$,
\item $\teps (\ueps) \rightharpoonup u_0$ ($\teps(\ueps) \rightarrow u_0$) in $L^q((0,T),L^p(\Omega \times \Gamma))$.
\end{enumerate}
\end{enumerate}
\end{lemma}

\end{appendix}

\bibliographystyle{abbrv}
\bibliography{arxiv}

\end{document}